\documentclass[moor]{informs4post}
\usepackage{eqndefns-left} 
\RequirePackage{tgtermes}
\RequirePackage{newtxtext}
\RequirePackage{newtxmath}
\RequirePackage{bm}
\RequirePackage{endnotes}

\SingleSpacedXI

\usepackage[sort&compress]{natbib}
 \bibpunct[, ]{[}{]}{,}{n}{}{,}%
 \def\bibfont{\small}%
\EquationsNumberedBySection 

\TheoremsNumberedBySection  
\ECRepeatTheorems  %

\MANUSCRIPTNO{MOOR-0001-2024.00}

\usepackage[english]{babel}
\usepackage[utf8]{inputenc} 
\usepackage[T1]{fontenc}    
\usepackage{url}            
\usepackage{booktabs}       
\usepackage{amsfonts}       
\usepackage{nicefrac}       
\usepackage{microtype}      
\usepackage[margin=1in]{geometry}
\usepackage[ruled,linesnumbered]{algorithm2e}
\usepackage{amssymb}
\usepackage{amsfonts}
\usepackage{amsmath}
\usepackage{epsfig}
\usepackage{subcaption} 

\newtheorem{fact}{Fact}[section]

\usepackage{pdfcomment}
\usepackage{array}
\usepackage{enumitem}
\usepackage{graphicx}
\newcommand{\lin}[1]{\vec{#1}}

\newcommand{\linVp}{\vec{V}_p}
\newcommand{\linVd}{\vec{V}_d}
\newcommand{\linVu}{\vec{V}_{generic}}
\newcommand{\calFu}{\mathcal{F}_{generic}}
\newcommand{\Vu}{V_{generic}}

\newcommand{\opt}{\operatorname{OPT}}
\newcommand{\dist}{\operatorname{Dist}}

\newcommand{\gap}{\operatorname{Gap}}

\newcommand{\ep}{\operatorname{EP}}
\newcommand{\conv}{\operatorname{Conv}}
\newcommand{\eobj}{\operatorname{E_{obj}}}
\newcommand{\epobj}{\operatorname{E^p_{obj}}}

\newcommand{\condL}{\mathcal{L}}

\newcommand{\eps}{\varepsilon}

\newcommand{\R}{\mathbb R}

\newcommand{\simplephi}{\Phi}
\newcommand{\geophi}{\hat{\Phi}}

\newcommand{\calE}{{\cal E}}
\newcommand{\calF}{{\cal F}}

\newcommand{\calL}{{\cal L}}

\newcommand{\calN}{{\cal N}}

\newcommand{\calS}{{\cal S}}

\newcommand{\calU}{{\cal U}}

\newcommand{\calW}{{\cal W}}
\newcommand{\calX}{{\cal X}}
\newcommand{\calY}{{\cal Y}}
\newcommand{\calZ}{{\cal Z}}

\newcommand{\x}{{\boldsymbol x}}

\newcommand{\z}{{\boldsymbol z}}

\begin{document}


\RUNAUTHOR{Xiong}

\RUNTITLE{Accessible Complexity Bounds for the Restarted PDHG on LPs with a Unique Optimizer}

\TITLE{Accessible Complexity Bounds for Restarted PDHG on Linear Programs with a Unique Optimizer}

\ARTICLEAUTHORS{%
\AUTHOR{Zikai Xiong}
\AFF{Department of Industrial Engineering and Management Sciences, Northwestern University, 2145 Sheridan Road, Evanston, IL 60208, USA. \href{mailto:zikai.xiong@northwestern.edu}{zikai.xiong@northwestern.edu}.}
} 

\ABSTRACT{The restarted primal-dual hybrid gradient method (rPDHG) has recently emerged as an important tool for solving large-scale linear programs (LPs). For LPs with unique optima, we present an iteration bound of $O\left(\kappa\simplephi\cdot\ln\left(\frac{\kappa\simplephi\|w^\star\|}{\eps}\right)\right)$, where $\eps$ is the target tolerance, $\kappa$ is the standard matrix condition number, $\|w^\star\|$ is the norm of the optimal solution, and $\simplephi$ is a geometric condition number of the LP sublevel sets.  
This iteration bound is ``accessible'' in the sense that computing it is typically no more difficult than computing the optimal solution itself.
Indeed, we present a closed-form and tractably computable expression for $\simplephi$.  This enables an analysis of the ``two-stage performance'' of rPDHG: we show that the first stage identifies the optimal basis in ${O}\left(\kappa\simplephi\cdot\ln(\kappa\simplephi)\right)$ iterations, and the second stage computes an $\eps$-optimal solution in $O\left(\|B^{-1}\|\|A\|\cdot\ln\left(\frac{\xi}{\eps}\right)\right)$ additional iterations, where $A$ is the constraint matrix, $B$ is the optimal basis and $\xi$ is the smallest nonzero in the optimal solution. Furthermore, computational tests are consistent with our iteration bounds. We also show a reciprocal relation between the iteration bound and stability under data perturbation, which is also equivalent to (i) proximity to multiple optima, and (ii) the LP sharpness of the instance. 
}%

\FUNDING{Part of this work was performed at the Massachusetts Institute of Technology and was supported by AFOSR Grant No. FA9550-22-1-0356. Another part of this work was performed at Georgia Institute of Technology and was supported by ONR Award \#N00014-25-1-2088.}




\maketitle

\section{Introduction}\label{sec:Intro}

Linear programming has been a cornerstone of optimization since the 1950s, with far-reaching applications across diverse fields, including economics (see, e.g., \citet{greene2003econometric}), transportation (see, e.g., \citet{charnes1954stepping}),   manufacturing (see, e.g., \citet{bowman1956production,hanssmann1960linear}), computer science (see, e.g., \citet{cormen2022introduction}), and medicine (see, e.g., \citet{wagner2004large}) among many others (see, e.g., \citet{dantzig2002linear}). Linear programming algorithms have also been extensively researched in the past several decades. Almost all linear programming algorithms to date are based on either simplex methods or interior-point methods (IPMs). These classic methods form the foundation of modern solvers due to their reliability and robustness in providing high-quality solutions. However, both of them require repeatedly solving linear systems at each iteration using matrix factorizations, whose cost usually grows superlinearly in the size of the instance (as measured in the problem dimensions or the number of nonzeros in the data), unless the data sparsity pattern is suitable. Consequently, as problem size increases, these methods often become computationally impractical. Furthermore, matrix factorizations cannot efficiently leverage modern computational architectures, such as parallel computing on graphics processing units (GPUs). For these reasons, first-order methods (FOMs) are emerging as attractive solution algorithms because they are ``matrix-free,'' meaning they require no or perhaps only very few matrix factorizations, while their primary computational cost lies just in computing matrix-vector products when computing gradients and related quantities.  
As a result, FOMs can often exploit sparse linear algebra and parallel architectures (e.g., GPUs) more readily, and their per-iteration cost typically scales roughly linearly with the number of nonzeros.

The restarted primal-dual hybrid gradient method (rPDHG) has emerged as a particularly successful FOM for solving Linear Programs (LPs).  It directly addresses the saddle-point formulation of LP (see \citet{applegate2023faster}), automatically detects infeasibility (see \citet{applegate2024infeasibility}), and has natural extensions to conic linear programs in \citet{xiong2024role} and convex quadratic programs in \citet{lu2025practicalqp} and \citet{huang2025restarted}. This algorithm has led to various implementations on CPUs (PDLP by \citet{applegate2021practical}) and GPUs (cuPDLP by \citet{lu2025cupdlp} and cuPDLP-C by \citet{lu2023cupdlp-c}), equipped with several effective heuristics. Notably, the performance of the GPU implementations has surpassed classic algorithms (simplex methods and IPMs) on a significant number of problem instances as shown by \citet{lu2025cupdlp} and \citet{lu2023cupdlp-c}. 
The strong practical performance has also sparked considerable industrial interest from mathematical optimization software companies. To date, rPDHG has been integrated into the state-of-the-art commercial solvers COPT (see \cite{copt_logging_pdlp_gpu}), Xpress (see \cite{fico_xpress_hybrid_gradient_doc}) and Gurobi (see \cite{gurobi_13_faq_pdhg}) as a new base algorithm for LPs, complementing simplex methods and IPMs. It has also been added to Google OR-Tools (see \citet{applegate2021practical}), HiGHS (see \citet{coptgithub}), and NVIDIA cuOpt (see \cite{nvidia_cuopt_user_guide_lp_features_2512}). With rPDHG, many problems previously considered too large-scale for classic algorithms are now solvable; for instance, it is reported by \citet{pdlpnews} that a distributed version of rPDHG has been used to solve practical LP instances with more than $ 9.2\times 10^{10}$ nonzero entries in the constraint matrix, a scale far beyond the capabilities of traditional methods. Another example is a large-scale benchmark instance called \textsc{zib03}. This instance, which took 16.5 hours to solve in 2021 as reported by \citet{koch2022progress}, can now be solved in 15 minutes using rPDHG by \citet{lu2023cupdlp-c}. Furthermore, it has recently unlocked the potential of large-scale linear programming in real applications, including making targeted marketing policies by \citet{lu2025optimizing}, solving large-scale integer programming instances by \citet{de2024power}, and optimizing data center network traffic engineering by \citet{lu2024Scaling}, the latter of which has been deployed in Google's production environment.

Despite the strong performance of rPDHG on many LP instances, certain aspects of its practical behavior remain poorly understood. Indeed, rPDHG sometimes performs poorly, even for some very small LP instances. Additionally, minor data perturbation of some easily solvable instances can lead to instances with substantially increased computational cost. Also, it has been observed that rPDHG often exhibits a ``two-stage performance'' phenomenon in which the second stage exhibits much faster local convergence, but this phenomenon has not been adequately explained or otherwise addressed by suitable theory.

To better understand the underlying behavior of rPDHG, it is important to have theory that is consistent with practical performance. However, many aspects of the existing theory cannot be adequately evaluated for practical relevance due to the difficulty of actually computing the quantities in the theoretical bounds. \citet{applegate2023faster} establish the linear convergence rate of rPDHG using the global Hoffman constant of the matrix $K$ of the KKT system corresponding to the LP instance.  Roughly speaking, the Hoffman constant is equal to the reciprocal of the smallest nonzero singular values of the submatrices of $K$, of which there are exponentially many (see \citet{pena2021new}).  While intuition suggests that the Hoffman constant is itself an overly conservative quantity in the computational complexity, we do not know this empirically on nontrivial LP instances because the Hoffman constant is not computable in reasonable time. \citet{xiong2023computational} provide a tighter computational guarantee for rPDHG using two natural properties of the LP problem: LP sharpness and the ``limiting error ratio.'' Furthermore, for LPs with extremely poor sharpness and the broader family of conic LPs, \citet{xiong2024role} provide computational guarantees for rPDHG based on three geometric measures of the primal-dual (sub)level set geometry.  In addition, \citet{lu2025geometry} study the vanilla primal-dual hybrid gradient method (PDHG) using a trajectory-based analysis approach, and show the two-stage performance of PDHG based on the Hoffman constant of a smaller linear system. However, despite these studies, none provides an iteration bound that is reasonably easy to compute, and so we cannot ascertain the extent to which any of these iteration bounds align with computational practice. To compute the iteration bounds, all existing works require prohibitively expensive operations, such as directly computing Hoffman constants (e.g., \citet{applegate2023faster,lu2025geometry}), solving multiple additional optimization problems (e.g., \citet{xiong2023computational,xiong2024role,lu2025geometry}), or running a first-order method beforehand to obtain the solution trajectory (e.g., \citet{lu2025geometry}).

Because existing iteration bounds are hard to evaluate in practice, it remains unclear how well current theory matches computational practice.
This also makes it harder to analyze rPDHG and improve its empirical performance.
More broadly, without a convenient, practically evaluable iteration bound, it is difficult to understand rPDHG’s behavior on specific LP families, to provide theoretical support for effective heuristics, or to develop further practical enhancements.

This paper aims to make progress on the above issues by posing and trying to answer the following questions:  
\begin{itemize}
    \item  Can we derive an \textit{accessible} iteration bound for rPDHG? By \emph{accessible}, we mean that computing the quantities in the bound should rely on intrinsic properties of the data of the LP instance and its optimal solutions.  The computation is typically not more costly than solving the LP instance itself. (For example, the norm of the optimal solution is accessible, whereas Hoffman constants typically are not.)
    \item If we have an accessible iteration bound, can we use the bound to provide deeper insights into the practical performance of rPDHG? -- particularly regarding the two-stage performance phenomenon and the sensitivity to minor data perturbations?
\end{itemize}

This paper focuses on LP instances with unique optimal solutions (of the primal and dual) and hence unique optimal bases, and proves an accessible iteration bound that is easily computable and indeed has a closed-form expression. We acknowledge this unique optimum assumption is restrictive and is often violated in real-world LP instances due to special problem structures (e.g., network-flow structure can induce degeneracy and multiple optima). 
Our goal is therefore not to develop a universal accessible iteration bound for all practical LP instances, but rather to provide a fully computable analysis in a simplified setting, and provide new insights into the performance of rPDHG. As the first accessible iteration bound for rPDHG, it may inform future work on accessible bounds beyond the unique-optimum setting.

Although the unique optimum assumption is often violated in real applications, it is a standard simplifying assumption and frequently used to simplify analysis and convey insights. See, e.g., large-scale LP (e.g., \citet{liu2022primal} and \citet{xiong2023relation}) and other optimization problems, such as semidefinite programs (e.g., \citet{alizadeh1998primal}) and general convex optimization (e.g., \citet{drusvyatskiy2011generic}). 
The property of unique optima for LP holds almost surely under most models of randomly generated LP instances.  
The unique optimum assumption is also weaker than the nondegeneracy assumption (assuming all basic feasible solutions are nondegenerate). Once the primal and dual optimal basic feasible solutions are nondegenerate, the primal and dual solutions are unique and nondegenerate, which means the LP instance has one unique optimal basis. 
This nondegeneracy assumption is also used by almost all classic optimization textbooks, such as \citet{bertsimas1997introduction}, to simplify analysis and convey insights.

\subsection{Outline and contributions}

In this paper, we consider the following standard form LPs:
\begin{equation}\label{pro:primal LP}
	\begin{aligned}
		\min_{x\in \R^n}  \quad c^\top x   \quad \ \ \text{s.t.}     \  Ax = b \ ,  \ x\geq 0 \ .
	\end{aligned}
\end{equation}
where $A\in \mathbb{R}^{m\times n}$ is the constraint matrix, $b\in \mathbb{R}^m$ is the right-hand side vector, and $c\in \mathbb{R}^n$ is the objective vector. The corresponding dual problem is:
\begin{equation}\label{pro:dual LP}
	\begin{aligned}
		\max_{y\in \R^m, s\in\R^n}  \quad b^\top y  \quad \ \
		\text{s.t.} \   A^\top y + s = c \ ,\ s\geq 0 \ .
	\end{aligned}
\end{equation}
We will assume the optimal basis is unique (formally in Assumption~\ref{assump:unique_optima}), denoted by $B$, and let $x^\star$, $y^\star$ and $s^\star$ denote the optimal solutions. 
In Section \ref{sec:Preliminaries}, we revisit its saddle-point formulation and its symmetric reformulation on the space of $x$ and $s$. We also review rPDHG for solving LPs.

Section \ref{sec:global_linear_convergence} presents the main result of the paper: an accessible iteration bound of rPDHG that has a closed-form expression. The bound takes the form 
$$
O\left(\kappa\simplephi\cdot\ln\left(\kappa\simplephi\frac{\|(x^\star,s^\star)\|_2}{\eps}\right)\right)\ ,$$
where $\eps$ is the target tolerance, $\kappa$ is the standard matrix condition number, and $\simplephi$ is a geometric condition number of the LP sublevel sets that admits a closed-form expression. This new bound is actually proven equivalent to a bound in \citet{xiong2024role} (under the unique optimum assumption) but has a closed-form expression. Furthermore, $\simplephi$ has an even simpler upper bound:  $$
\simplephi    \le \frac{ \|x^\star+s^\star\|_1}{\min_{1\le i \le n}\left\{x_i^\star + s_i^\star\right\}} \cdot \left\|B^{-1} A \right\|\ .$$  
Here and throughout, for matrices  $\|\cdot\|$ denotes the spectral norm.

In Section \ref{sec:local_linear_convergence}, using the established accessible iteration bound, we provide a mathematical analysis of rPDHG's ``two-stage performance.'' Specifically, we show that Stage I achieves finite-time  optimal basis identification in 
$$
O\left(\kappa\simplephi\cdot\ln\left(\kappa\simplephi\right)\right)
$$ 
iterations, and Stage II exhibits a faster local convergence rate and computes an $\eps$-optimal solution in 
$$
O\left(\|B^{-1}\|\|A\|\cdot\ln\left(\tfrac{\min_{1\le i \le n}\left\{x_i^\star + s_i^\star\right\}}{\eps}\right)\right)
$$ 
additional iterations. The iteration bound of Stage II is independent of $\simplephi$ and may thus be significantly lower than that of Stage I. This provides at least a partial explanation for the ``two-stage performance'' in theory.

In Section \ref{sec:LP_sharpness}, using the expression of the new iteration bound, we study the relation between the iteration bound of rPDHG and stability under data perturbations, which is also equivalent to two other types of condition measures: proximity to multiple optima, and the LP sharpness of the instance. Specifically, we show that 
$$
\simplephi = \frac{\|x^\star\|_1+\|s^\star\|_1}{\min\{\zeta_p,\zeta_d\}} \ , $$
 where $\zeta_p$ and $\zeta_d$ denote the stability measures for the primal and dual problems. This relationship yields a new computational guarantee, and also interprets the impact of tiny data perturbations on the convergence rate of rPDHG.

In Section \ref{sec:experiments}, since the new iteration bounds can now be easily computed, we confirm their tightness via computational tests on LP instances. As predicted by the new iteration bounds, experiments show that tiny perturbations may indeed significantly alter  $\simplephi$ and the overall convergence rates. Additionally, the new iteration bounds are also confirmed to match practice, as our experiments show $\kappa\simplephi$ and $\|B^{-1}\|\|A\|$ indeed play important roles in the global linear convergence rates and the two-stage performance.
Since our new iteration bounds are proven equivalent to a bound in \citet{xiong2024role}, our experiments also confirm that the latter matches practical behavior.

\subsection{Other related works}

In addition to the previously discussed papers, several other studies have analyzed the performance of PDHG and its variants. 
\citet{hinder2023worst} and \citet{lu2024pdot} present instance-independent worst-case complexity bounds of rPDHG on totally-unimodular LPs and optimal transport problems. \citet{lu2022infimal} show that the last iterate of the vanilla PDHG without restarts also exhibits a linear convergence rate, dependent on the global Hoffman constant of the KKT system matrix. A recent concurrent work \citet{lu2024restarted} propose a new restart scheme for PDHG by restarting from the Halpern iterate instead of the average iterate. They prove an accelerated refined complexity bound compared to that of the vanilla PDHG proven in \citet{lu2025geometry}. This new bound is still based on the Hoffman constant of the reduced KKT system and employs a trajectory-based analysis approach. 

Furthermore, after the release of this paper, \citet{xiong2025high} shows that rPDHG has a high-probability polynomial iteration bound on LP instances whose data follows certain random distributions, via directly analyzing the accessible iteration bound presented in this paper on the random LP instances.  

There has been increasing interest in developing FOMs for LPs. \citet{xiong2024role} propose to use central-path Hessian-based rescaling to accelerate rPDHG, \citet{li2024pdhg} design a learning-to-optimize method to emulate PDHG for solving LPs, and \citet{lin2025pdcs} propose a PDHG-based solver for addressing general conic LP problems, which include LPs. Beyond PDHG, several other FOMs have been studied recently. \citet{lin2021admm} and \citet{deng2025enhanced} develop ABIP (and ABIP+), an ADMM-based interior-point method that leverages the framework of the homogeneous self-dual interior-point method and employs ADMM to solve the inner log-barrier problems. \citet{o2016conic} and \citet{o2021operator} develop SCS, applying ADMM directly to the homogeneous self-dual formulation for general conic LP problems.  \citet{basu2020eclipse} utilize accelerated gradient descent to solve a smoothed dual form of LP. \citet{wang2023linear} use overparametrized neural networks to solve entropically regularized LPs. Very recently, \citet{hough2024primal} use a Frank-Wolfe method to address the saddle-point problem formulation, and \citet{chen2025hpr} implement a Halpern Peaceman-Rachford method with semi-proximal terms to solve LPs.

\subsection{Notation} 
In this paper, we use $[n]$ as shorthand for $\{1,2,\dots,n\}$. For a matrix $A\in\mathbb{R}^{m\times n}$, $\operatorname{Null}(A):=\{x\in\mathbb{R}^n:Ax = 0\}$ denotes the null space of $A$ and $\operatorname{Im}(A) :=\{Ax:x\in\mathbb{R}^n\}$ denotes the image of $A$. For any $i\in[m]$ and $j\in[n]$, $A_{\cdot,j}$ and $A_{i,\cdot}$ denote the $j$th column and $i$th row of $A$, respectively. For any subset $\Theta$ of $[n]$, $A_{\Theta}$ denotes the submatrix of $A$ formed by the columns indexed by $\Theta$. 
We use $\|A\|_{\alpha,\beta}$ to denote the induced operator norm, i.e., $\|A\|_{\alpha,\beta} := \sup_{x\neq 0} \frac{\|Ax\|_\beta}{\|x\|_\alpha}$. Specifically, $\|A\|_{2,\infty}=\max_{1\le i \le m}\left\|A_{i,\cdot}\right\|_2$ and $\|A\|_{1,2}=\max_{j\in[n]}\left\|A_{\cdot,j}\right\|_2$. 
We let $\|\cdot\|_M$ denote the inner product ``norm'' induced by $M$, namely, $\|z\|_M :=\sqrt{z^\top Mz}$. Unless otherwise specified, for a vector $v$, $\|v\|$ denotes the Euclidean norm, and for a matrix $A$, $\|A\|$ denotes $\|A\|_{2,2}$, the spectral norm of $A$. For any set $\calX\subset \mathbb{R}^n$, $P_\calX: \mathbb{R}^n \to \mathbb{R}^n$ denotes the Euclidean projection onto $\calX$, namely, $P_\calX(x) := \arg\min_{\hat{x}\in \calX} \|x - \hat{x}\|$. For any $x \in \mathbb{R}^n$ and set $\calX\subset \mathbb{R}^n$, the Euclidean distance between $x$ and $\calX$ is denoted by $\dist(x,\calX):= \min_{\hat{x} \in \calX} \|x-\hat{x}\|$ and the $M$-norm distance between $x$ and $\calX$ is denoted by $\dist_M(x,\calX):= \min_{\hat{x} \in \calX} \|x-\hat{x}\|_M$. For any set $\calX\subseteq\mathbb{R}^n$, $\partial\calX$ denotes the boundary of $\calX$.  
For $x\in \mathbb{R}^n$, $x^+\in\mathbb{R}^n$ denotes the componentwise positive part, i.e., $(x^+)_i=\max\{x_i,0\}$. 
For any affine subspace $V$, we use $\lin{V}$ to denote the associated linear subspace corresponding to $V$. For any linear subspace $\lin{S}$ in $\mathbb{R}^n$, we use $\lin{S}^\bot$ to denote the corresponding complementary linear subspace of $\lin{S}$. In this paper, we use $O(\cdot)$ to hide factors of only absolute constants.

\section{Preliminaries and Background}\label{sec:Preliminaries}
Throughout this paper, we consider the primal problem \eqref{pro:primal LP} and its dual problem \eqref{pro:dual LP}. 
For simplicity, we assume the rows of $A$ are linearly independent and the LP is feasible and bounded. We impose the unique-optimum assumption later (Assumption~\ref{assump:unique_optima}).
A primal-dual solution pair $x$ and $(y,s)$ is optimal if and only if they are feasible and the duality gap is zero, i.e.,
\begin{equation}\label{def:gap_x_y}
	\gap(x,y):=c^\top x - b^\top y = 0 \ .
\end{equation}
Furthermore, \eqref{pro:primal LP} and \eqref{pro:dual LP} are  equivalent to the following saddle-point problem:
\begin{equation}\label{pro:saddlepoint_LP}
	\min_{x\in \R^n_+} \max_{y\in \R^m} \quad L(x,y) := c^\top x + b^\top y - (Ax)^\top y \ .
\end{equation}
A pair $\left(x^{\star}, y^{\star}\right)$ is a saddle point of \eqref{pro:saddlepoint_LP} if and only if $x^{\star}$ is primal optimal and $y^{\star}$ is dual optimal; in that case the associated dual slack is $s^{\star}:=c-A^{\top} y^{\star}$. Conversely, any primal-dual optimal pair $\left(x^{\star}, y^{\star}\right)$ is a saddle point of \eqref{pro:saddlepoint_LP}.

\subsection{Symmetric formulation of LP}\label{subsec:symmetric_dual}

Given that \eqref{pro:dual LP} includes the constraint $A^\top y + s = c$, it follows that for any feasible $(y,s)$, $y = (AA^\top)^{-1}A(c- s)$.
Let us define $q:=A^{\top}\left(A A^{\top}\right)^{-1} b$; then the objective function of $y$ is equivalent to an objective function of $s$, i.e., $b^\top y = q^\top (c-s)$, and \eqref{pro:dual LP} is thus equivalent to the following (dual) problem on $s$ :
\begin{equation}\label{pro:dual LP_s}
	\max _{s \in \mathbb{R}^n} \ q^{\top}(c-s) \quad \text { s.t. } s \in c+\operatorname{Im}(A^{\top}), s \geq 0 .
\end{equation}
We denote the duality gap for the pair $(x,s)$ as $\gap(x,s):= c^\top x - q^\top (c - s)$, which is equivalent to $\gap(x,y)$ when $A^\top y +s = c$.

Note that the feasible set of the primal problem \eqref{pro:primal LP} is the intersection of the affine subspace $V_p:=q + \operatorname{Null}(A)$ and the nonnegative orthant. Similarly, the feasible set of \eqref{pro:dual LP_s} is the intersection of $V_d:=c+\operatorname{Im}\left(A^{\top}\right)$ and the nonnegative orthant. We can thus rewrite the primal and dual problems in the following symmetric form:
\begin{equation}\label{pro:symmetric_primal_dual}
	\begin{aligned}
		 &  &  \min_{x\in\mathbb{R}^n} & \  c^\top x                                     & \quad                                   &   &  \max_{s\in\mathbb{R}^n} & \  q^\top (c - s)                                 \\
		 &            & \text{s.t.}                               & \ x \in \calF_p:= V_p\cap\mathbb{R}^n_+
		 & \quad      &                                           & \text{s.t.}                                     & \ s \in \calF_d:= V_d\cap\mathbb{R}^n_+                      
	\end{aligned}
\end{equation}
Let $\linVp$ and $\linVd$ denote the linear subspaces associated with the affine subspaces $V_p$ and $V_d$, respectively. These subspaces are orthogonal complements. 
We then use $\calX^\star$ and $\calS^\star$ to denote the optimal solutions for the primal and the dual problem, respectively.
Notably, any change in $c$ within the space of $\linVp$ does not affect $\calX^\star$, $\calS^\star$, $V_p$, $V_d$, $\calF_p$, or $\calF_d$. Without loss of generality, we may sometimes assume that $c$ is in $\operatorname{Null}(A)$, which can be achieved by replacing $c$ with $P_{\linVp}(c)$ beforehand. This leads to the following symmetric properties for \eqref{pro:primal LP} and \eqref{pro:dual LP_s}.
\begin{fact}\label{fact: symmetric LP formulation}
Suppose that $Ac = 0$. Then $\linVd$ is the orthogonal complement of $\linVp$, i.e., $\linVd = \linVp^\bot$. Furthermore, $q^\top c = 0$, and the objective function of \eqref{pro:dual LP_s} is equal to $-q^\top s$, and $\gap(x,s)$ is equal to $c^\top x + q^\top s$. Additionally, $c \in \linVp$ and $c = \arg\min_{v\in V_d} \|v\|$, and $q \in \linVd$ and $q = \arg\min_{v \in V_p}\|v\|$.
\end{fact}
\noindent
This reformulation of the dual was, to the best of our knowledge, first introduced in \citet{todd1990centered}.  
We will use the notation $\calW^\star$ to denote the primal-dual pairs, i.e.,
\begin{equation}\label{def:calW_star}
	\calW^\star := \calX^\star \times \calS^\star = \left\{(x^\star,s^\star): x^\star \in \calX^\star, s^\star \in \calS^\star\right\} \ .
\end{equation} 
Our focus is on computing $\eps$-optimal solutions, which are essentially solutions sufficiently close to $\calW^\star$, as defined below. 
\begin{definition}[$\eps$-optimal solution]\label{def:eps_optimal}
	A solution $w$ is said to be $\eps$-optimal if the Euclidean distance between $w$ and $\calW^\star$ is less than $\eps$, i.e.,
	$$
	\dist(w,\calW^\star)\le \eps \ .
	$$ 
\end{definition}

\subsection{Restarted Primal-dual hybrid gradient method (rPDHG)}
The vanilla primal-dual hybrid gradient method (abbreviated as PDHG) was introduced by \citet{esser2010general,pock2009algorithm} to solve general convex-concave saddle-point problems, of which \eqref{pro:saddlepoint_LP} is a specific subclass. For LP problems, let $z$ denote the primal-dual pair $(x,y)$, and then iteration of PDHG, denoted by $z^{k+1}=(x^{k+1},y^{k+1})\leftarrow \textsc{OnePDHG}(z^k)$, is defined as follows:
\begin{equation}\label{eq_alg: one PDHG}
	\left\{\begin{array}{l}
		x^{k+1} \leftarrow \left(x^k-\tau\left(c-A^{\top} y^k\right)\right)^+  \\
		y^{k+1} \leftarrow y^k+\sigma\left(b-A\left(2 x^{k+1}-x^{k}\right)\right)
	\end{array}\right.
\end{equation}
where $\tau$ and $\sigma$ are the primal and dual step-sizes, respectively.

Algorithm \ref{alg: PDHG with restarts}  presents the general restart scheme for PDHG. We refer to this algorithm as ``rPDHG,'' short for restarted-PDHG.
\begin{algorithm}[htbp]
	\SetAlgoLined
	{\bf Input:} Initial iterate $z^{0,0}:=(x^{0,0}, y^{0,0})=(0,0)$, $\ell \gets 0$, step-sizes $\tau$ and $\sigma$, and $\beta\in(0,1)$ \;
	\Repeat{\text{Either $z^{\ell,0}$ is a saddle point or $z^{\ell,0}$ satisfies some other convergence condition }}{
		\textbf{initialize the inner loop:} inner loop counter $k\gets 0$ \;
		\Repeat{satisfying the $\beta$-restart condition}{
			\textbf{conduct one step of PDHG: }$z^{\ell,k+1} \gets \textsc{OnePDHG}(z^{\ell,k})$ \;\label{line:onepdhg}
			\textbf{compute the average iterate in the inner loop. }$\bar{z}^{\ell,k+1}\gets\frac{1}{k+1} \sum_{i=1}^{k+1} z^{\ell,i}$
			\label{line:average} \;  \label{line:output-is-average-of-iterates}
			$k\gets k+1$ \;
		}\label{line:restart_condition}
		\textbf{restart the outer loop:} $z^{\ell+1,0}\gets \bar{z}^{\ell,k}$, $\ell\gets \ell+1$ \;
	}
	{\bf Output:} $z^{\ell,0}$ ($ \ = (x^{\ell,0},  y^{\ell,0})$)
	\caption{rPDHG: restarted-PDHG}\label{alg: PDHG with restarts}
\end{algorithm}

Line \ref{line:onepdhg} of Algorithm \ref{alg: PDHG with restarts}  is an iteration of the vanilla PDHG as described in \eqref{eq_alg: one PDHG}. For each iterate $z^{\ell,k} = (x^{\ell,k},y^{\ell,k})$, we define $s^{\ell,k}:= c - A^\top y^{\ell,k}$ and $\bar{s}^{\ell,k}:= c - A^\top \bar{y}^{\ell,k}$. The double superscript indexes the outer iteration counter followed by the inner iteration counter, so that $z^{\ell,k}$ is the $k$-th inner iteration of the outer loop with index $\ell$.
Line \ref{line:restart_condition} of Algorithm \ref{alg: PDHG with restarts}  specifies an easily verifiable restart condition proposed by \citet{applegate2023faster} and also used by \citet{xiong2023computational,xiong2024role} and the practical implementation by \citet{applegate2021practical}.  We will formally specify the $\beta$-restart condition and recall its key properties in Section~\ref{subsec:adaptive_restart_condition}.

The primary computational effort of Algorithm \ref{alg: PDHG with restarts} is the \textsc{OnePDHG} in Line \ref{line:onepdhg}, which involves two matrix-vector products. In contrast to traditional methods such as simplex and interior-point methods, rPDHG does not require any matrix factorizations. It is worth noting that the step-sizes $\tau$ and $\sigma$ need to be sufficiently small to ensure convergence.  In particular, if $M:=	\begin{pmatrix}
		\frac{1}{\tau}I_n & -A^\top             \\
		-A                & \frac{1}{\sigma}I_m
	\end{pmatrix}$ is positive semi-definite, then \citet{chambolle2011first} prove rPDHG's iterates will converge to a saddle point of \eqref{pro:saddlepoint_LP}. The above requirement can be equivalently expressed as:
\begin{equation}\label{eq:general_stepsize}
	\tau > 0, \ \sigma >0, \ \text{ and } \ \tau\sigma \le \frac{1}{ \|A\|^2 }  \ .
\end{equation}
Furthermore, the matrix $M$ turns out to be particularly useful in analyzing the convergence of rPDHG through its induced inner product norm defined as $\| z \|_M := \sqrt{z^\top M z}$. This norm will be extensively employed throughout the remainder of this paper.

\subsection{LPs with unique optima}
This paper focuses particularly on LPs with unique optima, the problems satisfying the following assumption:
\begin{assumption}\label{assump:unique_optima}
	The linear optimization problem \eqref{pro:primal LP} has a unique optimal solution $x^\star$, and the dual problem \eqref{pro:dual LP} has a unique optimal solution $(y^\star,s^\star)$, i.e., $\calX^\star = \{x^\star\}$, $\calY^\star = \{y^\star\}$ and $\calS^\star = \{s^\star\}$.
\end{assumption}
When $\calS^\star$ is a singleton, actually  $\calY^\star$ is a singleton if and only if the rows of the constraint matrix $A$ are linearly independent. This assumption is equivalent to having a unique optimal basis, and is also equivalent to the case that the primal and dual optimal basic feasible solutions are nondegenerate. 
Randomly generated instances are known to be nondegenerate almost surely (see \citet{borgwardt2012simplex}).  
The unique optimum assumption is weaker than the nondegeneracy assumption that all basic feasible solutions are nondegenerate, because it requires nondegeneracy only at the optimal solutions.
The unique optimum assumption and the stronger nondegeneracy assumption are often used in large-scale linear programming (see, e.g., \citet{liu2022primal,xiong2023relation}), semidefinite programs (see, e.g., \citet{alizadeh1998primal}), and general convex optimization (see, e.g., \citet{drusvyatskiy2011generic}). But in practice, due to special structures of real problems, this assumption rarely holds.  

Under Assumption \ref{assump:unique_optima}, the primal-dual pair of optimal solutions, $x^\star$ and $(y^\star,s^\star)$, are optimal basic feasible solutions, corresponding to the optimal basis
$\Theta:=\{i \in [n]:x^\star_i> 0 \}$. Let $\bar{\Theta}$ denote the complement of $\Theta$, i.e., $\bar{\Theta} := [n]\setminus \Theta $. Due to strict complementary slackness, $\bar{\Theta} = \{i \in [n]:s_i^\star > 0\}$. As $x^\star$ is an optimal basic feasible solution, there are exactly $m$ components in $\Theta$ and $n-m$ components in $\bar{\Theta}$. 

Since the algorithm is invariant under permutation of the variables, for simplicity of notation in this paper we assume that the optimal basis is $\{1,2,\dots,m\}$ and use $B$ and $N$ to denote the submatrices $A_{\Theta}$ and $A_{\bar{\Theta}}$, respectively. In other words,
\begin{equation}\label{eq:optimal_basis_[m]}
	\Theta = [m] = \{1,2,\dots,m\}\, , \ \ \bar{\Theta} = [n]\setminus [m]=\{m+1,m+2,\dots,n\} \ \text{ and } \ A =  \begin{pmatrix}
		B & N
	\end{pmatrix}  \ .
\end{equation} 
With the above $\Theta$ and $\bar{\Theta}$, the indices of the nonzero entries of $x^\star$ are exactly $[m]$, and the indices of the nonzero entries of $s^\star$ are exactly $[n]\setminus [m]$.  

Later in the paper we will frequently use the following quantities of the matrix $A$:
\begin{equation}\label{eq  def lamdab min max}
	\lambda_{\max }:=\sigma_{\max }^{+}(A), \  \lambda_{\min }:=\sigma_{\min }^{+}(A), \ \kappa:=\frac{\lambda_{\max }}{\lambda_{\min }} 
\end{equation} 
where $\sigma_{\max}^+(A)$ and $\sigma_{\min}^+(A)$ denote the largest and the smallest nonzero singular values of $A$, respectively. And $\kappa$ is often referred to as the matrix condition number of $A$.

\section{Closed-form Complexity Bound of rPDHG}\label{sec:global_linear_convergence}

This section presents the main result of the paper: an iteration bound of the global linear convergence that has a closed-form expression.  
First of all, we define the following quantity $\simplephi$:
\begin{equation}\label{def_geometric_measure}
	\simplephi :=  \big(\|x^\star\|_1+\|s^\star\|_1\big) \cdot \max\left\{
	\max_{1 \le j \le n-m} \frac{\sqrt{\left\|(B^{-1} N)_{\cdot,j}\right\|^2+1} }{s^\star_{m+j}} \, , \ \max_{1\le i \le m} \frac{\sqrt{\left\|(B^{-1} N)_{i,\cdot}\right\|^2+1} }{x^\star_i}
	\right\}  \ .
\end{equation}
\noindent Notably, it leads to the following iteration bound of rPDHG.
\begin{theorem}\label{thm:closed-form-complexity}
	Suppose  Assumption \ref{assump:unique_optima} holds and $Ac = 0$.
	When running Algorithm \ref{alg: PDHG with restarts} (rPDHG) with $\tau =   \frac{1}{2\kappa}$, $\sigma =  \frac{1}{2\lambda_{\max}\lambda_{\min}}$ and $\beta := 1/e$ to solve the LP, the total number of \textsc{OnePDHG} iterations required to compute an $\eps$-optimal solution is at most
	\begin{equation}\label{eq_thm:closed-form-complexity}
		O\left( \kappa  \simplephi   \cdot \ln\left(\kappa  \simplephi \cdot \frac{ \| w^\star\|}{\eps}\right)  \right) \ .
	\end{equation} 
\end{theorem} 

This new computational guarantee for rPDHG is an accessible iteration bound, as it has a closed-form expression that can be easily computed once the optimal solution has been identified. Examining the definition of $\simplephi$ in \eqref{def_geometric_measure}, $B^{-1}A$ is the simplex tableau at the optimal basis $B$.  Overall, $\simplephi$ is in closed form of the optimal solution/basis, making its computation almost as easy as solving the LP itself. Given the optimal basis, the matrix $B^{-1}N$ can also be easily computed via one matrix factorization followed by one matrix multiplication. Overall, computing $\simplephi$ is almost as easy as solving the LP itself. In addition,  $\kappa$ can be computed by one singular value decomposition of $A$. Consequently, the bound \eqref{eq_thm:closed-form-complexity} in Theorem \ref{thm:closed-form-complexity} is accessible because the main calculation is solving the linear program, computing $B^{-1}N$ and the condition number of $A$. Except for the very small-scale LP instances, it is usually not substantially more expensive than solving the LP itself.

Among $\kappa$ and $\simplephi$, $\kappa$ is a standard definition and easy to compute and analyze. It is solely determined by the matrix $A$ and is independent of the problem's geometry. Conversely, although $\simplephi$ is defined in terms of the matrix $A$ and the optimal solutions, it is equivalent to an intrinsic measure of the geometry detailed later in Section \ref{subsec:level_set_condition_numbers}. Indeed, for a fixed representation of the LP, $\simplephi$ is not affected by any parameters of Algorithm \ref{alg: PDHG with restarts}. In addition, replacing the constraint $Ax=b$ with any preconditioned constraint $DAx = Db$ does not change $\simplephi$ either.

\begin{remark}[Step-size ratios and primal--dual reweighting]\label{rmk:stepsize_reweighting}
It should be noted that changing the primal--dual step-size ratio is equivalent to rescaling the right-hand side and objective vectors (\citet{applegate2023faster}).  For $\omega_1,\omega_2>0$, consider the reweighted LP
\begin{equation}\label{pro:reweighted_primal_LP}
    \min_{\tilde{x}\in\mathbb{R}^n} \quad (\omega_1 c)^\top \tilde{x}
    \quad \text{s.t.} \quad
    A\tilde{x}=\omega_2 b,\quad \tilde{x}\ge 0 .
\end{equation}
Applying PDHG to \eqref{pro:reweighted_primal_LP} with step sizes $(\tilde{\tau},\tilde{\sigma})$ is equivalent, under the change of variables $\tilde{x}^k=\omega_2 x^k,$ $\tilde{y}^k=\omega_1 y^k,$
to applying PDHG to the original LP \eqref{pro:primal LP} with step sizes
$(\tau,\sigma) = \big(\frac{\omega_1}{\omega_2}\tilde{\tau},\frac{\omega_2}{\omega_1}\tilde{\sigma}\big).$ The optimal basis is unchanged by this reweighting, while the optimal primal solution and dual slack vector transform as $\tilde{x}^\star=\omega_2 x^\star$, and $\tilde{s}^\star=\omega_1 s^\star$. 
Thus, if one wants to apply Theorem~\ref{thm:closed-form-complexity} to an rPDHG run with a different primal--dual step-size ratio, the relevant condition measures should be computed on the corresponding reweighted instance. For simplicity of presentation, Theorem \ref{thm:closed-form-complexity} is stated under $\tau/\sigma = \lambda_{\min}^2$. 
\end{remark}

Furthermore, $\simplephi$ has the following simplified upper bound:
\begin{proposition}\label{prop_complexity_of_w_xi}
	The following inequality holds for $\simplephi$:
	$$
	\simplephi    \le \frac{ \|x^\star+s^\star\|_1}{\min_{1\le i \le n}\left\{x_i^\star + s_i^\star\right\}} \cdot \|B^{-1} A \| \ .
	$$
\end{proposition}
\noindent 
Due to strict complementary slackness, all components of $x^\star + s^\star$ are strictly positive, and $\min_{1\le i \le n}\left\{x_i^\star + s_i^\star\right\}$ represents the minimum nonzero entry among $x^\star$ and $s^\star$. This upper bound is the product of two factors: (i) $\frac{\|x^\star+s^\star\|_1}{\min_{1\le i \le n}\left\{x_i^\star + s_i^\star\right\}}$, the ratio between the $\ell_1$-norm and the smallest nonzero of the optimal solution, and (ii) $\|B^{-1} A \|$, the spectral norm of $B^{-1}A$. For readers familiar with simplex methods,  $B^{-1}A$ is the simplex tableau at the optimal basis $B$.  Its proof directly computes the relaxation of $\simplephi$; we defer it to Appendix \ref{sec:proofs_global_linear_convergence}.

It should be noted that $\simplephi$ is also relevant to condition numbers of other methods, beyond its connection to the tableau in simplex methods. Firstly, $\min_{1\le i \le n}\left\{x_i^\star + s_i^\star\right\}$ and the ratio $\frac{\|w^\star\|}{\min_{1\le i \le n}\left\{x_i^\star + s_i^\star\right\}}$ appear in classic complexity analyses of interior-point methods, including the convergence behavior (e.g., \citet{guler1993convergence}), finite convergence to optimal solutions (e.g., \citet{ye1992finite}), and identification of the optimal face (e.g, \citet{mehrotra1993finding}). Additionally, \citet{lu2025geometry} demonstrate that PDHG (without restarts) exhibits faster local linear convergence within a neighborhood whose size relates to $\min_{1\le i \le n}\left\{x_i^\star + s_i^\star\right\}$. Furthermore, $\frac{\|w^\star\|}{\min_{1\le i \le n}\left\{x_i^\star + s_i^\star\right\}}$ also appears in finite termination analysis of interior-point methods by \citet{potra1994quadratically,anstreicher1999probabilistic}, in the form that is multiplied by certain norms of $B^{-1}N$.
Later in Section \ref{sec:local_linear_convergence}, we will show that $\simplephi$ also plays an important role in rPDHG's finite time identification of the optimal basis. Notably, while these condition numbers typically appear inside logarithmic terms in the complexity of interior-point methods, rPDHG's complexity is linear with respect to $\simplephi$. This suggests that $\simplephi$ has more profound implications for the complexity and practical convergence rates of rPDHG compared to interior-point methods. Beyond interior-point methods, the upper bound $\frac{\|x^\star\|_1}{\min_{1\le i \le m}x_{i}^\star}$ also plays a crucial role in the complexity analysis of simplex and policy-iteration methods for discounted Markov decision problems (see \citet{ye2011simplex}).

The rest of this section presents the proof of Theorem \ref{thm:closed-form-complexity}.  Section \ref{subsec:level_set_condition_numbers} recalls the sublevel set condition numbers defined by \citet{xiong2024role} and their roles in rPDHG. Furthermore, Section \ref{subsec:level_set_condition_numbers} shows a key lemma of the equivalence relationship between $\simplephi$ and the sublevel set condition numbers, which helps prove Theorem \ref{thm:closed-form-complexity}. After that, Section \ref{subsec:proof_key_lemma} proves the key lemma of the equivalence relationship.

\subsection{Sublevel-set condition numbers and the proof of Theorem \ref{thm:closed-form-complexity}}\label{subsec:level_set_condition_numbers}

Recall that $\calF_p = V_p\cap \R^n_+$ and $\calF_d = V_d\cap \R^n_+$ denote the feasible sets of the primal and dual problems, respectively.  Let $\calF:= \calF_p \times \calF_d$ represent the primal-dual feasible set of the solution pair $(x,s)$.  The optimal solution can then be characterized as 
$$
\calW^\star := \calF \cap \{w= (x,s)\in\mathbb{R}^{2n}: \gap(w)  = 0\} \ ,
$$
the set of feasible solutions with zero duality gap. For 
$w= (x,s)$, we write $\gap(w):= \gap(x,s)$.
The $\delta$-sublevel set is similarly characterized as the feasible solutions whose duality gap is less than or equal to $\delta$, formally defined as follows:
\begin{definition}[$\delta$-sublevel set $\calW_\delta$]\label{def level set}
	For $\delta \ge  0$, the $\delta$-sublevel set $\calW_\delta$ is defined as:
	\begin{equation}\label{eq delta level set}
		\calW_\delta :=   \calF \cap \{w= (x,s)\in\mathbb{R}^{2n}: \gap(w) \le \delta\} \ .
	\end{equation}
\end{definition}
\noindent
Based on $\calW_\delta$, \citet{xiong2024role} introduce the following two geometric condition numbers: the diameter $D_\delta$ and the conic radius $r_\delta$.
\begin{definition}[Condition numbers of $\calW_\delta$]\label{def sublevelset conditionnumbers}
	For $\delta \ge 0$, the diameter of $\calW_\delta$ is defined as
	\begin{equation}\label{eq diameter}
		D_\delta := \max_{u,v \in \calW_\delta} \|u - v\| \ .
	\end{equation}
	And the conic radius of $\calW_\delta$ is the optimal objective value of the optimization problem
	\begin{equation}\label{eq radius}
		r_\delta:= \left(\max_{w \in \calW_\delta, r \ge 0 } \ r \quad \ \  \operatorname{s.t. }  \ \big\{\hat{w}:\|\hat{w}-w\|\le r\big\}\subseteq \mathbb{R}^{2n}_+ \right) \ ,
	\end{equation}
	which is also equal to $\left(\max_{w \in \calW_\delta } \dist(w,\partial \mathbb{R}^{2n}_+)\right)$ and $\left(\max_{w \in \calW_\delta } \min_{i\in[2n]}w_i\right)$. 
	(Indeed, for any $w\in\mathbb{R}^{2n}_+$, the largest Euclidean ball centered at $w$ contained in $\mathbb{R}^{2n}_+$ has radius $\min_{i\in[2n]} w_i$, which equals $\dist(w,\partial\mathbb{R}^{2n}_+)$.)
\end{definition}

These condition numbers play a crucial role in the iteration bound of rPDHG as follows:

\begin{lemma}\label{lm:global_linear}
	Suppose  Assumption \ref{assump:unique_optima} holds and $Ac = 0$.
	When running Algorithm \ref{alg: PDHG with restarts} (rPDHG) with $\tau =   \frac{1}{2\kappa}$, $\sigma =  \frac{1}{2\lambda_{\max}\lambda_{\min}}$ and $\beta := 1/e$ to solve the LP,
	the total number of  \textsc{OnePDHG} iterations  $T$ required to obtain the first outer iteration $N$ such that  $w^{N,0}=(x^{N,0},c-A^\top y^{N,0})$ is $\eps$-optimal is bounded above by
	\begin{equation}\label{eq_thm_gloabal_convergence_T}
			T \le 198 \kappa  \left(\lim\inf_{\delta \searrow 0} \frac{D_\delta}{r_\delta}\right)   \left[ \ln\left(198 \kappa  \left(\lim\inf_{\delta \searrow 0} \frac{D_\delta}{r_\delta}\right)\right)+ \ln\left(\frac{ \| w^\star\|}{\eps}\right)\right] \ .
	\end{equation}
\end{lemma}

The complete proof of this lemma is given in Appendix \ref{sec:proofs_global_linear_convergence}. For simplicity of notation, the rest of the paper uses $\geophi$ to denote $\left(\lim\inf_{\delta \searrow 0} \frac{D_\delta}{r_\delta}\right)$:  
\begin{equation}\label{eq:limiting_ratio}
	\geophi  := \lim\inf_{\delta \searrow 0} \frac{D_\delta}{r_\delta} \ .
\end{equation}
Lemma \ref{lm:global_linear} implies that the linear convergence rate is mostly determined by $\kappa$, a condition number of the constraint matrix, and $\geophi$, a condition number of the sublevel set. When $\calW^\star$ is a singleton, the sublevel set $\calW_\delta$ is always inside the tangent cone to $\calF$ at $w^\star$. Although looking similar, $\geophi$ is not equivalent to the ``width'' of the tangent cone at $w^\star$ (see a formal definition in \citet{freund1999condition}).  The latter is an inherent property of the tangent cone, but the former is also influenced by the direction of  $(c,q)$.

Actually, we have the following critical lemma, which states the approximate equivalence between $\simplephi$ and $\geophi$:
\begin{lemma}\label
	{lm:compute_p_d_condition_numbers}
	Suppose that Assumption \ref{assump:unique_optima} holds, and let $\simplephi$ and $\geophi$ be defined in \eqref{def_geometric_measure} and \eqref{eq:limiting_ratio}, respectively.
	The geometric condition number $\geophi$ and $\simplephi$ are equivalent up to a constant factor of $2$, i.e.,
	\begin{equation}
		\simplephi \le \geophi \le 2 \simplephi \ .
	\end{equation}
\end{lemma}

Because of Lemma \ref{lm:compute_p_d_condition_numbers},  all previous discussions for $\simplephi$ also apply to $\geophi$. In addition, we can now directly prove Theorem \ref{thm:closed-form-complexity}. 
\proof{Proof of Theorem \ref{thm:closed-form-complexity}.}
	Directly applying Lemma \ref{lm:compute_p_d_condition_numbers} in the iteration bound of Lemma \ref{lm:global_linear} yields the desired iteration bound \eqref{eq_thm:closed-form-complexity}.\Halmos\endproof
Furthermore, since $\simplephi$ is  equivalent to $\geophi$ up to a constant, the new accessible iteration bound of Theorem \ref{thm:closed-form-complexity} is also equivalent to the iteration bound of Lemma \ref{lm:global_linear} (an iteration bound of \citet{xiong2024role}) up to a constant. \citet{xiong2024role} point out that proper central path based Hessian rescaling can improve $\geophi$ to at most $2n$, so Lemma \ref{lm:compute_p_d_condition_numbers} indicates that this rescaling can also improve $\simplephi$ to at most $2n$.

We now prove Lemma \ref{lm:compute_p_d_condition_numbers} in Section \ref{subsec:proof_key_lemma}.

\subsection{Proof of the approximate equivalence between $\simplephi$ and $\geophi$}\label{subsec:proof_key_lemma}

The sublevel set can be equivalently regarded as the ``primal sublevel set'' for an artificial LP problem whose variables contain both the primal and dual variables. Section \ref{subsubsec:define_primal_levelset} extends the definitions of the sublevel set to the primal space only and demonstrates how to compute its condition numbers approximately. Subsequently, Section \ref{subsubsec:compute_p_d_condition_numbers} illustrates how to approximate the sublevel-set condition numbers by treating them as primal sublevel-set condition numbers of an artificial problem, thereby proving Lemma \ref{lm:compute_p_d_condition_numbers}.

\subsubsection{Condition numbers of the primal sublevel set}\label{subsubsec:define_primal_levelset}

Recall that the primal feasible set is $\calF_p := V_p  \cap \R^n_+$,  the intersection of the nonnegative orthant $\R^n_+$ and the affine subspace of the linear equality constraints. We define the objective error of $x$ as $\epobj(x):= c^\top x - f^\star$, where $f^\star$ is the optimal objective $c^\top x^\star$ for an optimal $x^\star$ of \eqref{pro:primal LP}. The optimal primal solution is the feasible solution with zero objective error. The primal $\delta$-sublevel set is then defined as:
\begin{equation}\label{eq delta primal and dual level set}
	\calX_\delta :=   \calF_p \cap \left\{x\in\mathbb{R}^{n}: \epobj(x):=c^\top x - f^\star \le \delta\right\} \ ,
\end{equation}
the sets of feasible primal solutions whose objective error does not exceed $\delta$.
Analogous to Definition \ref{def sublevelset conditionnumbers}, we define the diameter and conic radius of $\calX_\delta$ as 
\begin{equation}\label{def:primal_level_set_condition_numbers}
D_\delta^p := \max_{u,v \in \calX_\delta} \|u - v\| \ \ \ \text{ and }\ \ \ r_\delta^p:=\max_{x\in\calX_\delta}\dist(x,\partial\R^n_+) \ .
\end{equation}
Now we show the representation of $\calX_\delta$ and how to compute $D_\delta^p$ and $r_\delta^p$. 

\vspace{5pt}
\textbf{Convex hull representation of the primal sublevel set $\calX_\delta$.}
Under Assumption \ref{assump:unique_optima}, the optimal primal and dual solutions are unique and nondegenerate, corresponding to a unique optimal basis. Each edge emanating from $\calX^\star = \{x^\star\}$ connects to a basic feasible solution of an adjacent basis. There are exactly $n-m$ entering basic variables. Let the corresponding directions of these edges be given by the vectors $u^1, u^2,\dots,u^{n-m} \in \R^n$. Since  $\Theta = [m]$ and $\bar{\Theta} = [n]\setminus [m]$, these vectors can be computed as follows:
\begin{equation}\label{eq:directions_edges_primal}
	u^j_{[m]} := -B^{-1} N_{\cdot,j},\ \   u^j_{m+j} := 1, \  \text{ and }   \ u^j_k := 0 \text{ for all }k \notin [m] \text{ and } k \neq m+j 
\end{equation}
for each $j\in[n-m]$.
Therefore, the $n-m$ edges are as follows:
\begin{equation}
	\calE^j := \left\{x^\star + \theta \cdot u^j: \theta \ge 0, \ x^\star + \theta \cdot u^j\in \calF_p\right\} \quad \text{for each }j\in[n-m] \ .
\end{equation}
If $u^j \ge 0$, then $\calE^j$ is an extreme ray.
Otherwise, $\calE^j$ connects to an adjacent basic feasible solution of $x^\star$.
Based on these edges, when $\delta$ is sufficiently small so that it is no larger than the extreme points' best nonzero objective error $\bar{\delta}_p$ (which is always strictly positive for LP) defined by
\begin{equation}\label{eqdef:best_suboptimal_gap_primal}
	\bar{\delta}_p:= \left\{
	\begin{array}{ll}
		\min\{\epobj(x): x\in \ep_{\calF_p}\setminus \calX^\star\} & \quad \text{ if $\ep_{\calF_p}\setminus \calX^\star \neq \emptyset$ } \\
		+\infty                                                    & \quad \text{ if $\ep_{\calF_p}\setminus \calX^\star = \emptyset$ }    \\
	\end{array}
	\right.
\end{equation}
then $\calX_{\delta}$ can be represented as the convex hull of these edges. Here we use $\ep_{\calF_p}$ to denote the set of extreme points of $\calF_p$. 
\begin{lemma}\label{lm:conv_hull_calX}
	Suppose that Assumption \ref{assump:unique_optima} holds and $\delta \in (0,\infty)$ lies in $(0,\bar{\delta}_p]$. Then $\calX_\delta$ is represented by the following convex hull formulation:
	\begin{equation}
		\calX_\delta = \conv\left(\{x^\star\} \cup \{ x^j:j\in [n-m]\}\right) 	\end{equation}
	where
	\begin{equation}\label{eq:def_x^i}
		x^j := x^\star + \frac{\delta}{s^\star_{m + j}}\cdot u^j	 \text{ for each }j\in [n-m]  \ .
	\end{equation}
\end{lemma}
\begin{proof}{Proof.}
	Because $\delta \in (0,\bar{\delta}_p]$, the halfspace $\{x: c^\top x < f^\star + \bar{\delta}_p\}$  contains no other basic feasible solutions of $\calF_p$ and intersects no other edges except the $n-m$ edges emanating from $\calX^\star$, namely $\calE^1,\calE^2,\dots,\calE^{n-m}$. Therefore, in addition to $x^\star$, the other extreme points of $\calX_\delta$ are the intersection points of the hyperplane $\{x: c^\top x = f^\star + \delta\}$ and the edges $\calE^1,\calE^2,\dots,\calE^{n-m}$. These intersection points all exist because $\delta \in (0,\bar{\delta}_p]\cap(0,\infty)$. Moreover, they are precisely the $\{ x^j:j\in [n-m]\}$ defined in \eqref{eq:def_x^i}, as the objective errors $\epobj(x^j)$ all equal $\delta$ (which is because $\epobj(x^j)=(x^j)^\top s^\star = x^j_{m + j} s^\star_{m+j} = \frac{\delta}{s^\star_{m+j}}\cdot u^j_{m+j}\cdot s^\star_{m+j} = \delta u^j_{m+j} = \delta$).  Therefore, $\calX_\delta$ is indeed the convex hull of the $n-m+1$ points in  $\{x^\star\}$ and $\{ x^j:j\in [n-m]\}$.\Halmos
\end{proof}

\vspace{5pt}
\textbf{Computing the diameter and conic radius of $\calX_\delta$.}
We now provide an approximation of $D^p_\delta$ for sufficiently small $\delta$.

\begin{lemma}\label{lm:upper_of_D}
	Suppose that Assumption \ref{assump:unique_optima} holds and $\delta \in (0,\bar{\delta}_p]$. Then we have
	\begin{equation}\label{eq:compute_primal_dual_diameter}
		\delta \cdot \max_{1\le j\le n-m} \frac{ \sqrt{\left\|B^{-1}N_{\cdot,j} \right\|^2 + 1  }  }{s^\star_{m+j}}  \le D^p_\delta \le 2\delta \cdot \max_{1\le j\le n-m} \frac{\sqrt{\left\|B^{-1}N_{\cdot,j} \right\|^2 + 1  }   }{s^\star_{m+j}} \ .
	\end{equation}
\end{lemma}
\begin{proof}{Proof.}
    The diameter of a polyhedron is the maximum distance between any two extreme points of the polyhedron. Thus, $D^p_{\delta} \le \max_{i,j\in [n-m]}\|x^i - x^\star\| + \|x^j - x^\star\| \le 2 \cdot \max_{1\le j\le n-m}\|x^j - x^\star\| = 2\delta \cdot \max_{1\le j\le n-m} \frac{\|u^j\| }{s^\star_{m+j}}$, where the last equality uses Lemma \ref{lm:conv_hull_calX}. Conversely, $D^p_{\delta} \ge \max_{1\le j\le n-m}\|x^j - x^\star\|$, which equals $\delta \cdot \max_{1\le j\le n-m} \frac{\|u^j\| }{s^\star_{m+j}}$ (by Lemma \ref{lm:conv_hull_calX}). Finally, note from \eqref{eq:directions_edges_primal} that $\|u^j\| = \sqrt{\left\|B^{-1}N_{\cdot,j} \right\|^2 + 1  }$ so the proof is completed.\Halmos
\end{proof}

Next, we show how to exactly compute $r_\delta$ when $\delta$ is sufficiently small. Specifically, we study $\delta$ small enough so that for $\{x^j: 1\le j \le n-m\}$ defined in \eqref{eq:def_x^i}:
\begin{equation}\label{eq:small_delta_2_0}
	\min_{1\le k \le m} \  x^j_{k} \ge  \ x^j_{m+j} \ \text{for all } j\in[n-m] \ .
\end{equation}
In other words, $\delta$ is small enough such that $x^j_{m+j}$ is one of the smallest nonzeros of $x^j$ for all $j\in[n-m]$. 
Using \eqref{eq:def_x^i} and \eqref{eq:directions_edges_primal}, the inequalities in \eqref{eq:small_delta_2_0} are equivalent to
$
x^\star_k + \tfrac{\delta}{s^\star_{m+j}}\,u^j_k \;\ge\; \tfrac{\delta}{s^\star_{m+j}}$
for all $ j\in[n-m]$ and $k\in[m]$.
If $u^j_k\ge 1$ this inequality holds for every $\delta\ge0$, while if $u^j_k<1$ it is further equivalent to $0\le \delta \le \frac{x^\star_k\,s^\star_{m+j}}{1-u^j_k}$. We therefore introduce  
\begin{equation}\label{eq:def_hat_deltap}
\hat{\delta}_p
:=  \min \left\{
\tfrac{x^\star_k\,s^\star_{m+j}}{1-u^j_k} \; : \; 1\le j\le n-m,\,1\le k\le m, \text{ and } u^j_k<1 \right\}
\end{equation}
where the minimum is interpreted as $+\infty$ if the set is empty. For every $0<\delta\le\hat{\delta}_p$, condition \eqref{eq:small_delta_2_0} holds. Assumption~\ref{assump:unique_optima} implies $x^\star_k>0$ and $s^\star_{m+j}>0$, and thus $\hat{\delta}_p >0$.

\begin{lemma}\label{lm:lower_of_r}
	Suppose that Assumption \ref{assump:unique_optima} holds. For any $\delta >0 $ and $\delta \le  \min\{\bar{\delta}_p, \hat{\delta}_p\}$, we have
	\begin{equation}\label{eq:compute_primal_dual_radius}
		r^p_\delta = \frac{\delta}{\|s^\star\|_1}   \ .
	\end{equation}
\end{lemma}
\begin{proof}{Proof.}
	By definition, $r^p_\delta = \max_{x \in  \calX_\delta} \ \dist(x,\partial \mathbb{R}^{n}_+) = \max_{x \in \calX_\delta} \ \min_{1\le l \le n}  x_l$.
	Since $\delta \le \bar{\delta}_p$, using the convex hull formulation of $\calX_\delta$ presented in Lemma \ref{lm:conv_hull_calX}, $r_\delta^p$ can be equivalently written as:
	\begin{equation}\label{eq:lm:lower_of_r_1}
		r^p_\delta = \max_{\substack{\lambda \in \R^{n-m+1}_+ \\ \sum_{j=1}^{n-m+1}\lambda_j = 1}} \min_{1\le l \le n}  \ x(\lambda)
		_l  \quad  \text{in which} \quad x(\lambda):=\lambda_{n-m+1}\cdot x^\star + \sum_{j=1}^{n-m}\lambda_j\cdot x^j   \ .
	\end{equation}
	Due to the above definition of $x(\lambda)$ and the definition of $x^j$ in \eqref{eq:def_x^i}, for each $m+j$ in $\bar{\Theta} = [n]\setminus [m]$, the component $x(\lambda)_{m+j}$ is given by $\lambda_j \cdot x^j_{m+j}$.

    We now claim that a smallest component of $x(\lambda)$ is of an index in $[n]\setminus [m]$. On the one hand,
	\begin{equation}\label{eq:lm:lower_of_r_2}
		\min_{1\le i \le m} x(\lambda)_{i} \overset{\eqref{eq:lm:lower_of_r_1}}{\ge} \lambda_{n-m+1}\cdot \min_{1\le i \le m} x^\star_{i} + \sum_{j=1}^{n-m}\lambda_j\cdot \min_{1\le i \le m}x^j_{i} \ge  0 + \sum_{j=1}^{n-m}\lambda_j\cdot x^j_{m+j} \ .
	\end{equation}
	where the last inequality uses \eqref{eq:small_delta_2_0} because $\delta < \hat{\delta}_p$.
	On the other hand, because $x(\lambda)_{m+j}=\lambda_j \cdot x^j_{m+j}$,
	\begin{equation}\label{eq:lm:lower_of_r_3}
		\sum_{j=1}^{n-m}\lambda_j\cdot x^j_{m+j} = \sum_{j=1}^{n-m}x(\lambda)_{m+j} \ge \min_{1\le j\le n-m} x(\lambda)_{m+j} \ .
	\end{equation}
	Overall, for this small $\delta$,  \eqref{eq:lm:lower_of_r_2} and \eqref{eq:lm:lower_of_r_3} ensure that a smallest component of $x(\lambda)$ is of an index in $[n]\setminus [m]$. 
	
	Consequently, when computing $r^p_\delta$ we only need to consider the components in $[n]\setminus [m]$.
	\begin{equation}\label{eq:lm:lower_of_r_4}
		r^p_\delta \overset{\eqref{eq:lm:lower_of_r_1}}{=} \max_{\substack{\lambda \in \R^{n-m+1}_+ \\ \sum_{j=1}^{n-m+1}\lambda_j = 1}} \min_{1\le l \le n}  \ x(\lambda)
		_l
		= \max_{\substack{\lambda \in \R^{n-m+1}_+ \\ \sum_{j=1}^{n-m+1}\lambda_j = 1}} \min_{m+1\le l \le n}  \ x(\lambda)_l
		= \max_{\substack{\lambda \in \R^{n-m+1}_+ \\ \lambda_{n-m+1} = 0 \\ \sum_{j=1}^{n-m}\lambda_j = 1}}  \min_{m+1\le l \le n}  \ x(\lambda)
		_l
	\end{equation}
	where the last equality follows from $x^\star_{[n]\setminus [m]} = 0$, which implies that $\lambda_{n-m+1}$ in an optimal $\lambda$ for $r_\delta^p$ must be $0$. The value of $r^p_{\delta}$ in \eqref{eq:lm:lower_of_r_4} is then equal to the optimal objective of
	\begin{equation}\label{eq:lm:lower_of_r_5}
		\left(\begin{aligned}
				\max\limits_{\lambda \in \R^{n-m}} & \ \min\limits_{1\le j\le n-m}\lambda_j\cdot x^j_{m+j} \\
				\text{s.t.}    \ \                     & \ \sum\limits_{1\le j\le n-m}\lambda_j = 1, \ \lambda \ge 0
			\end{aligned}\right) =
		\left(\begin{aligned}
				\max\limits_{\lambda \in \R^{n-m}} & \ \min\limits_{1\le j\le n-m}\lambda_j\cdot \frac{\delta}{s^\star_{m+j}} \\
				\text{s.t.}   \ \                      & \ \sum\limits_{1\le j\le n-m}\lambda_j = 1, \ \lambda \ge 0
			\end{aligned}\right)
	\end{equation}
	where the equality uses $x^j_{m+j} = \frac{\delta}{s^\star_{m+j}}$ according to \eqref{eq:def_x^i} and \eqref{eq:directions_edges_primal}. Finally, the optimal solution $\lambda^\star$ of \eqref{eq:lm:lower_of_r_5} is given by $\lambda^\star_j = \frac{s^\star_{m+j}}{\sum_{k=1}^{n-m} s^\star_{m+k}} = \frac{s^\star_{m+j}}{\|s^\star\|_1} $ for each $j$, and the optimal objective is equal to $\frac{\delta}{\|s^\star\|_1}$. This establishes \eqref{eq:compute_primal_dual_radius} and completes the proof.\Halmos
\end{proof}

It is noteworthy that if we similarly define the dual sublevel set $\calS_\delta$ and then \citet[Theorem 3.2.]{freund2003primal} shows  $\delta \le r_\delta^p \cdot \max_{s\in \calS_\delta}\|s\|_1 \le 2\delta$. A direct application of this result yields $\frac{1}{\|s^\star\|_1} \le \lim_{\delta \to 0} \frac{r_\delta^p}{\delta}  \le  \frac{2}{\|s^\star\|_1}$.  Lemma \ref{lm:lower_of_r} provides a slightly stronger result by precisely computing $r_\delta^p$.

\subsubsection{Proof of Lemma \ref{lm:compute_p_d_condition_numbers}}\label{subsubsec:compute_p_d_condition_numbers}

In this subsection, we prove Lemma \ref{lm:compute_p_d_condition_numbers}. We begin by demonstrating that the sublevel set $\calW_\delta$ can be regarded as a primal sublevel set of an artificial problem. Using the results of Section \ref{subsubsec:define_primal_levelset}, we then show how to approximate $D_\delta$ and compute $r_\delta$, which subsequently allows us to approximate $\geophi$ and prove Lemma \ref{lm:compute_p_d_condition_numbers}.

Problem \eqref{pro:dual LP_s} can be transformed into the subsequent standard-form problem
\begin{equation}\label{pro:standard_dual_LP_s}
	\max _{s \in \mathbb{R}^n} \ q^{\top}(c-s) \quad \text { s.t. } Qs = Qc, \ s \geq 0
\end{equation}
for any $Q \in \R^{(n-m)\times n}$ whose rows are linearly independent and orthogonal to the rows of $A$ so that $\operatorname{Null}(Q) = \operatorname{Im}(A^\top)$. This equivalence holds because $\operatorname{Im}(A^\top) + c$ in \eqref{pro:dual LP_s} is identical to $\{s:Qs = Qc\}$ in \eqref{pro:standard_dual_LP_s}. For problem \eqref{pro:standard_dual_LP_s}, the optimal basis is $\bar{\Theta} = [n]\setminus [m]$, the complement of $\Theta$. Although multiple choices of $Q$ exist,  the matrix $Q_{\bar{\Theta}}^{-1} Q_{\Theta}$ is always equal to $-(B^{-1}N)^\top$.
\begin{lemma}\label{lm:formula_QinvQ}
	Suppose that Assumption \ref{assump:unique_optima} holds. The matrix $Q_{\bar{\Theta}}^{-1} Q_{\Theta}$ is equal to $-(B^{-1}N)^\top$.
\end{lemma}
\begin{proof}{Proof.}
	Given that $\operatorname{Null}(Q) = \operatorname{Im}(A^\top)$, we have $QA^\top = 0$, i.e.,
	\begin{equation*} 
		0 =QA^\top =  \left(\begin{smallmatrix}
			Q_{\Theta} & Q_{\bar{\Theta}}
		\end{smallmatrix}\right)
		\left(\begin{smallmatrix}
			A_{\Theta}^\top \\ A_{\bar{\Theta}}^\top
		\end{smallmatrix}\right) =  Q_{\Theta} A_{\Theta}^\top + Q_{\bar{\Theta}} A_{\bar{\Theta}}^\top =  Q_{\Theta} B^\top + Q_{\bar{\Theta}} N^\top \ .
	\end{equation*}
	Since the optimal bases $B$ and $Q_{\bar{\Theta}}$ are of full rank, we obtain $Q_{\bar{\Theta}}^{-1} Q_{\Theta} = -N^\top (B^\top)^{-1}= -(B^{-1}N)^\top$.\Halmos
\end{proof}

Overall, the primal problem \eqref{pro:primal LP} and the dual problem \eqref{pro:standard_dual_LP_s} can be combined and reformulated as an equivalent standard-form LP problem in the product space of $x$ and $s$:
\begin{equation}\label{pro:standard_primal_dual_xs}
	\min_{w = (x,s)\in \mathbb{R}^{2n}} \  \begin{pmatrix}
		c \\
		q
	\end{pmatrix}^\top w \quad \text { s.t. } \begin{pmatrix}
		A & 0 \\
		0 & Q
	\end{pmatrix} w = \begin{pmatrix}
		b \\
		Qc
	\end{pmatrix}, \ w \geq 0 \ .
\end{equation}
The above \eqref{pro:standard_primal_dual_xs} is also in standard form, and satisfies Assumption \ref{assump:unique_optima}.  Furthermore, the duality gap $\gap(x,s)$ of any $(x,s)$ is the same as the objective error for \eqref{pro:standard_primal_dual_xs} because
\begin{equation}\label{eq:gap_eobj}
	\gap(x,s) = c^\top x - q^\top (c - s) = \left(\begin{smallmatrix}
		c \\
		q
	\end{smallmatrix}\right)^\top w - q^\top c = \left(\begin{smallmatrix}
		c \\
		q
	\end{smallmatrix}\right)^\top (w - w^\star) 
\end{equation}
where the last equality follows from $0 = \gap(w^\star) = c^\top x^\star - q^\top (c - s^\star) = \left(\begin{smallmatrix}
	c \\
	q
\end{smallmatrix}\right)^\top w^\star -q^\top c$. 
The right-hand side of \eqref{eq:gap_eobj} is the objective error $\eobj(w)$ of $w$, defined by $\eobj(w):=\left(\begin{smallmatrix}
	c \\
	q
\end{smallmatrix}\right)^\top (w - w^\star) $.
Consequently, \eqref{eq:gap_eobj} implies that the  $\delta$-sublevel set $\calW_\delta$ is identical to the primal $\delta$-sublevel set of \eqref{pro:standard_primal_dual_xs}. Therefore, utilizing the results in Section \ref{subsubsec:define_primal_levelset}, we can directly approximately compute $D_\delta$ and $r_\delta$, by treating them as the condition numbers of the primal sublevel set of \eqref{pro:standard_primal_dual_xs}.

\begin{lemma}\label{lm:compute_Dr_primal_dual}
	Suppose that Assumption \ref{assump:unique_optima} holds.
	There exists a positive $\bar{\delta}$ such that for any $0 \le \delta \le \bar{\delta}$, it holds that
	\begin{equation}\label{eq:Dr_primal_dual}
		\hat{D}_\delta\le D_\delta \le 2 \hat{D}_\delta \quad \text{ and }\quad r_\delta = \frac{\delta}{\|x^\star\|_1 + \|s^\star\|_1} \ ,
	\end{equation}
	where
	\begin{equation}\label{eq:def_hat_D}
		\hat{D}_\delta := \delta \cdot \max\left\{
		\max_{1\le j\le n-m} \frac{\sqrt{\left\|(B^{-1} N)_{\cdot,j}\right\|^2+1} }{s^\star_{m + j}},  \max_{1\le i \le m} \frac{\sqrt{\left\|(B^{-1} N)_{i,\cdot}\right\|^2+1} }{x^\star_{i}}
		\right\}  \ .
	\end{equation}
\end{lemma} 
\begin{proof}{Proof.}
	Let $H$ denote the constraint matrix of \eqref{pro:standard_primal_dual_xs} for simplicity of notation. We use $\Omega$ to represent the indices of the optimal basis of \eqref{pro:standard_primal_dual_xs}, which is $\Theta \cup (n + \bar{\Theta})$. Here $n+\bar{\Theta}$ denotes $\{n+j:j\in\bar{\Theta}\}$.
	Similarly, the complementary set is $\bar{\Omega} = \bar{\Theta} \cup (n + \Theta)$. We use $\Omega(i)$ to denote the $i$-th smallest index component of $\Omega$ and use $\bar\Omega(j)$ to denote the $j$-th smallest index component of $\bar\Omega$.
	Both $\Omega$ and $\bar{\Omega}$ contain exactly $n$ components, and the optimal basis $H_{\Omega}$ is given by $\big(\begin{smallmatrix}
				A_{\Theta} & 0 \\ 0 & Q_{\bar{\Theta}}
			\end{smallmatrix}\big)$, or equivalently $\big(\begin{smallmatrix}
				B & 0 \\ 0 & Q_{\bar{\Theta}}
			\end{smallmatrix}\big)$.

	The optimal dual slack vector of \eqref{pro:standard_primal_dual_xs} is $\tilde{w}^\star=(s^\star,x^\star)$. Indeed, since $Aq=b$ and $Ax^\star=b$, we have $q-x^\star\in\operatorname{Null}(A)=\operatorname{Im}(Q^\top)$. Hence, there exists $\nu^\star$ such that $Q^\top\nu^\star=q-x^\star$. Together with $A^\top y^\star+s^\star=c$, the vector $(y^\star,\nu^\star)$ is dual feasible for \eqref{pro:standard_primal_dual_xs} with dual slack $\tilde{w}^\star=(s^\star,x^\star)$. Its dual objective value is
	\[
	b^\top y^\star+(Qc)^\top\nu^\star = c^\top x^\star+c^\top(q-x^\star)=q^\top c,
	\]
	which is equal to the primal objective value at $w^\star=(x^\star,s^\star)$ by zero duality gap. Therefore, $\tilde{w}^\star$ is an optimal dual slack vector. Moreover, it is unique: under Assumption~\ref{assump:unique_optima}, $w^\star=(x^\star,s^\star)$ is nondegenerate with basis $\Omega$, and complementary slackness together with the nonsingularity of $H_\Omega$ uniquely determines the optimal dual slack.

	We now prove the first half of \eqref{eq:Dr_primal_dual} using Lemma \ref{lm:upper_of_D}. 
	The term $\sqrt{\left\|B^{-1}N_{\cdot,j} \right\|^2 + 1  }$ in \eqref{eq:compute_primal_dual_diameter} is $\sqrt{\left\|H_{\Omega}^{-1} H_{\cdot,\bar{\Omega}(j)}\right\|^2 + 1}$ for problem \eqref{pro:standard_primal_dual_xs}. And $s^\star_{\bar{\Theta}}$ in \eqref{eq:compute_primal_dual_diameter} is $\tilde{w}_{\bar{\Omega}}^\star$ in problem \eqref{pro:standard_primal_dual_xs}. Therefore, Lemma \ref{lm:upper_of_D} implies 
	\begin{equation}\label{eq:Dr_primal_dual2}\textstyle
		\bar{D}_\delta \le D_\delta \le 2\bar{D}_\delta,\quad \text{ where }\bar{D}_\delta := \delta\cdot \max_{j\in[n]}\frac{\sqrt{\left\|H_{\Omega}^{-1} H_{\cdot,\bar{\Omega}(j)}\right\|^2 + 1}}{\tilde{w}_{\bar{\Omega}(j)}^\star} \ .
	\end{equation}

    To compute the value of $\bar{D}_\delta$, we consider two cases based on the structure of $\bar{\Omega} = \bar{\Theta} \cup (n + \Theta)$. 
	When $\bar{\Omega}(j) \in \bar{\Theta} $, we have $\tilde{w}^\star_{\bar{\Omega}(j)} = s^\star_{\bar{\Omega}(j)} $, and $H_{\Omega}^{-1} H_{\cdot,\bar{\Omega}(j)} = H_{\Omega}^{-1} \left( \begin{smallmatrix}
			A_{\cdot,\bar{\Omega}(j)} \\ 0
		\end{smallmatrix} \right) = \left(\begin{smallmatrix}
			B^{-1} A_{\cdot,\bar{\Omega}(j)} \\ 0
		\end{smallmatrix}\right) $.
	When $\bar{\Omega}(j) \in n+ \Theta$, we have $\tilde{w}^\star_{\bar{\Omega}(j)} = x^\star_{\bar{\Omega}(j) - n}$, and $H_{\Omega}^{-1} H_{\cdot,\bar{\Omega}(j)} = H_{\Omega}^{-1}  \left(\begin{smallmatrix}
			0 \\ Q_{\cdot,\bar{\Omega}(j)-n}
		\end{smallmatrix}\right)  = \left(\begin{smallmatrix}
			0 \\ Q_{\bar{\Theta}}^{-1} Q_{\cdot,\bar{\Omega}(j)-n}
		\end{smallmatrix}\right) =
		\left(\begin{smallmatrix}
			0 \\ -\left((B^{-1}N)^\top\right)_{\cdot,\bar{\Omega}(j)-n}
		\end{smallmatrix}\right)
	$, where the last equality uses Lemma \ref{lm:formula_QinvQ}. Therefore,  
	\begin{equation*}
		\begin{aligned}
		\bar{D}_\delta  & =
		\delta \cdot \max\left\{
		\max_{\bar{\Omega}(j)\in\bar{\Theta}} \frac{\sqrt{\left\|\left(\begin{smallmatrix}
			B^{-1} A_{\cdot,\bar{\Omega}(j)} \\ 0
		\end{smallmatrix}\right)\right\|^2+1} }{s^\star_{\bar{\Omega}(j)}},  \max_{\bar{\Omega}(j)\in n+\Theta} \frac{\sqrt{\left\|\left(\begin{smallmatrix}
			0 \\ -\left((B^{-1}N)^\top\right)_{\cdot,\bar{\Omega}(j)-n}
		\end{smallmatrix}\right)\right\|^2+1} }{x^\star_{\bar{\Omega}(j)-n}}
		\right\} 
		\\
		&\overset{\eqref{eq:optimal_basis_[m]}}{=} \delta \cdot \max\left\{
		\max_{1\le j\le n-m} \frac{\sqrt{\left\|(B^{-1} N)_{\cdot,j}\right\|^2+1} }{s^\star_{m+j}},  \max_{1\le i \le m} \frac{\sqrt{\left\|(B^{-1} N)_{i,\cdot}\right\|^2+1} }{x^\star_{i}}
		\right\} =  \hat{D}_\delta\ .
		\end{aligned}
	\end{equation*}
	Finally, substituting $\bar{D}_\delta = \hat{D}_\delta$ back to \eqref{eq:Dr_primal_dual2} proves the first half of \eqref{eq:Dr_primal_dual}.

	As for the second half of \eqref{eq:Dr_primal_dual}, note that $\tilde{w}^\star = (s^\star,x^\star)$, so by Lemma \ref{lm:lower_of_r}, when $\delta$ is sufficiently small we have $r_\delta = \frac{\delta}{\|\tilde{w}^\star\|_1} = \frac{\delta}{\|x^\star\|_1+\|s^\star\|_1}$.\Halmos
\end{proof}

\vspace{5pt}

Finally, we are ready to prove Lemma \ref{lm:compute_p_d_condition_numbers}.
\begin{proof}{Proof of Lemma \ref{lm:compute_p_d_condition_numbers}.}
	 Lemma \ref{lm:compute_Dr_primal_dual} establishes that $\hat{D}_\delta \le D_\delta \le 2\hat{D}_\delta$ for sufficiently small $\delta$, so we can deduce $\lim_{\delta \searrow 0} \frac{\hat{D}_\delta}{r_\delta} \le  \lim_{\delta \searrow 0} \frac{D_\delta}{r_\delta} \le 2\cdot\lim_{\delta \searrow 0} \frac{\hat{D}_\delta}{r_\delta}$. Note that $\geophi = \lim_{\delta \searrow 0} \frac{D_\delta}{r_\delta}$ as defined in \eqref{eq:limiting_ratio}. Substituting the values of $\hat{D}_\delta$ and $r_\delta$ in Lemma \ref{lm:compute_Dr_primal_dual} into \eqref{def_geometric_measure} yields $\simplephi = \lim_{\delta \searrow 0} \frac{\hat{D}_\delta}{r_\delta}$. Therefore,	$\simplephi \le \geophi \le 2\simplephi$.\Halmos
\end{proof}

\section{Finite-Time Optimal Basis Identification and Fast Local Convergence}\label{sec:local_linear_convergence}
In this section, we investigate the two-stage performance of rPDHG. It is frequently observed in practice that the behavior of rPDHG transitions to faster local linear convergence in a neighborhood of the optimal solution, in which the support sets of all iterates keep consistent. \citet{lu2025geometry} study this phenomenon for vanilla PDHG. In the first stage, PDHG converges in a sublinear rate until identifying the active set of the converging solution. In the second stage, PDHG turns to faster local linear convergence. Recently, \citet{lu2024restarted} propose the restarted Halpern PDHG (rHPDHG), a variant of rPDHG, and prove its two-stage performance behavior. However, the iteration bounds of the two stages proven above (for both PDHG and rHPDHG) are not accessible because they both depend on the Hoffman constant of a linear system that is determined by the converging solution. 
The Hoffman constant is hard to analyze, challenging to compute, and may be too conservative. And computing the trajectory of the algorithm may also be difficult.

Based on the new iteration bounds obtained in Section \ref{sec:global_linear_convergence} for LPs with unique optima, this section will show an accessible refined convergence guarantee of rPDHG that avoids Hoffman constants entirely. Although using the additional assumption of unique optimum, the new iteration bounds for the two stages both have closed-form expressions and are thus straightforward to analyze and compute once an optimal solution/basis is available:

\begin{itemize}
	\item In Stage I, rPDHG identifies the optimal basis within at most $O(\kappa\simplephi\cdot\ln(\kappa\simplephi))$ iterations.  
	\item In Stage II, having identified the optimal basis $\Theta$, the behavior of rPDHG transitions to faster local linear convergence that is no longer related to $\simplephi$. In this stage, components of index in $\Theta$ are sufficiently bounded away from $0$, while all other components equal $0$. 
\end{itemize}
The following theorem summarizes the iteration bounds of the two stages:

\begin{theorem}\label{thm:local_linear_convergence}

	Suppose  Assumption \ref{assump:unique_optima} holds and $Ac = 0$.
	Let Algorithm \ref{alg: PDHG with restarts} (rPDHG) run with $\tau =   \frac{1}{2\kappa}$, $\sigma =  \frac{1}{2\lambda_{\max}\lambda_{\min}}$ and $\beta := 1/e$ to solve the LP.  Let $T_1$ be the total number of  \textsc{OnePDHG} iterations required to obtain $N_1$ such that for all $N\ge N_1$ the positive components of $x^{N,0}$ exactly correspond to the optimal basis. Then,
	\begin{equation}\label{eq:basis_identification_T}
		T_1 \le  T_{basis}\ , \quad \text{in which} \ \ \ T_{basis}:=O\left(\kappa  {\simplephi} \cdot   \ln\left(\kappa  {\simplephi}  \right) \right) .
	\end{equation}
	Furthermore, let $T_2$ be the total number of  \textsc{OnePDHG} iterations required to obtain the first $N_2$ for which $ w^{N_2,0}$ is $\eps$-optimal. Then,
	\begin{equation}\label{eq:local_linear_convergence_T}
		T_2 \le  T_{basis} + T_{local},\quad \text{in which} \ \ \ T_{local}:=  O\left(\|B^{-1}\|\|A\|    \cdot \max\left\{0,\ \ln\left( \frac{\min_{1\le i \le n}\left\{x_i^\star + s_i^\star\right\}}{\eps}\right)\right\}\right) .
	\end{equation}
\end{theorem}

Theorem \ref{thm:local_linear_convergence} demonstrates that it takes at most $T_{basis}$ iterations for rPDHG to identify the optimal basis (independent of $\eps$), after which it requires at most  $T_{local}$ additional iterations to achieve $\eps$-optimality.  
The above iteration bounds are not aimed at improving upon the iteration bound in Theorem \ref{thm:closed-form-complexity}, the main value of this result is that both $T_{basis}$ and $T_{local}$ are accessible, and they do not contain any Hoffman constant and are straightforward to compute and analyze if the optimal solution is known.
Furthermore, $T_{local}$ is independent of $\simplephi$, which is frequently considerably larger. This provides a partial explanation for why rPDHG often becomes significantly faster in Stage II compared to Stage I.  

The first-stage iteration bound $T_{basis}$ exhibits a linear relationship with $\frac{\|w^\star\|_1}{\xi}$ (in the expression of $\simplephi$), where we use $\xi$ to denote the smallest nonzero of $x^\star$ and $s^\star$, written as follows:
\begin{equation}\label{def:xi}
	\xi:=\min_{1\le i \le n}\left\{x_i^\star + s_i^\star\right\} \ . 
\end{equation}
The first-stage iteration bounds of the vanilla PDHG (without restarts) and the restarted Halpern PDHG proven by \citet{lu2025geometry,lu2024restarted} also depend on $\frac{\|w^\star\|_1}{\xi}$ (in a different norm), but they are further multiplied by a Hoffman constant of a linear system.
An empirical comparison with this characterization will be presented in Section \ref{sec:experiments}. 
A similar dependence on $\frac{\|w^\star\|}{\xi}$ and certain norms of $B^{-1}A$ is also observed in the complexity analysis of finite-termination results for IPMs, such as \citet{potra1994quadratically,anstreicher1999probabilistic}. In those cases of IPMs, this dependence appears only within logarithmic terms, but it comes with higher per-iteration cost. 

Once the optimal basis is identified, the method's behavior  automatically transitions into the second stage without changing the algorithm or additional ``crossover'' operations. 
Notably, the coefficient $\|B^{-1}\|\|A\|$ is solely determined by the optimal basis and the constraint matrix, which could be upper bounded by a constant that is only determined by the constraint matrix. A similar complexity bound that only depends on the constraint matrix also holds for IPMs, proven by \citet{vavasis1996primal}. Note that a smaller $\xi$ may slightly decrease $T_{local}$ but simultaneously significantly increase $T_{basis}$ because $T_{basis}$ is linear in $\frac{\|w^\star\|_1}{\xi}$.

The significance of $\kappa\simplephi$ and $\|B^{-1}\|\|A\|$ in the two-stage performance will be empirically confirmed in Section \ref{sec:experiments}. 
The remainder of this section will show the proof sketch of Theorem \ref{thm:local_linear_convergence} and some important lemmas.

\textbf{Proof sketch of Theorem \ref{thm:local_linear_convergence} and outline of the rest of this section.}
In the rest of this section, we outline some key steps in the proof of Theorem \ref{thm:local_linear_convergence}. The proof proceeds in three steps.

\emph{(i) Global linear convergence with adaptive restart condition.}
First of all, we recall the proof idea of the global linear convergence of rPDHG in Section \ref{subsec:adaptive_restart_condition}. Under a certain ``sharpness'' condition with a parameter $\condL$ (formally defined in \eqref{condition L}) and the $\beta$-restart condition, a certain stationarity measure (the normalized duality gap) is forced to decrease by a fixed factor at every outer iteration, within at most $O(\condL/\beta)$ \textsc{OnePDHG} steps. This implies that rPDHG converges linearly to any arbitrary ``$\eps$-close'' neighborhood of the optimal solution within $\tilde{O}(\condL\cdot\log(1/\eps))$ \textsc{OnePDHG} steps, where the $\tilde{O}$ hides absolute constants and logarithmic terms of $\condL$ and $\|w^\star\|$.

\emph{(ii) Local linear sharpness and local linear convergence.}
While the global linear convergence relies on a sharpness condition with a constant $\condL$ that holds globally, the sharpness condition also holds with a potentially much smaller local constant $\condL_{\mathrm{loc}}$ in a neighborhood of $z^\star$, which leads to potentially faster local linear convergence. In Section \ref{subsec:basis_identification}, Lemma \ref{lm:local_condL} shows that once the optimal basis has been identified and the primal variables are bounded sufficiently away from zero, a local sharpness condition holds with a local constant $\condL_{\mathrm{loc}} = O(\|B^{-1}\|\|A\|)$. Plugging this $\condL_{\mathrm{loc}}$ into the same generic rPDHG convergence theorem yields the Stage-II iteration bound $T_{\mathrm{local}}$ in
\eqref{eq:local_linear_convergence_T}.

\emph{(iii) Finite-time basis identification and neighborhood entrance.}
In Section \ref{subsec:basis_identification}, Lemma \ref{lm:distance_until_local_linear} shows that the condition in Step~(ii) is satisfied once an iterate enters a sufficiently small neighborhood of $z^\star$. Because of the unique optimal basis assumption, if an iterate is close enough to the optimal solution, complementary slackness forces the primal iterate of the next iteration to have the same support set as $x^\star$ after one \textsc{OnePDHG} step. Furthermore, due to the ``nonexpansiveness'' property of PDHG, once an outer iterate enters this neighborhood, all subsequent outer iterates and averaged inner iterates inherit the condition in Step~(ii).

Combining these three steps gives the two-stage bound claimed in
Theorem~\ref{thm:local_linear_convergence}. 
The above three-step proof scheme is standard in the literature of the two-stage convergence, see e.g., \citet{lu2024restarted}, while we prove the accessible bounds in each step, which avoid any Hoffman constant.
The proof of the lemmas in the following subsection and the complete proof of Theorem \ref{thm:local_linear_convergence} are all deferred
to Appendix~\ref{app:proofs_local_linear_convergence}.

\subsection{Global linear convergence with adaptive restart condition}\label{subsec:adaptive_restart_condition}
This subsection recalls the global linear convergence of rPDHG with the $\beta$-restart condition. It is built on a metric of stationarity, called the ``normalized duality gap,'' proposed by  \citet{applegate2023faster}.

\begin{definition}[Normalized duality gap, \citet{applegate2023faster}]\label{def:normalized_duality_gap}
	For any $z = (x,y)\in \mathbb{R}^{m+n}$ and a certain $r > 0$,
	the normalized duality gap of the saddle-point problem \eqref{pro:saddlepoint_LP} is then defined as
	\begin{equation}\label{sunny}
		\rho(r;z) := \frac{1}{r}\sup_{  \hat{z} \in B(r;z) }  \big[ L(x,\hat{y}) - L(\hat{x},y) \big] 
	\end{equation}
	in which $
	B(r;z) := \left\{\hat{z} := (\hat{x},\hat{y}):  \hat{x}\in \R^n_+ \text{ and } \|\hat{z} -z\|_M \le r  \right\}$.
\end{definition}
\noindent
The normalized duality gap $\rho(r;z)$ is also a valid measure of the optimality and feasibility errors of $z$, and it can be efficiently computed or approximated in strongly polynomial time as shown by \citet{applegate2023faster}.  Furthermore, let the $k$-th iterate of PDHG be denoted as $z^k$, and let the average of the first $k$ iterates be $\bar{z}^k := \frac{1}{k}\sum_{i=1}^k z^i$. The normalized duality gap $\rho(\|z^0-\bar{z}^k\|_M;\bar{z}^k)$ of the average iterate $\bar{z}^k$ converges to $0$ sublinearly. See \citet{applegate2023faster,xiong2023computational} for more details.  

Then the $\beta$-restart condition holds if and only if $\ell=0$ and $k=1$, or
\begin{equation}\label{catsdogs}\rho(\|\bar{z}^{\ell,k} - z^{\ell,0}\|_M; \bar{z}^{\ell,k}) \le \beta \cdot \rho(\|z^{\ell,0} - z^{\ell-1,0}\|_M; z^{\ell,0}) \ , \end{equation}
for a chosen value of $\beta \in (0,1)$.
In words, the restart triggers once the normalized duality gap at $\bar{z}^{\ell, k}$, evaluated with radius $\| \bar{z}^{\ell, k}- z^{\ell, 0} \|_M$, is at most a factor $\beta$ times the normalized duality gap at $z^{\ell, 0}$, evaluated with radius $\| z^{\ell, 0}- z^{\ell-1,0} \|_M$. 

For LP problems, it can be shown that there exists a constant $\condL\ge 0$ such that the following condition holds for all outer-loop indices $\ell \ge 1$:
\begin{equation}\label{condition L}
	\dist_M(z^{\ell,0}, \calZ^\star) \le \rho(\|z^{\ell,0} - z^{\ell-1,0} \|_M;z^{\ell,0}) \cdot \condL \ .
\end{equation}
This condition indicates the $M$-distance to the optimal solutions is upper bounded by the normalized duality gap multiplied by the fixed constant $\condL$. \citet{applegate2023faster} refer to this as the ``sharpness'' property of the normalized duality gap. Under this condition, each inner loop requires at most $\left\lceil\frac{4\condL}{\beta}\right\rceil$ iterations to achieve sufficient decrease in the normalized duality gap (see Appendix \ref{sec:proofs_global_linear_convergence} for a formal statement and proof). 
Therefore, $\ell$ outer loops contain at most $O\left(\tfrac{\ell\condL}{\beta}\right)$ \textsc{OnePDHG} iterations, while reducing the normalized duality gap to a $\beta^\ell$ fraction of its initial value.  
This leads to a linear convergence rate dependent on $\condL$. 
Indeed, Lemma 3.13 of \citet{xiong2024role} demonstrates that \eqref{condition L} holds with $\calL=O(\kappa\geophi)$, which then leads to the global convergence rate in Lemma \ref{lm:global_linear}.

These results are the cornerstone of Step~(i) of the proof of Theorem \ref{thm:local_linear_convergence}: the $\beta$-restarts drive both the normalized duality gap and the $M$-distance to the solution set down at a global linear rate governed by $O(\condL)$.

\subsection{Local linear sharpness and local linear convergence}\label{subsec:proof_local_linear_convergence}

 The linear convergence of rPDHG relies on the condition \eqref{condition L}.  If \eqref{condition L} holds for all $\ell\ge 1$ with a constant $\condL \ge 0$, then the global linear convergence is established by showing the number of iterations required for each inner loop does not exceed  $O(\tfrac{\condL}{\beta})$. We will now demonstrate the existence of a close neighborhood of the optimal solution $z^\star$ within which \eqref{condition L} holds with a potentially much smaller $\condL_{\mathrm{loc}}$ than the global $\condL$, so rPDHG requires a smaller number of iterations for the inner loops and therefore has a faster local linear convergence rate.  
We demonstrate that the condition \eqref{condition L} indeed holds with an alternative $\condL_{\mathrm{loc}}$ for the iterates $\bar{z}=(\bar{x},\bar{y})$ if (i) each component of $\bar{x}_\Theta$ is sufficiently bounded away from $0$ and (ii) each component of $\bar{x}_{\bar{\Theta}}$ equals zero.
For simplicity, we assume that the optimal basis is $\{1,2,\dots,m\}$. We will use the following step-size dependent constant:
\begin{equation}\label{eq:def_c_tau_sigma}
	 c_{\tau,\sigma} := \max\left\{
		\frac{1}{\sqrt{\sigma} \lambda_{\min}  }, \ \frac{1}{\sqrt{\tau}}
		\right\}    \ .
\end{equation} 
When $\tau =   \frac{1}{2\kappa}$ and $\sigma =  \frac{1}{2\lambda_{\max}\lambda_{\min}}$ (the step-sizes used by Theorem \ref{thm:local_linear_convergence}), we have $c_{\tau,\sigma} = \sqrt{2\kappa}$.

\begin{lemma}\label{lm:local_condL}
	Under Assumption \ref{assump:unique_optima}, for any $\bar{z}=(\bar{x},\bar{y})$ and $r > 0$ such that
	\begin{equation}\label{eq:tigher_bound_region}
		(i) \ \ \bar{x}_{i} \ge r \sqrt{\tau} \text{ for } i\in[m] , \ \ \text{ and }  \ \ (ii) \ \bar{x}_{m+j} = 0 \text{ for } j\in[n-m] \ ,
	\end{equation}
	it holds that
	\begin{equation}\label{eq:local_condL}
		\|\bar z-z^\star\|_M \le \frac{2\|B^{-1}\|}{\sqrt{\tau\sigma}}\rho(r;\bar z).
	\end{equation}
\end{lemma}
\noindent
When the step-sizes are carefully chosen,  \eqref{eq:local_condL} in Lemma \ref{lm:local_condL} can be further simplified:
\begin{remark}\label{rmk:local_condL}
	With the choice of step-sizes $\tau = \frac{1}{2\kappa}$ and $\sigma = \frac{1}{2\lambda_{\max}\lambda_{\min}}$, \eqref{eq:local_condL} becomes
	\begin{equation}\label{eq_ofprop:local_condL_2}
		\|\bar{z} - z^\star\|_M \le 4\|B^{-1}\|\|A\| \cdot \rho(r;\bar{z}) \ .
	\end{equation}
\end{remark}

The above results are the foundation of Step~(ii) of the proof sketch for Theorem \ref{thm:local_linear_convergence}: there exists a neighborhood of $z^\star$ within which the sharpness condition \eqref{condition L} holds with  $\condL_{\mathrm{loc}} = O(\|B^{-1}\|\|A\|)$, leading to potentially faster local linear convergence.

\subsection{Finite-time basis identification and neighborhood entrance}\label{subsec:basis_identification}

We now show that the two conditions in Lemma \ref{lm:local_condL} are automatically satisfied by all average iterations $\bar{z}^{N,k}= (\bar{x}^{N,k},\bar{y}^{N,k}) = \frac{1}{k}\sum_{i=1}^k z^{N,i}$ once a previous outer loop iteration $z^{t,0}$ (with $t\le N$) is sufficiently close to the optimal solution $z^\star$ under the $M$-norm distance. 
Lemma~\ref{lm:distance_until_local_linear} below formalizes Step~(iii) of the proof of Theorem \ref{thm:local_linear_convergence}.
It quantifies how close an outer iterate $z^{t,0}$ needs to be to $z^\star$ in the $M$-norm in order for the subsequent averaged iterates $\bar{z}^{N,k}$ to identify the optimal basis and enter the neighborhood where Lemma~\ref{lm:local_condL} applies.
\begin{lemma}\label{lm:distance_until_local_linear}
	Under Assumption \ref{assump:unique_optima}, suppose that Algorithm \ref{alg: PDHG with restarts} (rPDHG) has any $\beta$-restart condition, and let the step-sizes satisfy \eqref{eq:general_stepsize} strictly. If there exists an outer-loop index $t$ such that
	\begin{equation}\label{eq:lm:distance_until_local_linear_0}
		\left\|z^{t,0} - z^\star\right\|_M \le  \bar{\eps}:= \frac{\sqrt{1 - \sqrt{\tau\sigma}\|A\|}}{3}\cdot \min\left\{
		\frac{1}{\sqrt{\tau}},   \sqrt{\tau}
		\right\}\cdot \left(\min_{1\le i \le n}\ \left\{x_i^\star + s_i^\star\right\}\right)  \ ,
	\end{equation}
	then for any $N \ge t$ and $k \ge 1$,  the following conditions hold:
	\begin{equation}\label{eq:lm:distance_until_local_linear_1}
		(i) \ \ \bar{x}^{N,k}_{i} \ge  \sqrt{\tau} \left\|\bar{z}^{N,k} - z^{N,0}\right\|_M \ \text{ for }i\in[m] \  , \ \ \text{ and }  \ \ (ii) \ \bar{x}^{N,k}_{m+j} = 0 \text{ for }j\in[n-m] \ .
	\end{equation}
	 Moreover, the positive components of $\bar{x}^{N,k}$ correspond exactly to the optimal basis.
\end{lemma}
With the above lemma, directly using Lemma \ref{lm:global_linear} yields the finite-time basis identification and neighborhood entrance. This lemma is the key of Step~(iii) of the proof of Theorem \ref{thm:local_linear_convergence}.

Overall, Lemmas \ref{lm:local_condL} and \ref{lm:distance_until_local_linear} provide the foundation of the proof of Theorem \ref{thm:local_linear_convergence}. 
In Stage I, rPDHG converges to a neighborhood of the optimal solution such that condition \eqref{eq:lm:distance_until_local_linear_0} of Lemma \ref{lm:distance_until_local_linear} is satisfied. The number of iterations in this stage is determined by the linear convergence rate established in Lemma \ref{lm:global_linear}. In Stage II, rPDHG converges to the optimal solution with accelerated local linear convergence, owing to the potentially smaller $\condL$ provided by Lemma \ref{lm:local_condL}. The complete proofs of Lemmas \ref{lm:local_condL} and \ref{lm:distance_until_local_linear} and Theorem \ref{thm:local_linear_convergence} are all deferred to Appendix \ref{app:proofs_local_linear_convergence}.

\section{Relationship of $\simplephi$ with Stability under Data Perturbations, Proximity to Multiple Optima, and LP Sharpness}\label{sec:LP_sharpness}

The previous sections established new accessible iteration bounds for rPDHG in terms of the geometric quantity $\simplephi$, which admits a closed-form expression. In this section, we relate $\simplephi$ to a measure of stability under data perturbations, and we use this relationship to derive a computational guarantee stated directly in terms of these quantities.
This measure of stability is also equivalent to two other condition measures (i) proximity to multiple optima, and (ii) the LP sharpness of the instance. 
These interpretations help explain rPDHG's sensitivity to small data perturbations, and reveal connections to other works; for example, \citet{lu2025geometry,lu2024restarted} use a parameter of ``near-degeneracy'' to bound the duration of the Stage I of PDHG, and \citet{xiong2023computational,xiong2023relation} provide computational guarantees of rPDHG using the LP sharpness.

We start by defining two quantities of primal and dual stabilities. For the original problem \eqref{pro:primal LP}, let the perturbed problem be as follows:
\begin{equation}\label{pro:perturbed primal LP}
	\min \ \tilde{c}^\top x \quad \text{s.t.} \ Ax = \tilde{b} \ , \ x \ge 0 
\end{equation}
where $\tilde{c}$ and $\tilde{b}$ might be the perturbed versions of $c$ and $b$ respectively. When Assumption \ref{assump:unique_optima} holds for \eqref{pro:primal LP}, we define $\zeta_p$ and $\zeta_d$ as follows:
\begin{equation}\label{eq:primal_stability_perturbed_LP}
	\zeta_p:=\inf\left\{
		\|\Delta c\| : \text{ $\Theta$ is not the unique optimal basis for \eqref{pro:perturbed primal LP} with $\tilde{c} = c + \Delta c$ and $\tilde{b}=b$} 
	\right\} \ ,
\end{equation}
\begin{equation}\label{eq:dual_stability_perturbed_LP}
	\zeta_d:=\inf\left\{
		\|\Delta b\|_{(AA^\top)^{-1}} : \text{ $\Theta$ is not the unique optimal basis for \eqref{pro:perturbed primal LP} with $\tilde{c} = c  $ and $\tilde{b}=b + \Delta b$} 
	\right\} \ .
\end{equation}
The $\zeta_p$ and $\zeta_d$ denote the size of the smallest perturbation on the cost vector $c$ and the right-hand side vector $b$, respectively, such that the optimal basis becomes different. In other words, the larger they are, the more stable the optimal basis is under data perturbations on $b$ and $c$. Here $\zeta_d$ uses the $(AA^\top)^{-1}$-norm instead of the Euclidean norm because later we will show that $\zeta_d$ can be defined in the symmetric way to $\zeta_p$ on the symmetric form \eqref{pro:symmetric_primal_dual}.

More importantly, $\simplephi$ has a close relationship with $\zeta_p$ and $\zeta_d$, leading to a new computational guarantee using $\zeta_p$ and $\zeta_d$. Below is the result of this section.

\begin{theorem}\label{thm:compexity_using_zeta}
	Suppose Assumption \ref{assump:unique_optima} holds. The following relationship holds for $\simplephi$, $\zeta_p$, and $\zeta_d$:
	\begin{equation}\label{eq:thm:compexity_using_zeta_1}
		\simplephi
		=  \frac{ \|x^\star\|_1+\|s^\star\|_1}{\min\left\{\zeta_p, \ \zeta_d \right\}}  \ .
	\end{equation}
	Therefore, in the identical setting of Theorem \ref{thm:closed-form-complexity}, the total number of \textsc{OnePDHG} iterations required to compute an $\eps$-optimal solution is at most  
	$$
	O\left(\kappa \cdot \frac{ \|x^\star\|_1+\|s^\star\|_1}{\min\left\{\zeta_p, \ \zeta_d\right\}} \cdot \ln\left( \frac{\kappa  \frac{ \|x^\star\|_1+\|s^\star\|_1}{\min\left\{\zeta_p, \ \zeta_d\right\}}\cdot \|w^\star\|}{\eps}\right)\right).
	$$  
\end{theorem}

This theorem implies that the less stable the optimal basis is under data perturbations, the larger the value of $\simplephi$, and the more iterations rPDHG might require to compute an $\eps$-optimal solution.  
Actually in Section \ref{sec:experiments} we provide empirical support of these bounds via experiments on LP instances.
Small values of $\min\left\{\zeta_p, \ \zeta_d \right\}$ may significantly affect the performance of rPDHG, because they stay in the denominator in the expression of $\simplephi$ in \eqref{eq:thm:compexity_using_zeta_1}. It should be noted that $\zeta_p$ and $\zeta_d$ are not intrinsic properties of the constraint matrix as they are also dependent on $c$ and $b$, so $\zeta_p$ and $\zeta_d$ do not affect the Stage II iteration bound $T_{local}$ in Theorem \ref{thm:local_linear_convergence}.

In the remainder of this section, Section~\ref{subsec:def_LP_sharpness} introduces the equivalence relationship between the measure of stability under data perturbations and two other condition measures, proximity to multiple optima and LP sharpness.  We then highlight the implications of these relationships for rPDHG's sensitivity to small perturbations and discuss the connection with other work. Finally, Section~\ref{subsec:proof_zeta_thm} presents the proof of Theorem~\ref{thm:compexity_using_zeta}.

\subsection{Stability under data perturbations, proximity to multiple optima, LP sharpness}\label{subsec:def_LP_sharpness}

Both the primal and dual problems in the symmetric form \eqref{pro:symmetric_primal_dual} are instances of the following generic form of LP:
\begin{equation}\label{pro:generic LP}
	\min  \ g^\top u \quad \ \ \text{s.t.} \ \ u \in \calFu := \Vu\cap \R^n_+ \
\end{equation}
where the feasible set $\calFu$ is the intersection of the nonnegative orthant $\R^n_+$ and an affine subspace $\Vu$. The objective function $g^\top u$ is a linear function. We denote the optimal solution of \eqref{pro:generic LP} by $\calU^\star$, in which $u^\star$ is an optimal solution.   We let $\opt(\check{g})$ denote the set of optimal solutions of the generic LP \eqref{pro:generic LP} with the objective vector equal to $\check{g}$.  For example, $\opt(g) = \calU^\star$.
The primal problem \eqref{pro:primal LP} is an instantiation of \eqref{pro:generic LP} with $\calFu = \calF_p$, $\Vu=V_p$, and $g = c$. Similarly, the dual problem \eqref{pro:dual LP_s} is another instantiation of \eqref{pro:generic LP} with $\calFu = \calF_d$, $\Vu=V_d$, and $g = q$.
With this symmetric form, the stability under data perturbation $\zeta$ can be defined as follows.
\begin{definition}[Stability under data perturbations]\label{def: stability}  The stability under data perturbations is defined as
	\begin{equation}\label{eq:stability}
		\zeta  : = \inf_{\Delta g}\big\{
			\|\Delta g\|: 
				\opt(g + \Delta g) \ne \emptyset  \text{ and } 
				\opt(g + \Delta g) \not\subseteq \opt(g) 
			\big\} \ .
	\end{equation}
\end{definition} 
The $\zeta$ values for \eqref{pro:primal LP} and \eqref{pro:dual LP_s} are essentially $\zeta_p$ and $\zeta_d$, respectively, as stated in the following remark.
\begin{remark}\label{rmk:zeta_p_d}
	Suppose Assumption \ref{assump:unique_optima} holds. The values of $\zeta$ for \eqref{pro:primal LP} and \eqref{pro:dual LP_s} are equal to $\zeta_p$ and $\zeta_d$, respectively.
\end{remark}
The proof of Remark \ref{rmk:zeta_p_d} is straightforward and thus deferred to Section \ref{subsec:proof_def_LP_sharpness}.

Furthermore, we introduce another two measures. 
The first one is the proximity to multiple optima, defined as the size of the smallest perturbation that leads to multiple optimal solutions. 
\begin{definition}[Proximity to multiple optima]\label{def: proximity to multiple optima}
	When $\opt(g)$ is a singleton, the proximity to multiple optima is defined as
	\begin{equation}\label{eq: proximity to multiple optima}
		\eta  = \min_{\Delta g}\big\{
		\|\Delta g\|:
		|\opt(g + \Delta g) | > 1
		\big\} \ . 
	\end{equation}
\end{definition}
\noindent Since having multiple optima is equivalent to having degenerate dual optimal solutions, $\eta$ can also be interpreted as the proximity to degenerate dual optima.

The second measure is LP sharpness (see \citet{xiong2023computational}), which measures how quickly the objective function grows away from the optimal solution set $\calU^\star$ (i.e., $\opt(g)$) among all feasible points. 
\begin{definition}[LP sharpness]\label{def: sharpness mu}  The LP sharpness  of \eqref{pro:generic LP} is defined as
	\begin{equation}\label{topsyturvy}
		\mu : = \inf_{u \in \calFu \setminus\calU^\star} \frac{\dist(u,\Vu \cap \left\{u \in \mathbb{R}^n:g^\top u = g^\top u^\star\right\}   )}{\dist(u,\ \calU^\star)}
		\ .
	\end{equation}
\end{definition}
\noindent Sharpness is a useful analytical tool (for example, see \citet{yang2018rsg,lu2022infimal}), and \citet{applegate2023faster} employ the sharpness of the normalized duality gap for the saddle-point problem to prove the linear convergence of rPDHG on LP.  Sharpness for LP, denoted by LP sharpness, is a more natural and intuitive measure of the original LP instance.

\vspace*{7pt} 
When $\calU^\star$ is a singleton, the above three measures are equivalent in the following sense.
\begin{proposition}\label{prop:equivalence_3_measures}
	When $\calU^\star$ is a singleton, the following relationship holds:
	\begin{equation}\label{eq:thm:equivalence_3_measures}
		\frac{1}{\big\|P_{\linVu}(g)\big\|} \cdot \zeta= \frac{1}{\big\|P_{\linVu}(g)\big\|} \cdot \eta = \mu  \ .
	\end{equation}
\end{proposition}
\noindent

Here $\zeta$ is normalized by the norm of $P_{\linVu}(g)$, indicating that $\zeta$ and $\eta$ are equivalent to $\mu$ in a relative sense. This normalization arises because $\mu$ is a purely geometric concept that remains invariant under positive scaling of $g$. We use the norm of $P_{\linVu}(g)$ rather than $g$ because the complementary part $P_{\linVu^\bot}(g)$ does not have any impact on the optimal solution and the smallest perturbation $\Delta g$ must lie in $\linVu$. 
The equivalence between $\mu$ and $\zeta$ has already been proven by \citet{xiong2023computational}, so it remains only to prove the equivalence between $\zeta$ and $\eta$, which is straightforward.
We defer the complete proof of Proposition \ref{prop:equivalence_3_measures} to Section \ref{subsec:proof_def_LP_sharpness}. 
We use $\mu_p$ and $\mu_d$ to denote the $\mu$ for the primal problem \eqref{pro:primal LP} and the dual problem \eqref{pro:dual LP_s}, for which $\left\|P_{\linVp}(c)\right\|$ and $\|q\|$ correspond to $\left\|P_{\linVu}(g)\right\|$ (Fact \ref{fact: symmetric LP formulation}).

\vspace*{7pt}
\textbf{Interpreting the sensitivity of rPDHG to perturbations and connections to other work.} \ \ 
Theorem \ref{thm:compexity_using_zeta} and the above equivalence relationships of $\zeta_p$, $\zeta_d$ with the other two condition measures also provide new insights into the performance of rPDHG. Here we detail them as follows.

It is often observed that rPDHG has good performance in some LP instances with multiple optimal solutions, but a minor data perturbation results in a substantial degradation in rPDHG's performance on the perturbed problem (see Section \ref{sec:experiments} for examples). Now we have some insight and a partial explanation for this phenomenon. According to Proposition \ref{prop:equivalence_3_measures} and Remark \ref{rmk:zeta_p_d}, $\zeta_p$ and $\zeta_d$ are equal to the proximity to multiple optima for \eqref{pro:primal LP} and \eqref{pro:dual LP_s}. The perturbed problem still stays in close proximity to the original unperturbed problem, so its $\min\{\zeta_p,\zeta_d\}$ is at most as large as the magnitude of the perturbation. Furthermore, $\min\{\zeta_p,\zeta_d\}$ lies in the denominator of the iteration bound of Theorems \ref{thm:compexity_using_zeta} (and Stage I iteration bound of Theorem \ref{thm:local_linear_convergence}), so a small perturbation may significantly increase the iteration bound. This explanation is complementary with \citet{lu2025geometry} which study the vanilla PDHG and explain the phenomenon by the size of the region for local fast convergence. Our result shows that perturbations significantly affect the iteration bound for Stage I of rPDHG as well. This effect of perturbations on rPDHG will be confirmed by computational experiments in Section \ref{sec:experiments}.

Moreover, Theorem \ref{thm:compexity_using_zeta} together with Remark \ref{rmk:zeta_p_d} and Proposition \ref{prop:equivalence_3_measures} also provide new computational guarantees that use $\kappa$, the size of the optimal solutions, and the LP sharpness $\mu_p$ and $\mu_d$. They are simpler and more intuitive guarantees than those in \citet{xiong2023computational}, because the latter also involves the limiting error ratio. This simplification is due to the unique optimum assumption. 
Notably, our iteration bounds are strictly better than the iteration bounds in \citet[Corollary 4.3]{xiong2023relation}, which also use this additional assumption but exhibit quadratic dependence on the reciprocals of $\mu_p$ and $\mu_d$.

\subsection{Proof of Remark \ref{rmk:zeta_p_d} and Proposition \ref{prop:equivalence_3_measures}}\label{subsec:proof_def_LP_sharpness}

We first prove Remark \ref{rmk:zeta_p_d}:
\begin{proof}{Proof of Remark \ref{rmk:zeta_p_d}.}
	For $\zeta$ of \eqref{pro:primal LP}, $\opt(c+\Delta c) \not\subseteq \opt(c)$ if and only if $\Theta$ is no longer an optimal basis or the optimal basis is no longer unique for the perturbed problem. This proves $\zeta_p$ is equal to $\zeta$ of \eqref{pro:primal LP}.

	Similarly,  $\zeta$ of \eqref{pro:dual LP_s} is defined as the Euclidean norm of the smallest perturbation $\Delta q$ on the objective function so that $\bar{\Theta}$ is no longer the unique optimal basis for \eqref{pro:dual LP_s}. Since all feasible solutions are in the affine subspace $c+\operatorname{Im}(A^\top)$, the smallest perturbation $\Delta q$ must lie in $\operatorname{Im}(A^\top)$ and there exists $\Delta b_0\in\mathbb{R}^m$ such that $\Delta q = A^\top (AA^\top)^{-1} \Delta b_0$. When applying the right-hand side perturbation $\Delta b_0$ on \eqref{pro:primal LP}, the corresponding dual problem \eqref{pro:dual LP_s} has the objective vector $A^\top (AA^\top)^{-1}(b+\Delta b_0)$ that is exactly equal to $q + \Delta q$. In this case $\|\Delta q\| =  \|A^\top (AA^\top)^{-1}\Delta b\| = \|\Delta b\|_{(AA^\top)^{-1}}$. Therefore, the smallest Euclidean norm of the objective vector perturbation $\Delta q$ on \eqref{pro:dual LP_s} such that $\bar{\Theta}$
	is no longer the unique optimal basis is equal to the smallest $(AA^\top)^{-1}$-norm of right-hand side perturbation $\Delta b$ on \eqref{pro:primal LP} such that $\Theta$ is no longer the unique optimal basis. This proves  $\zeta_d$ is equal to $\zeta$ of \eqref{pro:dual LP_s}.\Halmos
\end{proof}

We then prove Proposition \ref{prop:equivalence_3_measures}. Before that, we show a useful result that will be frequently used later. 
\begin{lemma}[Theorem 5.1 of \citet{xiong2023computational}]\label{lm: sharpness and perturbation}
	LP sharpness  is equivalent to stability under data perturbations through the relation: $\mu = \zeta\cdot\frac{1}{\big\|P_{\linVu}(g)\big\|}$. 
\end{lemma} 
\noindent With Lemma \ref{lm: sharpness and perturbation} we can prove Proposition \ref{prop:equivalence_3_measures}.
\begin{proof}{Proof of Proposition \ref{prop:equivalence_3_measures}.}
	First of all, according to Lemma \ref{lm: sharpness and perturbation}, it suffices to prove first equality in \eqref{eq:thm:equivalence_3_measures}.
	In the case that $\opt(g)$ (i.e., $\calU^\star = \{u^\star\}$) is a singleton, $\zeta$ is the smallest magnitude of the perturbation that leads to multiple optimal solutions at the threshold at which a new solution is added to $\opt(g + \Delta g)$ while $u^\star$ remains optimal. Therefore, $\zeta$ is equal to $\eta$. This finishes the proof.  \Halmos
\end{proof}

\subsection{Proof of Theorem \ref{thm:compexity_using_zeta}}\label{subsec:proof_zeta_thm}

The key to proving Theorem \ref{thm:compexity_using_zeta} is the following lemma:
\begin{lemma}\label{lm:compute_p_d_mu}
	Suppose that Assumption \ref{assump:unique_optima} holds. The $\zeta_p$ and $\zeta_d$ have the following expression:
	\begin{equation}\label{eq_of_lm:compute_p_d_mu}
		\zeta_p =  \min_{1\le j\le n-m}\frac{s^\star_{m+j}}{\sqrt{\left\|(B^{-1} N)_{\cdot,j}\right\|^2+1}}  \   \text{ and } \  \zeta_d =  \min_{1\le i \le m}\frac{x^\star_{i}}{\sqrt{\left\|(B^{-1} N)_{i,\cdot}\right\|^2+1}} \ .
	\end{equation}
\end{lemma}

Before proving it, we recall how to compute the LP sharpness $\mu$ of the generic LP \eqref{pro:generic LP} by computing the smallest sharpness along all of the edges emanating from the optimal solutions.  

\begin{lemma}[A consequence of Theorem 5.2 of \citet{xiong2023computational}]\label{lm:compute_generic_mu}
	Suppose that Assumption \ref{assump:unique_optima} holds.
	Let $\calU^\star = \{u^\star\}$ and the directions of the edges emanating from $\calU^\star$ be $v^1,v^2,\dots,v^{n-m}$. Then for any given $\bar{\eps} > 0$, the LP sharpness $\mu$ is characterized as follows:
	\begin{equation}\label{eq:compute_mu_p_d}
		\mu = \min_{1\le j\le n-m} \frac{\dist(u^\star + \bar{\eps}\cdot v^j, \Vu \cap \left\{u \in \mathbb{R}^n:g^\top u = g^\top u^\star\right\})}{\|\big(u^\star + \bar{\eps}\cdot v^j\big) - u^\star\|}
	\end{equation}
\end{lemma}  
\noindent
Now we are ready to prove Lemma \ref{lm:compute_p_d_mu}.
\begin{proof}{Proof of Lemma \ref{lm:compute_p_d_mu}.}
    Let $\mu_p$ and $\mu_d$ denote the LP sharpness for \eqref{pro:primal LP} and \eqref{pro:dual LP_s} respectively. 
    We first compute $\mu_p$ and then $\zeta_p=\left\|P_{\linVp}(c)\right\|\mu_p$ (Lemma \ref{lm: sharpness and perturbation}). For \eqref{pro:primal LP}, $u^\star =  x^\star$, $\Vu = V_p$ and $\left\{u \in \mathbb{R}^n:g^\top u = g^\top u^\star\right\} = \left\{x:\epobj(x)=0\right\}$ in the  generic LP \eqref{pro:generic LP}, with the directions of connected edges given by $\{u^j:1\le j \le n-m\}$ as defined in \eqref{eq:directions_edges_primal}. Furthermore, the following equalities hold:
	\begin{equation}\label{eq:compute_mu_p_1}
		\begin{aligned}
			     & \dist\left(x^\star + \bar{\eps}\cdot u^j,V_p\cap \left\{x:\epobj(x)=0\right\}\right) = \frac{\epobj(x^\star + \bar{\eps}\cdot u^j)}{\left\|P_{\linVp}(c)\right\|}
			\\
			= \  & \frac{(s^\star)^\top (x^\star + \bar{\eps}\cdot u^j )  - (s^\star)^\top x^\star}{\left\|P_{\linVp}(c)\right\|}  = \bar{\eps}\cdot \frac{(s^\star)^\top    u^j }{\left\|P_{\linVp}(c)\right\|}  = \bar{\eps}\cdot \frac{   s^\star_{m+j} u^j_{m+j}}{\left\|P_{\linVp}(c)\right\|}
			= \bar{\eps}\cdot \frac{   s^\star_{m+j}  }{\left\|P_{\linVp}(c)\right\|}
		\end{aligned}
	\end{equation}
	for all $j\in[n-m]$.
	Here the second and third equalities use the result that $\epobj(x) = \gap(x,s^\star) = (s^\star)^\top x - (s^\star)^\top x^\star $  for any $x \in V_p$. 
	The fourth equality holds because $s^\star_{[m]} = 0$ and $u^j_{[n]\setminus([m]\cup\{m+j\})}=0$.
	The final equality uses $u^j_{m+j} = 1$. 
	In addition, for all $j\in[n-m]$ we have
	\begin{equation}\label{eq:compute_mu_p_2}
		\|(x^\star + \bar{\eps}\cdot u^j) - x^\star\| = \bar{\eps}\cdot \|u^j\| = \bar{\eps} \cdot \sqrt{\left\|(B^{-1} N)_{\cdot,j}\right\|^2+1}  \ .
	\end{equation}
	Substituting  \eqref{eq:compute_mu_p_1} and \eqref{eq:compute_mu_p_2} into the definition \eqref{eq:compute_mu_p_d} of $\mu_p$ yields the expression: $$
		\mu_p = \frac{1}{\left\|P_{\linVp}(c)\right\|}\cdot \min_{1\le j\le n-m}\frac{s^\star_{m+j}}{\sqrt{\left\|(B^{-1} N)_{\cdot,j}\right\|^2+1}}\ .
	$$
	Finally, substituting it back to $\zeta_p=\left\|P_{\linVp}(c)\right\|\mu_p$ proves the first half of \eqref{eq_of_lm:compute_p_d_mu}.

    Next we compute $\mu_d$ and then $\zeta_d = \left\|P_{\linVd}(q)\right\|\mu_d$ (Lemma \ref{lm: sharpness and perturbation}). For \eqref{pro:dual LP_s}, we repeat the above process on \eqref{pro:standard_dual_LP_s}, and then we can symmetrically obtain 
		$
		\mu_d = \frac{1}{\left\|P_{\linVd}(q)\right\|}\cdot \min_{1\le i \le m}\frac{x^\star_{i}}{\sqrt{\left\|(Q_{\bar{\Theta}}^{-1} Q_{\Theta})_{\cdot,i}\right\|^2+1}}$.
    By Lemma \ref{lm:formula_QinvQ}, $Q_{\bar{\Theta}}^{-1} Q_{\Theta} = -(B^{-1}N)^\top$, and thus $\left\|(Q_{\bar{\Theta}}^{-1} Q_{\Theta})_{\cdot,i}\right\| = \|(B^{-1}N)_{i,\cdot}\|$. Combined with $\zeta_d=\left\|P_{\linVd}(q)\right\|\mu_d$, this completes the proof.\Halmos
\end{proof}

With Lemma \ref{lm:compute_p_d_mu}, we can now prove Theorem \ref{thm:compexity_using_zeta}.

\begin{proof}{Proof of Theorem \ref{thm:compexity_using_zeta}.}
	Directly substituting the expressions of $\zeta_p$ and $\zeta_d$ into the expression \eqref{def_geometric_measure}  of $\simplephi$ completes the proof.\Halmos
\end{proof}

\section{Experimental Confirmation}\label{sec:experiments}

This section evaluates the empirical relevance of the accessible iteration bounds developed above.
Section~\ref{subsec:perturbation_exp} uses two families of LP instances to illustrate the reciprocal relationship between the magnitude of data perturbations and the condition measure $\simplephi$, as predicted by Section~\ref{sec:LP_sharpness}. 
Section~\ref{subsec:experiments_condition_numbers} then studies the two-stage behavior of rPDHG on randomly generated LPs and selected MIPLIB 2017 LP relaxations. 
The experiments compare the observed Stage-I iteration counts with $\kappa\simplephi\ln(\kappa\simplephi)$ and the observed Stage-II iteration counts with $\|B^{-1}\|\|A\|$, as suggested by Theorem~\ref{thm:local_linear_convergence}. 
We also compare against the no-restart PDHG method analyzed by \citet{lu2025geometry} and the corresponding quantities appearing in their two-stage bounds.
The code used to generate the experiments in Section~\ref{subsec:experiments_condition_numbers} is publicly available.\footnote{The code is available at \url{https://github.com/ZikaiXiong/PDLP-complexity-bound}.}

In all rPDHG (Algorithm \ref{alg: PDHG with restarts}) experiments, we use the fixed step sizes from Theorem~\ref{thm:closed-form-complexity},
\begin{equation}\label{eq:step_sizes_exp}
    \tau=\frac{1}{2\kappa},
    \qquad
    \sigma=\frac{1}{2\lambda_{\max}\lambda_{\min}},
    \qquad
    \beta=1/e .
\end{equation}
Here $\lambda_{\max}$ and $\lambda_{\min}$ denote the largest and smallest nonzero singular values of the constraint matrix $A$. 
These step sizes are not tuned for computational performance. They are chosen so that the empirical iteration counts can be compared directly with the theoretical settings of Theorems~\ref{thm:closed-form-complexity} and~\ref{thm:local_linear_convergence}. 
For a run with a different primal--dual step-size ratio, the corresponding condition measures should be computed on the associated reweighted instance, as explained in Remark~\ref{rmk:stepsize_reweighting}.

The main implementation difference from the theoretical presentation is the computation of the normalized duality gap. 
Instead of the original $M$-norm version, we use the separable $\widetilde M$-norm defined by $\|(x,y)\|_{\widetilde M}
:=
\sqrt{\frac{1}{\tau}\|x\|^2+\frac{1}{\sigma}\|y\|^2}$, where $\widetilde M :=
\left(\begin{smallmatrix}
        \frac{1}{\tau}I_n & \\
        & \frac{1}{\sigma}I_m
\end{smallmatrix}\right)$.
This alternative is equivalent to the original normalized duality gap up to constant factors; see \citet{applegate2023faster,xiong2024role}. 
It is substantially easier to compute and is also used in practical PDHG implementations, such as \citet{applegate2021practical,lu2025cupdlp}. For computational efficiency, the restart condition and the stopping criterion are evaluated every 10 \textsc{OnePDHG} iterations.

\subsection{Effects of Data Perturbations}\label{subsec:perturbation_exp}

To evaluate rPDHG's sensitivity to perturbations in the objective vector $c$, we construct a family of standard-form LP instances \eqref{pro:primal LP} with data $\left(A^1, b^1, c^1\right)$, where 
$$
A^1 = \left[1 , 1, 1\right],  \ b^1=2 , \  \text{ and } c^1 =c_\gamma^1:= [2,-1,-1] +   \left[0,-\frac{\gamma}{2} ,\frac{\gamma}{2}\right]
$$
for parameter $\gamma \ge 0$.   This LP family, denoted by LP$_\gamma^1$, is designed to illustrate the effect of perturbations $[0,-\frac{\gamma}{2},\frac{\gamma}{2}]$ on the objective vector $c^1_0$. 
When $\gamma = 0$, LP$_\gamma^1$ has multiple optimal solutions along the line segment connecting $(0,2,0)$ and $(0,0,2)$. For $\gamma > 0$, the problem has a unique optimal solution at $(0,2,0)$. Since $\zeta_p$ is equivalent to the proximity to multiple optima, $\zeta_p$ is at most $O(\gamma)$ when $\gamma > 0$. As for $\kappa$, it is always equal to $1$ for LP$_\gamma^1$ instances.

For the family of problems LP$_\gamma^1$, the values of $\simplephi$ for different $\gamma$ values are as follows:
\vspace{5pt}
\begin{center}
	\begin{tabular}{c|ccccc}
		$\gamma$    & 1e0 & 1e-1  & 1e-2 & 1e-3 & 1e-4  \\ \hline
		$\simplephi$ of  LP$_\gamma^1$ & 6.4e0 & 4.5e1 & 4.3e2 & 4.2e3 & 4.2e4
		\end{tabular} 
\end{center}
\vspace{5pt}
These values of $\simplephi$  exhibit an approximately reciprocal dependence on the value of $\gamma$.
This observation aligns with \eqref{eq:thm:compexity_using_zeta_1} of Theorem \ref{thm:compexity_using_zeta}.
Figure \ref{fig_perturb:sub1} shows the convergence performance of rPDHG on LP$_\gamma^1$ instances for   $\gamma \in\{ 0,0.02,0.005,0.001\}$. The horizontal axis reports the number of iterations, while the vertical axis reports the relative error, defined as: $\mathcal{E}_r(x, y):=\frac{\left\|A x^{+}-b\right\|}{1+\|b\|}+\frac{\left\|(A^{\top} y - c)^{+}\right\|}{1+\|c\|}+\frac{\left|c^{\top} x^{+}-b^{\top} y\right|}{1+\left|c^{\top} x^{+}\right|+\left|b^{\top} y\right|}$ for iterates $(x,y)$. 
We use $\mathcal{E}_r(x, y)$ because it is easy to compute, and applicable when the problem has multiple optima. It is also a widely used standard tolerance (also used in \citet{lu2023cupdlp-c,applegate2021practical}). 
The results show that as $\gamma \searrow 0$, the number of iterations (of Stage I in particular) increases significantly, exhibiting an approximately reciprocal relationship with $\gamma$.  
At the same time, the local convergence behavior in Stage~II is essentially unchanged across these perturbations.  This is consistent with Theorem \ref{thm:local_linear_convergence}, which asserts that the local convergence rate depends on $\|B^{-1}\|\|A\|$ rather than on $\simplephi$.

\begin{figure}[htbp]
	\FIGURE
	{
		\hspace*{\fill} 
	\subcaptionbox{\small LP$^1_\gamma$\label{fig_perturb:sub1}}
		{\includegraphics[width=0.4\textwidth]{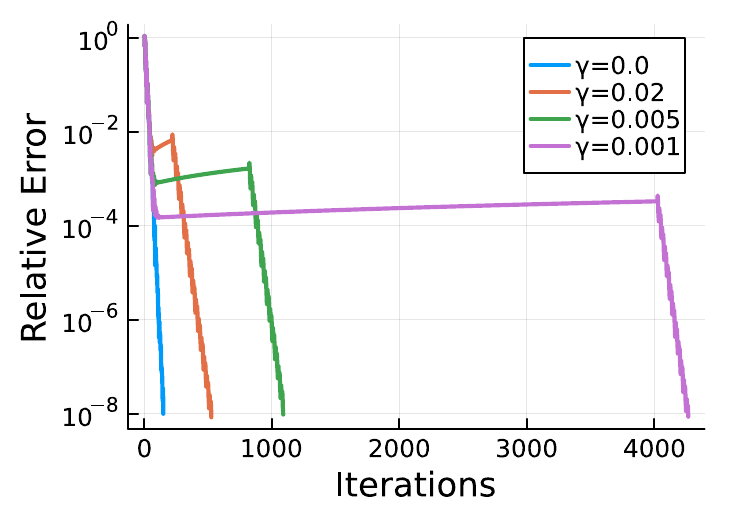}}
	\hfill
	\subcaptionbox{\small LP$^2_\gamma$\label{fig_perturb:sub2}}
		{\includegraphics[width=0.4\textwidth]{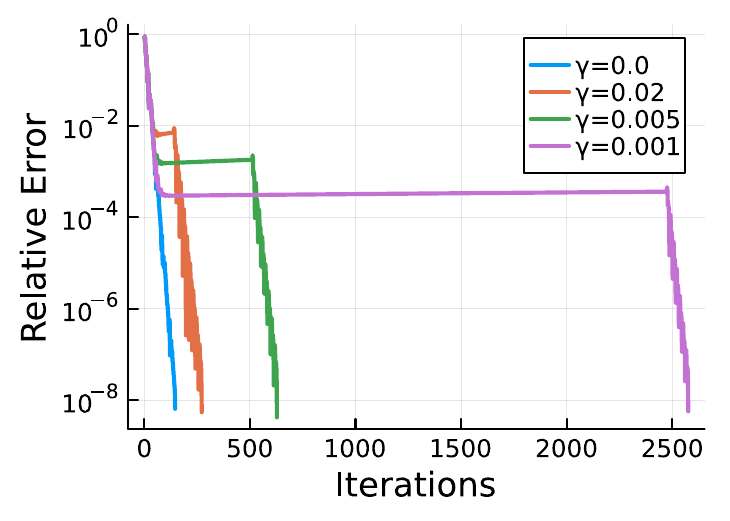}}
	\hspace*{\fill} 
	}
	{\small Convergence performance of rPDHG on two families of LP instances.\vspace{5pt}\label{fig:perturbation_plot}}
	{}
\end{figure}

We further construct another family of standard-form LP instances \eqref{pro:primal LP} with data $\left(A^2, b^2, c^2\right)$, where 
$$
A^2 = \left[\begin{smallmatrix}
    1 & 1 & -1 \\ 1 & 0 & 1
\end{smallmatrix}\right], \ c^2=[-0.5,1,0.5] , \ \text{ and }b^2 = b_\gamma^2:= [1,1] +  [\gamma,2\gamma]
$$ for  parameter $\gamma \ge 0$.  This LP family, denoted by LP$_\gamma^2$, is designed to illustrate the effect of perturbations $[\gamma,2\gamma]$ on the right-hand side vector $b^2_0$. The value of $\kappa$ is always equal to $1.22$ for all LP$_\gamma^2$ instances.
When $\gamma = 0$, LP$_\gamma^2$ has a unique optimal primal solution $[1,0,0]$. This solution is degenerate, implying multiple dual optimal solutions for LP$^2_0$. When $\gamma > 0$, the problem has a unique dual optimal solution. Since $\zeta_d$ is also equivalent to the proximity to multiple dual optimal solutions,  $\zeta_d$ is at most $O(\gamma)$ when $\gamma > 0$.

For the family of problems LP$_\gamma^2$, the values of $\simplephi$ for different $\gamma$ values are as follows:
\vspace{5pt}
\begin{center}
	\begin{tabular}{c|ccccc}
		$\gamma$     & 1e0 & 1e-1  & 1e-2 & 1e-3 & 1e-4 \\ \hline
		$\simplephi$ of  LP$_\gamma^2$ & 6.7e0 & 3.4e1 & 3.4e2 & 3.4e3 & 3.4e4
		\end{tabular}
\end{center}
\vspace{5pt}
These values of $\simplephi$ again exhibit a clear reciprocal relationship with $\gamma$. 
Figure \ref{fig_perturb:sub2} reports the convergence performance of rPDHG for LP$_\gamma^2$ instances for $\gamma \in\{ 0,0.02,0.005,0.001\}$. Although the perturbation is now applied to the right-hand side vector $b$, the observed behavior is nearly symmetric to that in LP$^1_\gamma$: smaller perturbations lead to a much longer Stage~I, while the Stage-II local convergence behavior remains largely unchanged. 
These observations are consistent with Theorems~\ref{thm:compexity_using_zeta} and~\ref{thm:local_linear_convergence}.

\subsection{Two-stage Performance of rPDHG}\label{subsec:experiments_condition_numbers}

This subsection tests how well the quantities $\kappa\simplephi$ and $\|B^{-1}\|\|A\|$ explain the practical two-stage performance of rPDHG. 
We consider two classes of instances: randomly generated LPs from Todd's model and selected LP relaxations from the MIPLIB 2017 library (see \citet{gleixner2021miplib2017}). 
For each class, we report five panels, grouped into four types of comparisons:
(i) rPDHG Stage-I iterations versus $\kappa\simplephi\ln(\kappa\simplephi)$; 
(ii) rPDHG Stage-II iterations versus $\|B^{-1}\|\|A\|$; 
(iii) no-restart PDHG Stage-I iterations versus the full Stage-I bound and $(R/\delta)^2$  from \citet{lu2025geometry}; 
(iv) no-restart PDHG Stage-II iterations versus the  non-logarithmic factor  of the Stage-II bound from \citet{lu2025geometry}.

The random LP instances are generated according to Todd's random LP model:
\begin{definition}[Random linear program]\label{def:random_lp}
	Let $u\sim\calN(0,1)$.
	Let the entries of the matrix $A$ be independent and identically distributed (i.i.d.) copies of $u$.
	The primal and  dual solutions $\hat{x}$ and $\hat{s}$ are generated as follows:
	\begin{equation}\label{eq:given_optimal_solution}
		\hat{x}_{\Theta} \in \R^m_+, \quad \hat{s}_{\bar{\Theta}} \in \R^{n-m}_+, \quad \hat{x}_{\bar{\Theta}} = 0, \quad \hat{s}_{\Theta} = 0,
	\end{equation}
	in which $\Theta = \{1,2,\dots,m\}$ and the components of $\hat{x}_{\Theta}$ and $\hat{s}_{\bar{\Theta}}$ are i.i.d. copies of $|u|$. The right-hand side $b$ is generated by $b = A\hat{x}$, and the cost vector $c$ is generated by $c = \hat{s}$. (Optionally, the cost vector $\bar{c}$ with the smallest norm is generated by $\bar{c} :=  \hat{s} + A^\top \hat{y}$, where $\hat{y} = \arg\min_{y\in\R^m}\|\hat{s} + A^\top \hat{y}\|$.)
\end{definition}

The random LP in Definition \ref{def:random_lp} is Model 1 of \citet{todd1991probabilistic}. Variants of this model have been analyzed by \citet{anstreicher1993average,ye1994toward,anstreicher1999probabilistic} to elucidate the average performance of interior-point methods. 
One can observe that $\hat{x}$ and $\hat{s}$ are the optimal primal-dual solution because they are feasible and satisfy the complementary slackness condition $\hat{x}^\top\hat{s} = 0$. Components of $\hat{x}_{\Theta}$ and $\hat{s}_{\bar{\Theta}}$ are all nonzero almost surely, and the LP instance has a unique optimum with optimal basis $\Theta$ almost surely. 
Since the optimal basis and the optimal solution are known prior to solving the problem,  $\kappa\simplephi$ and $\|B^{-1}\|\|A\|$ used in the iteration bound results can be easily computed.
Furthermore, replacing the cost vector $c$ with the smallest-norm cost vector $\bar{c}$ does not influence the optimality and degeneracy of $\hat{x}$ and $\hat{s}$, but $\bar{c}$ lies in $\operatorname{Null}(A)$, i.e., $A\bar{c} = 0$. To keep consistent with the theoretical results, we use $\bar{c}$ in the experiments.

We run rPDHG on 100 randomly generated LP instances from Definition~\ref{def:random_lp}, with $m=50$ and $n=100$. 
We denote the $i$-th instance by LP$_i$, for $i=1,\ldots,100$. 
Let $(x^{\ell,k},y^{\ell,k})$ denote the raw iterates, let $\bar x^{\ell,k}$ denote the averaged primal iterate in the current inner loop, and define $s^{\ell,k}:=c-A^\top y^{\ell,k}$. 
An instance is declared solved once
$
\|(x^{\ell,k},s^{\ell,k})-(x^\star,s^\star)\|_2 \le 10^{-4},
$
where $(x^\star,s^\star)$ is the optimal primal solution and dual slack vector. 
After the run terminates, we define the number of Stage-I iterations as the first recorded iteration index after which the support of the averaged primal iterate $\bar x^{\ell,k}$ matches the support of $x^\star$ and remains unchanged thereafter. 
The remaining iterations are counted as Stage~II. 
The active-set check is performed every 10 \textsc{OnePDHG} iterations, and the overall iteration limit is $10^7$.

Theorem~\ref{thm:closed-form-complexity} gives an overall iteration bound of  $O\big(\kappa\simplephi\cdot
\ln(\kappa\simplephi \cdot \frac{\|w^\star\|}{\eps})\big)$. 
The refined two-stage bound in Theorem~\ref{thm:local_linear_convergence} further predicts that Stage~I is controlled by
$O\big(\kappa\simplephi\cdot
\ln(\kappa\simplephi )\big)$
whereas Stage~II is controlled by
$O\big(\|B^{-1}\|\|A\|\cdot\ln(\frac{\xi}{\eps})\big)$.
Since the target tolerance is fixed in our experiments, we focus on the leading condition-measure components
$\kappa\simplephi\ln(\kappa\simplephi)$ and $\|B^{-1}\|\|A\|$.

Figure~\ref{fig:sub1} plots the empirical Stage-I iteration counts against $\kappa\simplephi\ln(\kappa\simplephi)$. 
The two quantities show a clear positive association. A linear regression fitted on the log scale used in the plot gives $R^2=0.6847$, indicating that this quantity explains a substantial fraction of the variation in $\log_{10}$ Stage-I iteration counts. 
Since $\simplephi$ and $\geophi$ are equivalent up to a constant factor by Lemma~\ref{lm:compute_p_d_condition_numbers}, this also provides empirical support for the role of the level-set geometry condition measure $\geophi$ studied in \citet{xiong2024role}.
Figure~\ref{fig:sub2} performs the analogous comparison for Stage~II. 
The empirical Stage-II iteration counts are strongly associated with $\|B^{-1}\|\|A\|$, with a log-scale regression value $R^2=0.6857$. 
This supports the prediction of Theorem~\ref{thm:local_linear_convergence} that the local convergence phase is governed by the optimal-basis quantity $\|B^{-1}\|\|A\|$ rather than by $\simplephi$. 
One of the 100 instances is excluded from Figures~\ref{fig:sub1} and~\ref{fig:sub2} because the normalized duality gap becomes too small to evaluate reliably due to numerical roundoff.

\begin{figure}[t]
    \centering
    \begin{subfigure}{0.32\textwidth}
        \includegraphics[width=\textwidth]{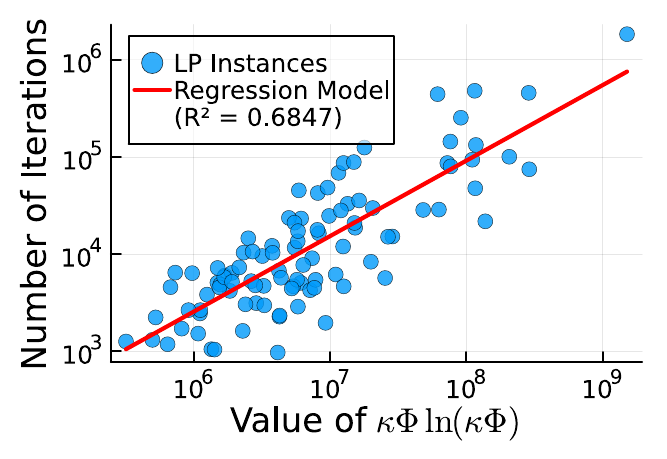}
        \caption{\small Our Stage I bound of rPDHG}
        \label{fig:sub1}
    \end{subfigure}
    \begin{subfigure}{0.32\textwidth}
        \includegraphics[width=\textwidth]{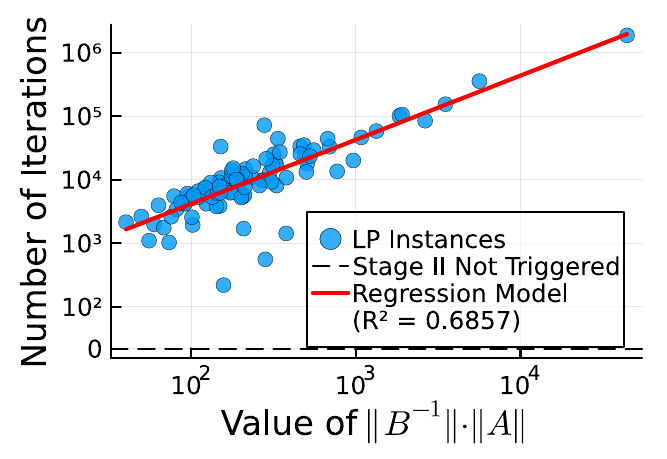}
        \caption{\small Our Stage II bound of rPDHG}
        \label{fig:sub2}
    \end{subfigure}
	\hfill
    \\[1em]
    \begin{subfigure}{0.32\textwidth}
        \includegraphics[width=\textwidth]{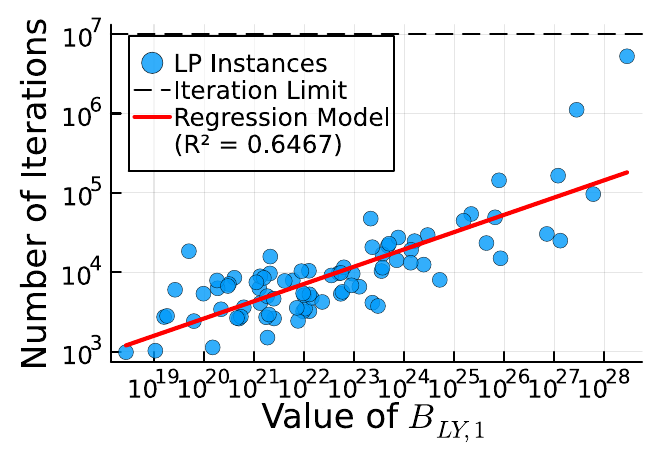}
        \caption{\small Stage I bound of \citet{lu2025geometry}}
        \label{fig:sub3}
    \end{subfigure} 
    \begin{subfigure}{0.32\textwidth}
        \includegraphics[width=\textwidth]{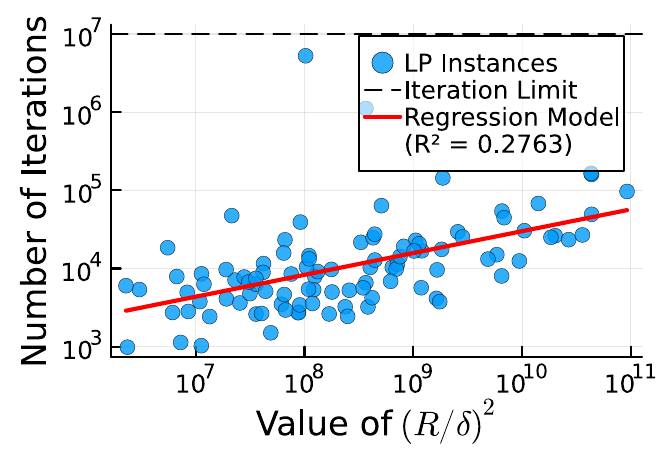}
        \caption{\small $(R/\delta)^2$ of \citet{lu2025geometry}}
        \label{fig:sub4}
    \end{subfigure} 
	    \begin{subfigure}{0.32\textwidth}
        \includegraphics[width=\textwidth]{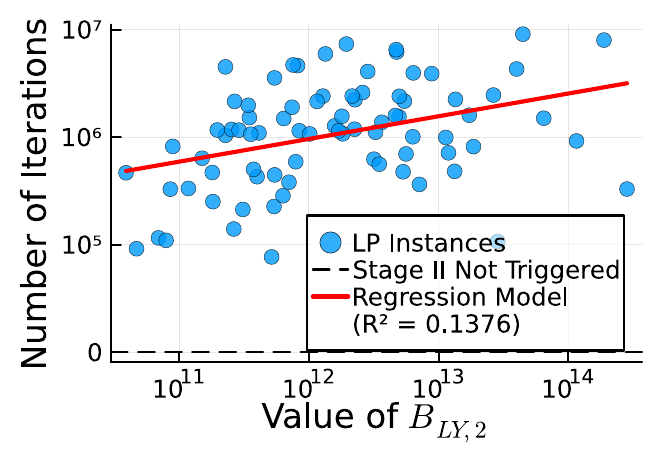}
        \caption{\small Stage II bound of \citet{lu2025geometry}}
        \label{fig:sub5}
    \end{subfigure} \caption{\small Scatter plots of \textsc{OnePDHG} iteration counts in Stage~I (optimal basis identification) and Stage~II (local convergence) on random LP instances. 
	Panels (a)--(b) report rPDHG and compare the observed counts with $\kappa\simplephi\ln(\kappa\simplephi)$ and $\|B^{-1}\|\|A\|$. 
	Panels (c)--(e) report no-restart PDHG and compare with $B_{LY,1}$, $(R/\delta)^2$, and $B_{LY,2}$. 
	Panels (c) and (e) include only instances for which the corresponding Hoffman bound computation succeeds. Panels (b) and (e) include only instances solved within the iteration limit.
\label{fig:scatter_plot}}
\end{figure}

We next compare with the two-stage bounds of \citet{lu2025geometry}. 
They analyze no-restart PDHG with equal primal and dual step sizes $\tau=\sigma=s$, where $s\le 1/(2\|A\|)$, and prove that no-restart PDHG exhibits a two-stage behavior: Stage~I identifies the active set of the limiting solution, and Stage~II enjoys faster local linear convergence. 
In the unique-optimizer setting considered here, active-set identification is equivalent to optimal-basis identification.
Specialized to our standard-form LP setting, the Stage-I bound of \citet{lu2025geometry} can be written as $O(B_{LY,1})$, where
\begin{equation}\label{eq:luyang_phaseI_bound_exp}
B_{LY,1}
:=
\max\left\{1,\frac{1}{s^2\alpha_{L_1}^2}\right\}
\left(\frac{R}{\delta}\right)^2
+
\frac{1}{s\|A\|}.
\end{equation}
Here 
$R=2\|z^{0,0}-z^\star\|+2\|z^\star\|+1$, and $\delta =
\min\left\{
\min_{i:s_i^\star>0}\frac{s_i^\star}{\|A\|},
\min_{j:x_j^\star>0}x_j^\star
\right\}$, 
and $\alpha_{L_1}$ is the sharpness parameter of a homogeneous linear inequality system determined by the optimal basis. 
Their Stage-II bound is $O\left(B_{LY,2} \ln\left(\frac{\delta}{\eps}\right)\right)$ for computing an $\eps$-optimal solution, where
\begin{equation}\label{eq:luyang_phaseII_bound_exp}
B_{LY,2}
:=
\frac{1}{s^2\alpha_{L_2}^2},
\end{equation} and $\alpha_{L_2}$ is the sharpness parameter of another homogeneous linear inequality system determined by the optimal basis.

In our comparison, we implement the no-restart PDHG method analyzed by \citet{lu2025geometry}. 
To make the algorithmic comparison use the same effective primal--dual step-size ratio as the rPDHG runs in \eqref{eq:step_sizes_exp}, we compute the Lu--Yang quantities on the corresponding reweighted instance from Remark~\ref{rmk:stepsize_reweighting}. 
Under this reweighting, running equal-step PDHG with step size $s$ on the reweighted instance is equivalent to running PDHG with the step sizes in \eqref{eq:step_sizes_exp} on the original instance.
The quantities $R$ and $\delta$ are directly computable once the optimal solution is known. 
The sharpness parameters $\alpha_{L_1}$ and $\alpha_{L_2}$ are reciprocals of Hoffman constants of the corresponding homogeneous linear inequality systems. 
Following the approach suggested by \citet{lu2025geometry}, we use the method of \citet{pena2024easily} to compute upper bounds on these Hoffman constants, and use their reciprocals as proxies for $\alpha_{L_1}$ and $\alpha_{L_2}$ in the reported bounds. 
This procedure involves solving some LP problems and log-barrier subproblems, for which we use the commercial conic solver MOSEK. 
These subproblems can be ill-conditioned, especially on MIPLIB instances. 
When MOSEK does not return a reliable solution, we omit the corresponding point from the plots involving $B_{LY,1}$ or $B_{LY,2}$. 
Because $(R/\delta)^2$ is easier to compute and does not require solving these log-barrier subproblems, we also report it separately as an alternative predictor of Stage-I iteration counts.

Figures~\ref{fig:sub3} and~\ref{fig:sub4} compare the no-restart PDHG Stage-I iteration counts with $B_{LY,1}$ and $(R/\delta)^2$, respectively. 
Both quantities show a visible positive association with the observed iteration counts, suggesting that they capture some aspects of instance difficulty. 
However, the fitted slopes in the log-scale plots are substantially smaller than one, so the observed iteration counts grow more slowly than these quantities. 
This contrasts with Figure~\ref{fig:sub1}, where $\kappa\simplephi\ln(\kappa\simplephi)$ gives a more direct scaling for the rPDHG Stage-I counts.
Figure~\ref{fig:sub5} compares $B_{LY,2}$ with the observed Stage-II iteration counts of the no-restart PDHG, restricted to instances solved before the iteration limit and for which the Hoffman bound computation succeeds. 
The association is weaker than in the Stage-I comparisons. 
Only 84 instances are shown in Figure~\ref{fig:sub3} because the computation of $\alpha_{L_1}$ fails on the remaining instances. 
Although all random LP instances enter Stage~II, some runs reach the overall iteration limit $10^7$ before satisfying the stopping criterion. Together with failures in computing $\alpha_{L_2}$, this leaves only 80 instances in Figure~\ref{fig:sub5}.

\begin{figure}[t]
    \centering
    \begin{subfigure}{0.32\textwidth}
        \includegraphics[width=\textwidth]{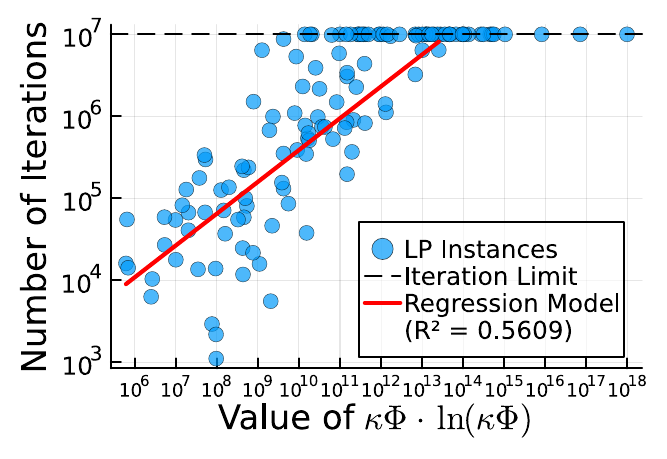}
        \caption{\small Our Stage I bound of rPDHG}
        \label{fig:sub1-miplib}
    \end{subfigure}
    \begin{subfigure}{0.32\textwidth}
        \includegraphics[width=\textwidth]{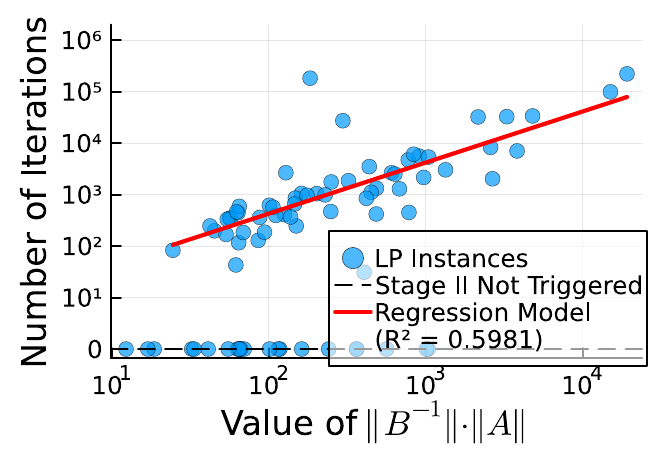}
        \caption{\small Our Stage II bound of rPDHG}
        \label{fig:sub2-miplib}
    \end{subfigure}
	\hfill
    \\[1em]
    \begin{subfigure}{0.32\textwidth}
        \includegraphics[width=\textwidth]{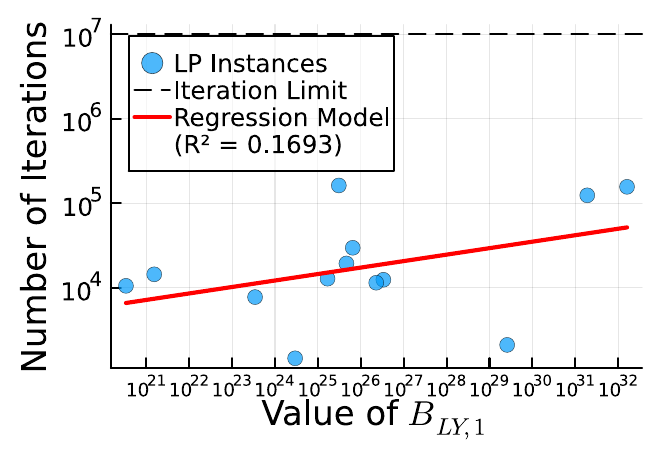}
        \caption{\small Stage I bound of \citet{lu2025geometry}}
        \label{fig:sub3-miplib}
    \end{subfigure} 
    \begin{subfigure}{0.32\textwidth}
        \includegraphics[width=\textwidth]{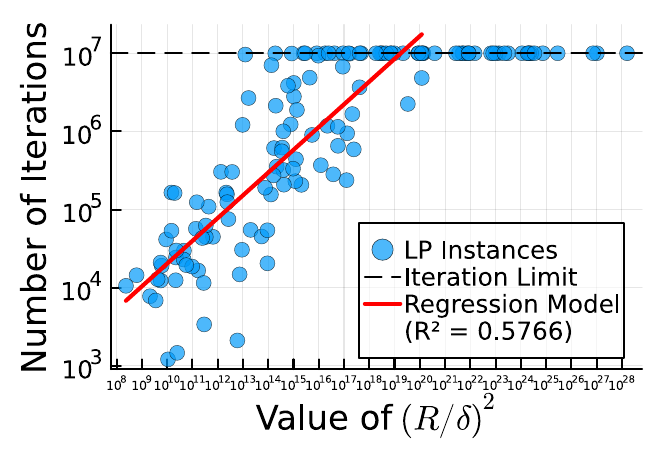}
        \caption{\small $(R/\delta)^2$ of \citet{lu2025geometry}}
        \label{fig:sub4-miplib}
    \end{subfigure} 
	    \begin{subfigure}{0.32\textwidth}
        \includegraphics[width=\textwidth]{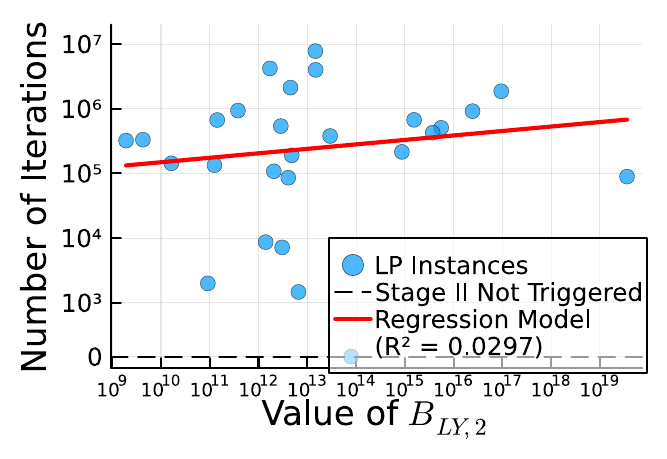}
        \caption{\small Stage II bound of \citet{lu2025geometry}}
        \label{fig:sub5-miplib}
    \end{subfigure} 
    \caption{\small
Scatter plots of \textsc{OnePDHG} iteration counts in Stage~I (optimal basis identification) and Stage~II (local convergence) on selected MIPLIB 2017 LP relaxation instances. 
	Panels (a)--(b) report rPDHG and compare the observed counts with $\kappa\simplephi\ln(\kappa\simplephi)$ and $\|B^{-1}\|\|A\|$. 
	Panels (c)--(e) report no-restart PDHG and compare with $B_{LY,1}$, $(R/\delta)^2$, and $B_{LY,2}$. 
	Panels (c) and (e) include only instances for which the corresponding Hoffman bound computation succeeds. Panels (b) and (e) include only instances solved within the iteration limit.
\label{fig:scatter_plot-miplib}}
\end{figure}

We next test the same predictions on selected LP relaxations from the MIPLIB 2017 library (see \citet{gleixner2021miplib2017}). 
We first apply the presolver PaPILO (see \citet{gleixner2023papilo}), which removes empty rows and columns, eliminates fixed variables, detects inconsistent bounds, and performs other reductions. 
We then convert each presolved instance to the standard form \eqref{pro:primal LP}. 
Since many instances do not satisfy Assumption~\ref{assump:unique_optima}, we perturb the data to obtain unique optimal solutions almost surely. 
Specifically, we replace
$c$ by
$c+0.01 \|c\|_2 \frac{v}{\|v\|_2},$ and replace 
$b$ by $b+0.01 \|b\|_2 \frac{A\bar v}{\|A\bar v\|_2}$,
where $v\in\mathbb R^n$ has i.i.d. entries distributed as $|u|$ with $u\sim\mathcal N(0,1)$, and $\bar v$ is an independent copy of $v$. 
After these perturbations, we apply Ruiz rescaling and Pock--Chambolle rescaling, as suggested by \citet{applegate2021practical}. 
The perturbations are small relative to the data scale but large enough to break degeneracy. 
We use the same algorithmic settings as in the random-LP experiments, with an overall iteration limit of $10^7$. 
To keep the computational effort manageable, we select instances satisfying $mn\le 5{,}000{,}000$, which leaves 132 instances in total.

Figure~\ref{fig:scatter_plot-miplib} reports the results on the selected MIPLIB instances. 
Some runs do not meet the stopping criterion within $10^7$ iterations, and some do not enter Stage~II before the iteration limit. 
Such runs are indicated by the dashed iteration-limit line in Figures~\ref{fig:sub1-miplib}, \ref{fig:sub3-miplib}, and~\ref{fig:sub4-miplib}. 
Figures~\ref{fig:sub2-miplib} and~\ref{fig:sub5-miplib} report Stage-II counts only for instances solved before the iteration limit.
As in the random LP experiments, Stage-I iteration counts are positively associated with $\kappa\simplephi\ln(\kappa\simplephi)$ (Figure~\ref{fig:sub1-miplib}), and Stage-II iteration counts are positively associated with $\|B^{-1}\|\|A\|$ (Figure~\ref{fig:sub2-miplib}). 
The regression models are fitted only on uncensored instances, and the corresponding $R^2$ values are $0.5609$ and $0.5981$, respectively.
In particular, although the slope of the fitted regression model of Figure~\ref{fig:sub1-miplib} seems smaller than $1$, the upper envelope of the points in Figure~\ref{fig:sub1-miplib} appears to grow approximately linearly with $\kappa\simplephi\ln(\kappa\simplephi)$. 
The two apparent outliers far above the regression model in Figure~\ref{fig:sub2-miplib} are due to the fact that our empirical Stage-II split is based only on support identification, whereas the Stage-II guarantee in Lemma \ref{lm:local_condL} requires a stronger local condition on the averaged iterates; in these instances, the reported Stage-II count includes a long final inner loop before the averaged iterates enter the optimal basis.

Computing $\alpha_{L_1}$ and $\alpha_{L_2}$ for the MIPLIB instances is substantially more challenging than for the random LPs, due to ill-conditioning in the log-barrier subproblems used by the method of \citet{pena2024easily}. 
Consequently, Figures~\ref{fig:sub3-miplib} and~\ref{fig:sub5-miplib} include only 13 and 26 instances, respectively. 
This illustrates that, although \citet{pena2024easily} provides useful characterizations and a polynomial-time method for bounding Hoffman constants of homogeneous linear inequality systems, numerically evaluating such constants can still be difficult on real-world LP instances. 
Figure~\ref{fig:sub4-miplib} reports the more easily computable quantity $(R/\delta)^2$. 
This quantity shows a strong association with the observed Stage-I iteration counts, but the actual counts still grow more slowly than $(R/\delta)^2$ suggests. 
This observation is consistent with Figure 5 of \citet{lu2025geometry}, which compares the actual iteration count with the value of $R/\delta$ for 50 MIPLIB instances.

\section{Summary, Limitations, and Future Directions}
\label{sec:summary_future}

In this paper, we establish an accessible iteration bound for rPDHG on LPs with unique primal and dual optimal solutions. 
Concretely, we prove that, rPDHG computes an $\eps$-optimal solution in at most $O\left( \kappa\simplephi \cdot 
	\ln\left(\kappa\simplephi\cdot \frac{\|(x^\star,s^\star)\|}{\eps}\right)\right)$ \textsc{OnePDHG} iterations, where $\kappa$ is the matrix condition number and $\simplephi$ is a geometric condition measure determined by the optimal solution and the optimal basis.  We further refine this result into a two-stage bound. In Stage~I, rPDHG identifies the optimal basis within $O\left(\kappa\simplephi\cdot\ln(\kappa\simplephi)\right)$ iterations. In Stage~II, after the optimal basis has been identified, rPDHG computes an $\eps$-optimal solution within $O\left(\|B^{-1}\|\|A\|\cdot 
	\max\left\{0,\ln\left(\frac{\min_{1\le i\le n}\{x_i^\star+s_i^\star\}}{\eps}\right)\right\}\right)$ additional iterations.  In particular, the Stage~II bound no longer depends on $\simplephi$, which provides an explanation for the faster local behavior often observed in practice.

	These iteration bounds are accessible because the quantities appearing in them have closed-form expressions in terms of the optimal solution and the optimal basis.   
	Our experiments suggest that these iteration bounds capture important features of the global and local performance of rPDHG. 
	Furthermore, the closed-form expression of $\simplephi$ also reveals its connections with several other condition measures, including stability under data perturbations, proximity to multiple optima, and LP sharpness.

	The accessible bounds developed in this paper have also led to follow-up directions. 
	For example, \citet{xiong2025high} applies these bounds to random LP instances and proves high-probability polynomial iteration guarantees for rPDHG under random input models. 
	In a different direction, the formulation of $\simplephi$ suggests possible practical heuristics for choosing the primal-dual step-size ratio, or equivalently for primal-dual reweighting; see \citet{xiong2025new}. 

	The main limitation of this paper is the unique optimum assumption. 
	Although this assumption enables a clean closed-form analysis through the optimal basis, many LP instances arising in real-world applications may have multiple optimal solutions. A natural future direction is therefore to develop similarly accessible iteration bounds without assuming unique primal and dual optima. More broadly, it would be useful to understand how to efficiently compute iteration bounds for general LPs.

	Another important future direction is to develop accessible iteration bounds for convex quadratic programs. 
	Recent work has extended PDHG-type methods to convex quadratic programs; see, for example, \citet{lu2025practicalqp,huang2025restarted}. 
	However, it remains unclear what the right analogue of the LP condition measure $\simplephi$ should be, or whether similarly computable iteration bounds can be obtained for quadratic programs.



\begin{APPENDICES}
    \SingleSpacedXI
    
    \section{Proofs of Section \ref{sec:global_linear_convergence}}\label{sec:proofs_global_linear_convergence}
    
    \subsection{Technical lemmas about the normalized duality gap and $\beta$-restart condition}
    
    Note that Section \ref{subsec:adaptive_restart_condition} has reviewed the basic information of the $\beta$-restart condition and the normalized duality gap, which were omitted in Section \ref{sec:global_linear_convergence} but will be heavily used throughout the proofs. Here we formally present the sublinear convergence result of the normalized duality gap summarized by \citet{applegate2023faster}, using an equivalent result presented by \citet{xiong2023computational}: 
    \begin{lemma}[Lemma 2.2 of \citet{xiong2023computational}]\label{lm: original sublinear PDHG}
        Suppose that $\sigma, \tau$ satisfy \eqref{eq:general_stepsize}. Then for any $z^0 := (x^0,y^0)$ with $x^0 \in \R^n_+$, the following inequality holds for all $k \ge 1$:
        \begin{equation}
            \rho(\|\bar{z}^k-z^0\|_M;\bar{z}^k) \le \frac{4\dist_M(z^0,\calZ^\star)}{k} \ .
        \end{equation}
    \end{lemma} 
    Then an upper bound on the number of iterations for inner loops in rPDHG is as follows. 
	The following lemma is the corresponding restart-count consequence.
    \begin{lemma}\label{lm:when_to_restart}
        Suppose that \eqref{condition L} holds for $z^{\ell,0}$ and $z^{\ell-1,0}$ with $\ell\ge 1$ and $\condL > 0$. Whenever $k \ge   \frac{4\condL}{\beta}$, sufficient decrease has been made on the normalized duality gap, i.e., the restart condition \eqref{catsdogs} is satisfied.
    \end{lemma} 
    \begin{proof}{Proof.}
        For $k \ge 1$, Lemma \ref{lm: original sublinear PDHG} implies:
        \begin{equation}\label{eq thm: complexity of PDHG with adaptive restart 1}
            \rho(\|\bar{z}^{\ell,k}-z^{\ell,0}\|_M;\bar{z}^{\ell,k}) \le \frac{4\dist_M(z^{\ell,0} ,\calZ^\star)}{k}\ .
        \end{equation}
    
        If $\rho(\|z^{\ell,0}-z^{\ell-1,0}\|_M; z^{\ell,0}) = 0$, then $z^{\ell,0}$ is a saddle point of \eqref{pro:saddlepoint_LP} and thus $z^{\ell,0}\in\calZ^\star$. In this case, $z^{\ell,k} = z^{\ell,0}$ for all $k\ge 1$, and any $k\ge 1$ satisfies \eqref{catsdogs}.
    
        If $\rho(\|z^{\ell,0}-z^{\ell-1,0}\|_M; z^{\ell,0}) \neq 0$, dividing both sides of \eqref{eq thm: complexity of PDHG with adaptive restart 1} by $\rho(\|z^{\ell,0}-z^{\ell-1,0}\|_M; z^{\ell,0})$ yields:
        \begin{equation}\label{eq thm: complexity of PDHG with adaptive restart 2}
            \frac{\rho(\|\bar{z}^{\ell,k}-z^{\ell,0}\|_M;\bar{z}^{\ell,k})}{\rho(\|z^{\ell,0}-z^{\ell-1,0}\|_M; z^{\ell,0})} \le \frac{4}{k} \cdot \frac{\dist_M(z^{\ell,0} ,\calZ^\star)}{\rho(\|z^{\ell,0}-z^{\ell-1,0}\|_M; z^{\ell,0})} \le \frac{4}{k}\cdot \condL
        \end{equation}
        where the last inequality follows from \eqref{condition L}. Therefore, when $k\ge\frac{4\condL}{\beta}$, the right-hand side is no larger than $\beta$, i.e., $\frac{4}{k} \cdot \condL \le \beta$, and the restart condition \eqref{catsdogs} is satisfied.\Halmos
    \end{proof}

    \subsection{Proof of Proposition \ref{prop_complexity_of_w_xi}}
    \begin{proof}{Proof of Proposition \ref{prop_complexity_of_w_xi}.}
        For the second term of multiplication in the right-hand side of \eqref{def_geometric_measure}, we have   
        \begin{equation}\label{eq_complexity_of_w_xi}
            \begin{aligned}
                       & \max\left\{
                \max_{1\le j\le n-m} \frac{\sqrt{\left\|(B^{-1} N)_{\cdot,j}\right\|^2+1} }{s^\star_{m+j}},  \max_{1\le i \le m} \frac{\sqrt{\left\|(B^{-1} N)_{i,\cdot}\right\|^2+1} }{x^\star_{i}}
                \right\}                                                                                        \\
                \le \  & \frac{\max\left\{
                    \max_{1\le j\le n-m} \sqrt{\left\|(B^{-1} N)_{\cdot,j}\right\|^2+1},   \max_{1\le i \le m} \sqrt{\left\|(B^{-1} N)_{i,\cdot}\right\|^2+1}
                \right\} }{\min\left\{\min_{1\le i \le m}x^\star_{i},\min_{1\le j\le n-m}s^\star_{m+j}\right\}} \\
                = \    & \frac{\max\left\{
                \sqrt{\|B^{-1} N\|_{1,2}^2+1},  \sqrt{\|B^{-1} N\|_{2,\infty}^2+1}
                \right\} }{\min_{1\le i \le n}\left\{x_i^\star + s_i^\star\right\}}
            \end{aligned}
        \end{equation}
        where the equality holds because $x^\star$ and $s^\star$ are strictly complementary. Furthermore, because $\|\cdot\|_{1,2}$ and $\|\cdot\|_{2,\infty}$ norms are upper bounded by the spectral norm (denoted by $\|\cdot\|$), we have the following  inequalities:
        \begin{equation*}
            \begin{aligned}
                 & \max\left\{
                \sqrt{\|B^{-1} N\|_{1,2}^2+1},  \sqrt{\|B^{-1} N\|_{2,\infty}^2+1}
                \right\} \le \sqrt{\|B^{-1} N\|^2+1} = \sqrt{\sigma_{\max}^+(B^{-1} N N^\top  B^{-\top}) + 1  }                                 \\
                 & = \sqrt{\sigma_{\max}^+(B^{-1}(N N^\top + BB^\top) B^{-\top})} = \sqrt{\sigma_{\max}^+(B^{-1}AA^\top B^{-\top})} = \|B^{-1} A\| \ .
            \end{aligned}
        \end{equation*} 
		Finally, by complementary slackness,   $\|x^\star\|_1+\|s^\star\|_1 = \|x^\star+s^\star\|_1$.  
        Applying these inequalities to the definition of $\simplephi$ in \eqref{def_geometric_measure} completes the proof.\Halmos
    \end{proof}

    \subsection{Proof of Lemma \ref{lm:global_linear}}\label{subapp:proof_of_thm_global_linear}
    
    In this subsection, we prove Lemma \ref{lm:global_linear}. 
    We begin with the following lemma.

    \begin{lemma}\label{thm L C}
        Suppose that $Ac = 0$. Algorithm \ref{alg: PDHG with restarts} (rPDHG) is run starting from $z^{0,0} = (x^{0,0},y^{0,0} ) = (0,0)$, and the step-sizes $\sigma$ and $\tau$ satisfy \eqref{eq:general_stepsize}. Then for all $\ell\ge 1$, it holds that
        \begin{equation}\label{eq thm L C}
            \dist_M(z^{\ell,0},\calZ^\star) \le
            \sqrt{2}c_{\tau,\sigma}\cdot\dist(w^{\ell,0},\calW^\star)          \le (3\sqrt{2} + 4)  c_{\tau,\sigma}^2\cdot \geophi \cdot \rho(\| z^{\ell,0} - z^{\ell-1,0}\|_M; z^{\ell,0}) \ .
        \end{equation}
        In other words, condition \eqref{condition L} holds with $\condL = (3\sqrt{2} + 4)  c_{\tau,\sigma}^2 \geophi$.
    \end{lemma}
    \noindent
    This lemma is Lemma 3.13 of \citet{xiong2024role} by taking limits as $\delta$ approaches $0$ on both sides.

    \begin{lemma}[Proposition 3.7 of \citet{xiong2024role}]\label{lm:distance_to_wstar}
        Suppose that $Ac = 0$ and $z^{0,0} = (0,0)$. Then $\|w^{0,0} - w^\star\| = \|(0,c) - w^\star\|\le \|w^\star\|$.
    \end{lemma}
    \begin{proof}{Proof of Lemma \ref{lm:global_linear}.}
       By Assumption \ref{assump:unique_optima}, let $\calZ^\star$ be $ \{z^\star\}$, then $\dist_M(z^{\ell,0},\calZ^\star) = \|z^{\ell,0}-z^\star\|_M$, and Lemma \ref{thm L C} states:
        \begin{equation}\label{eq thm_global_linear 1}
            \begin{aligned}
                \|z^{\ell,0}-z^\star\|_M \le \sqrt{2}c_{\tau,\sigma}\cdot \|w^{\ell,0}-w^\star\| & \le (3\sqrt{2}+4) c_{\tau,\sigma}^2 \cdot \geophi  \cdot \rho(\|z^{\ell,0}-z^{\ell-1,0}\|_M;z^{\ell,0})  .
            \end{aligned}
        \end{equation} 
        Substituting the step-sizes, we have $c_{\tau,\sigma} = \sqrt{2\kappa}$ and $(3\sqrt{2}+4)c_{\tau,\sigma}^2 \approx 16.4853\kappa \le 16.5\kappa$, so \eqref{eq thm_global_linear 1} gives \eqref{condition L} with $\condL=16.5\kappa\geophi$. Moreover,
        \begin{equation}\label{eq thm_global_linear 2}
            \begin{aligned}
                \|w^{\ell,0}-w^\star\| & \le (3\sqrt{2} + 4) \sqrt{\kappa} \geophi \cdot \rho(\|z^{\ell,0}-z^{\ell-1,0}\|_M;z^{\ell,0}) \le 8.25 \sqrt{\kappa} \geophi \cdot \rho(\|z^{\ell,0}-z^{\ell-1,0}\|_M;z^{\ell,0}) \ .
            \end{aligned}
        \end{equation}
    
        Let $\bar{T}_\eps$ denote the total number of \textsc{OnePDHG} iterations required to obtain the first outer iteration $N$ that satisfies $\rho(\|{z}^{N,0} - z^{N-1,0}\|_M; {z}^{N,0})\le \eps$. The following restart-count argument is standard for the normalized duality gap; variants appear, for example, in \citet{applegate2023faster,xiong2023computational}, and we include it here for completeness. Let $\rho_\ell:=\rho(\|z^{\ell,0}-z^{\ell-1,0}\|_M;z^{\ell,0})$ for $\ell\ge 1$ and $D_0:=\|z^{0,0}-z^\star\|_M$. By Lemma \ref{lm: original sublinear PDHG}, $\rho_1\le 4D_0$, and by the restart condition, $\rho_\ell\le \beta^{\ell-1}\rho_1\le 4\beta^{\ell-1}D_0$. Hence the first outer loop index $N$ with $\rho_N\le\eps$ satisfies $N-1\le \ln(4D_0/\eps)/\ln(1/\beta)+1$. Moreover, Lemma \ref{lm:when_to_restart} and \eqref{eq thm_global_linear 1} imply that every inner loop after the first one uses at most $\nu:=\left\lceil\frac{4\cdot 16.5\kappa\geophi}{\beta}\right\rceil$
        \textsc{OnePDHG} iterations. The first outer loop uses one \textsc{OnePDHG} iteration. Therefore, since $\beta=1/e$, 
        $\bar{T}_{\eps} \le 1+(N-1)\nu   \le 1+\nu\left(\ln\left(\frac{4D_0}{\eps}\right)/\ln(1/\beta)+1\right)  \le \nu\cdot \ln\left(\frac{12D_0}{\eps}\right). $ Here the last inequality uses $\nu\ge 11$ and $11(\ln 3-1)>1$. Also, since $16.5\kappa\geophi\ge1$, we have $\nu\le 66e\kappa\geophi+1\le198\kappa\geophi$. Consequently, we obtain
        \begin{equation}
            \bar{T}_{\eps} \le 198\kappa\geophi  \cdot \ln\left(\frac{12 \|z^{0,0}-z^\star\|_M}{\eps}\right) \ .
        \end{equation}
        Note that \eqref{eq thm_global_linear 2} ensures that when the normalized duality gap is sufficiently small, the distance to the optimal solution is correspondingly small. Therefore,
        \begin{equation}\label{eq thm_global_linear 3}
            T \le \bar{T}_{\frac{\eps}{8.25 \sqrt{\kappa} \geophi}}   \le 198\kappa  {\geophi}  \cdot \ln\left(\frac{99 \sqrt{\kappa} \geophi \cdot \|z^{0,0}-z^\star\|_M}{\eps}\right) \le  198\kappa  {\geophi}  \cdot \ln\left(\frac{198 \kappa \geophi \cdot \|w^{0,0}-w^\star\|}{\eps}\right)
        \end{equation} 
        where the last inequality follows from \eqref{eqlm change of norm} and $c_{\tau,\sigma} = \sqrt{2\kappa}$. Finally, by Lemma \ref{lm:distance_to_wstar}, we have $\|w^{0,0} - w^\star\|\le \|w^\star\|$. Consequently, \eqref{eq thm_global_linear 3} leads to \eqref{eq_thm_gloabal_convergence_T} of Lemma \ref{lm:global_linear}.\Halmos
    \end{proof} 
    
    Furthermore, similar results also exist for the goal of obtaining an iterate $z^{N,0}$ satisfying $\|z^{N,0}-z^\star\|_M \le \eps$, which will be useful later in Appendix \ref{app:proofs_local_linear_convergence}.
    \begin{remark}\label{remark:complexity_of_PDHG}
        Under the same conditions as Lemma \ref{lm:global_linear}, let $\tilde{T}$ denote the total number of \textsc{OnePDHG} iterations required to obtain the first outer iteration $N$ that satisfies both $\rho(\|z^{N-1,0}-z^{N,0}\|_M;z^{N,0})\le\frac{\eps}{16.5\kappa \geophi}$ and $\|{z}^{N,0} - z^\star\|_M \le \eps$. Then, $\tilde{T} \le   198\kappa  \geophi  \cdot \ln\left(\frac{ 396 \kappa^{1.5} \geophi \cdot \|w^\star\|}{\eps}\right).$  
    \end{remark}
    \noindent
    The proof of Remark \ref{remark:complexity_of_PDHG} is almost identical to that of Lemma \ref{lm:global_linear}, except \eqref{eq thm_global_linear 3} is replaced by the inequality derived from $\tilde{T} \le \bar{T}_{\frac{\eps}{16.5\kappa \geophi}}$ (due to \eqref{eq thm_global_linear 1}).

    \section{Proofs of Section \ref{sec:local_linear_convergence}}\label{app:proofs_local_linear_convergence}
    
    \subsection{Proofs of Lemmas \ref{lm:local_condL} and \ref{lm:distance_until_local_linear}} 
First of all, we defined the $\tilde{M}$-norm
\begin{equation}\label{eq:define_Mtilde-norm}
	\|(x, y)\|_{\tilde{M}}:=\sqrt{\frac{1}{\tau}\|x\|^2+\frac{1}{\sigma}\|y\|^2}\ \ \ \text{ where } \ \ \ \tilde{M}:=\left(\begin{smallmatrix}
		\frac{1}{\tau} I_n &                      \\
						   & \frac{1}{\sigma} I_m
	\end{smallmatrix}\right) .
\end{equation} 
When $\tau$ and $\sigma$ are sufficiently small, the $M$-norm and $\tilde{M}$-norm are equivalent up to well-specified constants related to $\tau$ and $\sigma$ (Proposition 2.6 of \citet{xiong2023computational}, with its $N$-norm written here as the $\tilde{M}$-norm). 
For any point $z :=(x,y) \in \mathbb{R}^{n+m}$, and $w := (x,c - A^\top y)$,  it holds that
	\begin{equation}\label{eqlm change of norm}
		\sqrt{1-\sqrt{\tau \sigma} \lambda_{\max }} \cdot \|z\|_{\tilde{M}}  \le
		\|z\|_M    \le \sqrt{2}\cdot \|z\|_{\tilde{M}} \le
		\sqrt{2}c_{\tau,\sigma} \cdot \|w\| \ .
	\end{equation}  
For example, if $\tau = \frac{1}{2\kappa}$ and $\sigma = \frac{1}{2\lambda_{\max}\lambda_{\min}}$, then \eqref{eqlm change of norm} becomes $\frac{\sqrt{2}}{2}  \|z\|_{\tilde{M}}  \le
	\|z\|_M  \le \sqrt{2} \|z\|_{\tilde{M}} \le
	2\sqrt{\kappa} \|w\|$. This result will be extensively used later.

\begin{proof}{Proof of Lemma \ref{lm:local_condL}.}
	Let $\bar{s}=c-A^{\top} \bar{y}$ and $\bar{w}=(\bar{x}, \bar{s})$. From the definition of $\rho(r ; \cdot)$ we have:
	\begin{equation}\label{eq:local_condL_1}
		L(\bar{x}, y)-L(x, \bar{y}) \leq r \rho(r ; \bar{z}) \text { for any } z=(x, y) \in B(r ; \bar{z})
	\end{equation}
	where recall that $B(r;\bar{z})$ is defined in Definition \ref{def:normalized_duality_gap}.

	Firstly, we prove that
	\begin{equation}\label{eq:local_condL_2}
		\|\bar{x} - x^\star\| \le \frac{\|B^{-1}\|}{\sqrt{\sigma}}\cdot \rho(r;\bar{z}) \ .
	\end{equation}
	As the optimal basis is unique, $x^\star$ is represented by its basic and nonbasic parts: $x^\star_{[m]} = B^{-1}b$ and $x^\star_{[n]\setminus[m]} = 0$. Consequently, due to \eqref{eq:tigher_bound_region},
    \begin{equation}\label{eq:local_condL_3}
        \left\|\bar{x}_{[m]} - x^\star_{[m]}\right\| = \left\|\bar{x}_{[m]} - B^{-1}b\right\| \le \|B^{-1}\|\cdot \left\|B \bar{x}_{[m]} - b\right\| \text{ and } \left\|\bar{x}_{[n]\setminus [m]} - x^\star_{[n]\setminus[m]}\right\| = 0 \ .
    \end{equation}
	Let $u=b- B\bar{x}_{[m]}$ and define $y:=\bar{y}+\sqrt{\sigma} r \cdot u /\|u\|$. Let $z:=(\bar{x}, y)$, and then $z \in B(r ; \bar{z})$. Thus, from \eqref{eq:local_condL_1} we obtain
	\begin{equation}\label{eq:local_condL_4}
		r \rho(r ; \bar{z}) \geq L(\bar{x}, y)-L(\bar{x}, \bar{y})=(b-A \bar{x})^{\top}(y-\bar{y})=(b-B \bar{x}_{[m]})^{\top}(y-\bar{y})=\sqrt{\sigma} r\|u\|,
	\end{equation}
    implying $\|u\|=\|b-B \bar{x}_{[m]}\| \leq \frac{\rho(r ; \bar{z})}{\sqrt{\sigma}}$. Substituting this result back into \eqref{eq:local_condL_3} yields \eqref{eq:local_condL_2}.

	Secondly, we prove that
	\begin{equation}\label{eq:local_condL_5}
		\|\bar{y} - y^\star\| \le  \frac{\|B^{-1}\|}{\sqrt{\tau}} \cdot \rho(r;\bar{z}) \ .
	\end{equation} Given that the optimal basis is $[m]$, we have $B^\top y^\star = c_{[m]}$, and $y^\star = (B^\top)^{-1}c_{[m]}$. Consequently,
	\begin{equation}\label{eq:local_condL_6}
		\left\|\bar{y} - y^\star\right\| = \left\|\bar{y} - (B^\top)^{-1}c_{[m]}\right\| \le \left\|(B^\top)^{-1}\right\|\cdot \left\|B^\top \bar{y} - c_{[m]}\right\| = \left\|B^{-1}\right\|\cdot \left\|B^\top \bar{y} - c_{[m]}\right\|\ .
	\end{equation}
	Let $v = c_{[m]} - B^\top \bar{y}$ and define $x$ as follows: $x_{[m]} := \bar{x}_{[m]} - \sqrt{\tau}r\cdot \tfrac{v}{\|v\|}$ and $x_{[n]\setminus[m]} := 0$. 	Note that due to the condition $\bar{x}_{i} \ge r\sqrt{\tau}$ for all $i\in[m]$ in \eqref{eq:tigher_bound_region}, $x$ remains in $\R^n_+$. Now, let $z := (x,\bar{y})$, and then $z\in B(r;\bar{z})$. Thus, from \eqref{eq:local_condL_1}, we derive:
	\begin{equation}\label{eq:local_condL_7}
		r \rho(r ; \bar{z}) \geq L(\bar{x}, \bar{y})-L(x, \bar{y})=(c-A^{\top} \bar{y})^{\top}(\bar{x}-x)= (c_{[m]} - B^\top \bar{y})^{\top}(\bar{x}_{[m]}-x_{[m]}) = \sqrt{\tau}r \|v\| \ ,
	\end{equation}
	implying $\|v\| = \|B^\top \bar{y} - c_{[m]}\| \le \frac{\rho(r;\bar{z})}{\sqrt{\tau}}$. Substituting this result back into \eqref{eq:local_condL_6} yields \eqref{eq:local_condL_5}.


	Finally, combining \eqref{eq:local_condL_2} and \eqref{eq:local_condL_5}, we obtain
	\begin{small}
\begin{equation}
\|\bar z-z^\star\|_{\tilde M}^2 =
\frac{1}{\tau}\|\bar x-x^\star\|^2 + \frac{1}{\sigma}\|\bar y-y^\star\|^2 \le
\frac{1}{\tau}\cdot \frac{\|B^{-1}\|^2}{\sigma}\rho(r;\bar z)^2 + \frac{1}{\sigma}\cdot \frac{\|B^{-1}\|^2}{\tau}\rho(r;\bar z)^2 = \frac{2\|B^{-1}\|^2}{\tau\sigma}\rho(r;\bar z)^2 .  
\end{equation}\end{small}
Using \eqref{eqlm change of norm}, namely
\(\|\bar z-z^\star\|_M\le \sqrt{2}\|\bar z-z^\star\|_{\tilde M}\), proves
$
\|\bar z-z^\star\|_M
\le
\frac{2\|B^{-1}\|}{\sqrt{\tau\sigma}}\rho(r;\bar z).
$ \Halmos
\end{proof}

Before proving Lemma \ref{lm:distance_until_local_linear}, we recall the nonexpansive property of rPDHG proven by \citet{applegate2023faster}. We use the more convenient format presented by \citet{xiong2024role}.
\begin{lemma}[Lemma 2.2 of \citet{xiong2024role}]\label{lm:nonexpansive}
	For any $\ell,k\ge 0$, it holds that $\|z^{\ell,k}-z^\star\|_M\le \|z^{\ell,0}-z^\star\|_M$ and $\|\bar{z}^{\ell,k} - z^\star \|_M\le  \left\|z^{\ell,0} - z^\star\right\|_M$. For any  $\ell_1,\ell_2$ such that $\ell_2\ge \ell_1\ge 0$, it holds that $\|z^{\ell_2,0}-z^\star\|_M \le \|z^{\ell_1,0}-z^\star\|_M$.
\end{lemma}

\begin{proof}{Proof of Lemma \ref{lm:distance_until_local_linear}.}
	The proof contains two steps. In the first step, we prove that when
	\begin{equation}\label{eq:distance_until_local_linear_0}
		\left\|z^{t,0} - z^\star\right\|_M \le   \frac{\sqrt{1 - \sqrt{\tau\sigma}\|A\|}}{3\sqrt{\tau}}   \cdot \xi \ ,
	\end{equation}
	then for any $N \ge t$ and $k \ge 1$, item (i) of \eqref{eq:lm:distance_until_local_linear_1} holds. 

	We begin by establishing a lower bound for $\min_{1\le i \le m}  \bar{x}^{N,k}_{i}$, the left-hand side of item (i): 
	\begin{equation}\label{eq:distance_until_local_linear_1_0}
		\min_{1\le i \le m} \ \bar{x}^{N,k}_{i} \ge \left(\min_{1\le i \le m} \ x_{i}^\star \right) - \left\|\bar{x}^{N,k}_{[m]} - x^\star_{[m]}\right\| \ge \xi - \left\|\bar{x}^{N,k}_{[m]}  - x^\star_{[m]}\right\| \ .
	\end{equation}
	The second term on the right-hand side of \eqref{eq:distance_until_local_linear_1_0}  can be bounded as follows:
	\begin{equation}\label{eq:distance_until_local_linear_1}
		\left\|\bar{x}^{N,k}_{[m]} - x^\star_{[m]}\right\| \le \left\|\bar{x}^{N,k} - x^\star\right\| 
		\overset{\eqref{eq:define_Mtilde-norm}}{\le} \sqrt{\tau}\left\|\bar{z}^{N,k} - z^\star\right\|_{\tilde{M}} \overset{\eqref{eqlm change of norm}}{\le} \frac{\sqrt{\tau}\left\|\bar{z}^{N,k} - z^\star\right\|_M}{\sqrt{1 - \sqrt{\tau\sigma}\|A\|}} \overset{\text{Lemma \ref{lm:nonexpansive}}}{\le} \frac{\sqrt{\tau}\left\|z^{t,0} - z^\star\right\|_M}{\sqrt{1 - \sqrt{\tau\sigma}\|A\|}}
	\end{equation}
	for any $t \le N$. In addition, due to the  nonexpansive property, the right-hand side of item (i) is upper bounded by:
	\begin{equation}\label{eq:distance_until_local_linear_2}
		\left\|\bar{z}^{N,k} - z^{N,0}\right\|_M \le \left\|\bar{z}^{N,k} - z^{\star}\right\|_M + \left\|z^{\star} - z^{N,0}\right\|_M \le 2 \left\|z^{\star} - z^{t,0}\right\|_M \le \frac{2\left\|z^{t,0} - z^\star\right\|_M}{\sqrt{1 - \sqrt{\tau\sigma}\|A\|} } 
	\end{equation}
	where the last inequality holds because $\sqrt{1 - \sqrt{\tau\sigma}\|A\|} \le 1$.
    Finally, we have 
	\begin{equation}\label{eq:distance_until_local_linear_3}
		\begin{aligned}
		& \left(\min_{1\le i \le m} \ \bar{x}^{N,k}_{i}\right) - \sqrt{\tau}\left\|\bar{z}^{N,k} - z^{N,0}\right\|_M
		\overset{\eqref{eq:distance_until_local_linear_1_0},\eqref{eq:distance_until_local_linear_1}}{\ge} \xi - \frac{\sqrt{\tau}\left\|z^{t,0} - z^\star\right\|_M}{\sqrt{1 - \sqrt{\tau\sigma}\|A\|} } - \sqrt{\tau}\left\|\bar{z}^{N,k} - z^{N,0}\right\|_M
		\\ &\overset{  \eqref{eq:distance_until_local_linear_2}}{\ge} \xi - \frac{\sqrt{\tau}\left\|z^{t,0} - z^\star\right\|_M}{\sqrt{1 - \sqrt{\tau\sigma}\|A\|} } - 2 \sqrt{\tau} \left\|z^{\star} - z^{t,0}\right\|_M  \ge \xi - \frac{3\sqrt{\tau}\left\|z^{t,0} - z^\star\right\|_M}{\sqrt{1 - \sqrt{\tau\sigma}\|A\|} } \ ,
		\end{aligned}
	\end{equation}
	in which the last term is nonnegative when \eqref{eq:distance_until_local_linear_0} holds.
	This shows that when $\left\|z^{t,0} - z^\star\right\|_M$ is small enough and satisfies \eqref{eq:distance_until_local_linear_0}, the left-hand side of \eqref{eq:distance_until_local_linear_3} is nonnegative, and item (i) of \eqref{eq:lm:distance_until_local_linear_1} holds. This completes the first step of the proof.

	In the second step, we prove that when
	\begin{equation}\label{eq:distance_until_local_linear_4}
		\left\|z^{t,0} - z^\star\right\|_M \le   \frac{\sqrt{\tau}\sqrt{1 - \sqrt{\tau\sigma}\|A\|}}{2}   \cdot \xi,
	\end{equation}
	then for any $N \ge t$ and $k \ge 1$ item (ii) of \eqref{eq:lm:distance_until_local_linear_1} holds. 

	Let $\alpha$ denote $\left(\min_{1\le j \le n-m,k\ge0} s^{N,k}_{m+j}\right)$, then according to \eqref{eq_alg: one PDHG} (in Line \ref{line:onepdhg} of Algorithm \ref{alg: PDHG with restarts}),
	\begin{equation*}
		x^{N,k}_{[n]\setminus[m]} = \left(x^{N,k-1}_{[n]\setminus[m]} - \tau s^{N,k-1}_{[n]\setminus[m]}\right)^+ \le \left(x^{N,k-1}_{[n]\setminus[m]} - \tau \alpha\right)^+ 
	\end{equation*}
	and applying this inequality recursively as $k$ decreases to $0$ yields
	\begin{equation}\label{eq:distance_until_local_linear_4_0}
		x^{N,k}_{[n]\setminus[m]} \le \left(x^{N,k-1}_{[n]\setminus[m]} - \tau \alpha\right)^+ \le \left(x^{N,k-2}_{[n]\setminus[m]} - 2\tau \alpha\right)^+  \le \dots \le \left(x^{N,0}_{[n]\setminus[m]} - k\tau \alpha\right)^+  \ .
	\end{equation}
	Therefore, if $\alpha \ge 0$ and
	\begin{equation}\label{eq:distance_until_local_linear_4_new}
	\left(\max_{1\le j\le n-m}  x^{N,0}_{m+j}\right) \le  \tau \cdot \left(\min_{1\le j \le n-m,k\ge0} s^{N,k}_{m+j} \right) \ ,
	\end{equation}
	then $x^{N,k}_{[n]\setminus[m]} = 0$ for all $k\ge 1$ and thus $\bar{x}^{N,k}_{[n]\setminus[m]} = \frac{1}{k}\sum_{i=1}^k x_{[n]\setminus[m]}^{N,i}= 0$ for all $k\ge 1$. Intuitively, since Assumption \ref{assump:unique_optima} holds, and $x^{N,0}$ and $s^{N,k}$ converge to the optimal solution, $x^{N,0}_{[n]\setminus[m]}$ should converge to $0$ and $s^{N,k}_{[n]\setminus[m]}$ should stay away from $0$. In the rest of the proof we will establish a lower bound of $\left(\min_{1\le j \le n-m,k\ge0}s^{N,k}_{m+j}\right)$ and an upper bound of $\left(\max_{1\le j\le n-m}  x^{N,0}_{m+j}\right)$.

	Similar to the first step, for the dual iteration $y^{N,k}$ (and $s^{N,k} = c - A^\top y^{N,k}$), we have 
	\begin{equation}\label{eq:distance_until_local_linear_5}
		\begin{aligned}
		\left(\min_{1\le j\le n-m}s^{N,k}_{m+j}\right) & \ge \left(\min_{1\le j\le n-m}  s_{m+j}^\star \right) - \left\|s^{N,k}_{[n]\setminus[m]} - s^\star_{[n]\setminus[m]}\right\| 
		\\
		&  = \xi - \left\|A_{[n]\setminus[m]}^\top (y^{N,k} - y^\star)\right\|  \ge \xi - \left\|A_{[n]\setminus[m]} \right\| \cdot \left\|y^{N,k}  - y^\star\right\| \ .
		\end{aligned}
	\end{equation}
	As for $\left\|y^{N,k} - y^\star \right\|$ in the last term of the above inequality, we derive:
	\begin{equation}\label{eq:distance_until_local_linear_6}
		\left\|y^{N,k} - y^\star \right\|  \overset{\eqref{eq:define_Mtilde-norm}}{\le} \sqrt{\sigma}\left\|z^{N,k} - z^\star\right\|_{\tilde{M}} \overset{\eqref{eqlm change of norm}}{\le} \frac{\sqrt{\sigma}\left\|z^{N,k} - z^\star\right\|_M}{\sqrt{1 - \sqrt{\tau\sigma}\|A\|}} \overset{\text{Lemma \ref{lm:nonexpansive}}}{\le} \frac{\sqrt{\sigma}\left\|z^{t,0} - z^\star\right\|_M}{\sqrt{1 - \sqrt{\tau\sigma}\|A\|}} \ .
	\end{equation} 
	Combining \eqref{eq:distance_until_local_linear_5} and \eqref{eq:distance_until_local_linear_6} yields  a valid lower bound  of $\left(\min_{1\le j\le n-m}s^{N,k}_{m+j}\right)$ for all $k\ge 0$: 
	\begin{equation}\label{eq:distance_until_local_linear_7}
		\left(\min_{1\le j\le n-m}s^{N,k}_{m+j}\right)\ge  \xi - \sqrt{\sigma}\|A_{[n]\setminus[m]}\| \cdot \frac{\left\|z^{t,0} - z^\star\right\|_M}{\sqrt{1 - \sqrt{\tau\sigma}\|A\|}} \ge  \xi - \frac{1}{\sqrt{\tau}} \cdot \frac{\left\|z^{t,0} - z^\star\right\|_M}{\sqrt{1 - \sqrt{\tau\sigma}\|A\|}}  \ ,
	\end{equation}
	where the second inequality holds because $\sqrt{\tau\sigma}\|A_{[n]\setminus[m]}\| \le \sqrt{\tau\sigma}\|A\| \le 1$ (due to the step-size requirement \eqref{eq:general_stepsize}).
	The above \eqref{eq:distance_until_local_linear_7} presents a lower bound of $\left(\min_{1\le j \le n-m,k\ge0}s^{N,k}_{m+j}\right)$.

	On the other hand, we also have the following upper bound of $\left(\max_{1\le j\le n-m}x^{N,0}_{m+j}\right)$ for $N \ge t$:
	\begin{equation}\label{eq:distance_until_local_linear_8}
		\left(\max_{1\le j\le n-m}x^{N,0}_{m+j} \right)\le \left\|x^{N,0}_{[n]\setminus[m]} \right\| = \left\|x^{N,0}_{[n]\setminus[m]} - x^\star_{[n]\setminus[m]} \right\| \le \|x^{N,0} - x^\star \|\overset{\eqref{eq:distance_until_local_linear_1}}{\le} \frac{\sqrt{\tau}\left\|z^{t,0} - z^\star\right\|_M}{\sqrt{1 - \sqrt{\tau\sigma}\|A\|}} \ .
	\end{equation} 
	Finally, applying the upper and lower bounds \eqref{eq:distance_until_local_linear_7} and \eqref{eq:distance_until_local_linear_8}, we find that \eqref{eq:distance_until_local_linear_4_new} holds when
	$$
		\frac{\sqrt{\tau}\left\|z^{t,0} - z^\star\right\|_M}{\sqrt{1 - \sqrt{\tau\sigma}\|A\|}}  \le \tau \xi - \sqrt{\tau} \cdot \frac{\left\|z^{t,0} - z^\star\right\|_M}{\sqrt{1 - \sqrt{\tau\sigma}\|A\|}}  \ ,
	$$
	which is equivalent to \eqref{eq:distance_until_local_linear_4}. This completes the second step of the proof.

	Finally, when both \eqref{eq:distance_until_local_linear_0} and \eqref{eq:distance_until_local_linear_4}  hold, which is satisfied by \eqref{eq:lm:distance_until_local_linear_0}, then for any $N \ge t$ and $k \ge 1$, both item (i) and (ii) of \eqref{eq:lm:distance_until_local_linear_1} hold. Furthermore, \eqref{eq:distance_until_local_linear_1} and \eqref{eq:distance_until_local_linear_1_0} ensure $\bar{x}^{N,k}_{\Theta}>0$, while the item (ii) ensures $\bar{x}^{N,k}_{\bar\Theta} = 0$. Therefore, the positive components of $\bar{x}^{N,k}$ correspond exactly to the optimal basis. This completes the proof.\Halmos
\end{proof}

	\subsection{Proof of Theorem \ref{thm:local_linear_convergence}}
    \begin{proof}{Proof of Theorem \ref{thm:local_linear_convergence}.}
        We first prove \eqref{eq:basis_identification_T}. According to Lemma \ref{lm:distance_until_local_linear}, once $N_0$ satisfies $\|z^{N_0,0}-z^\star\|_M\le \bar{\eps}$, which is equivalent to: 
        \begin{equation}\label{eq:local_linear_convergence_3}
            \|z^{N_0,0}-z^\star\|_M\le\bar{\eps}= \frac{\sqrt{2}}{6}\cdot \frac{1}{\sqrt{2\kappa}}\cdot \xi  = \frac{\xi}{6\sqrt{\kappa}} \ , 
        \end{equation}
        then for all $N > N_0$, we have:
        \begin{equation}\label{eq:local_linear_convergence_2}
            (i) \ \ x^{N,0}_{i} \ge  \sqrt{\tau} \|z^{N,0} - z^{N-1,0}\|_M \ \text{ for }i\in[m] \  , \ \ \text{ and }  \ \ (ii) \ x^{N,0}_{m+j} = 0 \text{ for }j\in[n-m] \ ,
        \end{equation}
        and the positive components of $x^{N,0}$ correspond exactly to the optimal basis. Therefore, $N_1 \le N_0 + 1$ and $T_1$ is bounded above by the number of \textsc{OnePDHG} iterations required to obtain $z^{N_0+1,0}$.
      
        According to Remark \ref{remark:complexity_of_PDHG}, the number of \textsc{OnePDHG} iterations needed to obtain such a $z^{N_0,0}$ is upper bounded by the number of iterations $\tilde{T}$ required to obtain $z^{\tilde{N}_0,0}$ such that $\rho(\|z^{\tilde{N}_0-1,0}-z^{\tilde{N}_0,0}\|_M;z^{\tilde{N}_0,0})\le \frac{\bar{\eps}}{16.5\kappa{\geophi}}$, and
        \begin{equation}\label{eq:local_linear_convergence_4}
            \tilde{T} \le 198\kappa  {\geophi}  \cdot \ln\left(\frac{ 2376 \kappa^2 {\geophi}  \| w^\star\|}{\xi}\right) \ .
        \end{equation}

        Furthermore, Lemma \ref{thm L C} implies that, with the step-sizes of Theorem \ref{thm:closed-form-complexity}, for all outer-loop indices $\ell$:
        \begin{equation}\label{eq:global_sharpness}
            \begin{aligned}
                \|z^{\ell,0}-z^\star\|_M
                 \le (3\sqrt{2}+4)c_{\tau,\sigma}^2 \geophi \cdot \rho(\|z^{\ell-1,0}-z^{\ell,0}\|_M;z^{\ell,0})  \le 16.5\kappa \geophi \cdot \rho(\|z^{\ell-1,0}-z^{\ell,0}\|_M;z^{\ell,0}) \ .
            \end{aligned}
        \end{equation}  
        Here we have used $c_{\tau,\sigma}=\sqrt{2\kappa}$ and $6\sqrt{2}+8\le 16.5$. Lemma \ref{lm:when_to_restart} guarantees that the number of additional \textsc{OnePDHG} iterations before obtaining the next outer loop iteration $z^{\tilde{N}_0+1,0}$ is at most
        $
            \left\lceil\frac{4(3\sqrt{2}+4)c_{\tau,\sigma}^2{\geophi}}{\beta}\right\rceil
            \le 
            \left\lceil\frac{4 \cdot 16.5\kappa {\geophi}}{\beta}\right\rceil .
        $
        Overall, since $\tilde{N}_0 \ge N_0$, the number of \textsc{OnePDHG} iterations required before obtaining $z^{\tilde{N}_0+1,0}$ is at most:
        $$
        \tilde{T} + \left\lceil\frac{4 \cdot 16.5\kappa {\geophi}}{\beta}\right\rceil  \overset{\eqref{eq:local_linear_convergence_4}}{\le} 
        198\kappa\geophi  \cdot \ln\left(\frac{ 2376 \kappa^2 {\geophi}  \| w^\star\|}{\xi}\right) + \left\lceil\frac{4 \cdot 16.5\kappa {\geophi}}{\beta}\right\rceil 
         \ ,
        $$
        which we use $T_{basis}$ to denote. 
        Because $\simplephi$ is equivalent to $\geophi$ as demonstrated by Lemma \ref{lm:compute_p_d_condition_numbers} and $ \frac{\|w^\star\|}{\xi} \le \simplephi $ from the definition, $T_{basis}$ reduces to  $O(\kappa\simplephi\ln(\kappa\simplephi))$ in \eqref{eq:basis_identification_T}.

        Next, we prove \eqref{eq:local_linear_convergence_T}. We first study how large $(N_2-\tilde{N}_0-1)$ could be. We mainly consider the case $N_2 > \tilde{N}_0 + 1$; otherwise $N_2 \le \tilde{N}_0 + 1$ and $T_2 \le T_{basis}$. In this case $N_2 > \tilde{N}_0 + 1$, because $N_2$ is the first iteration such that $\|w^{N_2,0}-w^\star\| \le \eps$, for the previous iteration we have $\|w^{N_2-1,0}-w^\star\| > \eps$. Dividing the second inequality in Lemma \ref{thm L C} by $\sqrt{2}c_{\tau,\sigma}$ and substituting $c_{\tau,\sigma}=\sqrt{2\kappa}$ yields \eqref{eq thm_global_linear 2}. Therefore,
        \begin{equation}\label{eq:local_linear_convergence_7}
            \rho\left(\|z^{N_2-1,0} - z^{N_2-2,0} \|_M;z^{N_2-1,0}\right) > \frac{\eps}{ (3\sqrt{2}+4)\sqrt{\kappa}\geophi }  > \frac{\eps}{ 8.25\sqrt{\kappa}\geophi } \ .
        \end{equation}
        On the other hand, the definition of $\tilde{N}_0$ uses the rounded denominator $16.5\kappa\geophi$ derived from $(3\sqrt{2}+4)c_{\tau,\sigma}^2\geophi$ in \eqref{eq:global_sharpness}. Hence,
        \begin{equation}\label{eq:local_linear_convergence_8}
            \rho\left(\|z^{\tilde{N}_0,0} - z^{\tilde{N}_0+1,0} \|_M;z^{\tilde{N}_0+1,0}\right) \le \beta\cdot\frac{\bar{\eps}}{16.5\kappa {\geophi}} = \frac{1}{e}\cdot\frac{1}{16.5 \kappa {\geophi}} \cdot \frac{\xi}{6\sqrt{\kappa}} \le  \frac{\xi}{99 e \cdot\kappa^{1.5} {\geophi}} \le  \frac{\xi}{269 \cdot\kappa^{1.5} {\geophi}} \ .
        \end{equation}
        Furthermore, due to the restart condition:
        \begin{equation}\label{eq:local_linear_convergence_9}
            \rho\left(\|z^{N_2-1,0} - z^{N_2-2,0} \|_M;z^{N_2-1,0}\right) \le \beta^{N_2-\tilde{N}_0-2} \cdot \rho(\|\z^{\tilde{N}_0,0} - z^{\tilde{N}_0+1,0} \|_M;z^{\tilde{N}_0+1,0}) \ .
        \end{equation}
        Substituting \eqref{eq:local_linear_convergence_7} and \eqref{eq:local_linear_convergence_8} into \eqref{eq:local_linear_convergence_9} yields an upper bound of $(N_2-\tilde{N}_0-1)$:
        \begin{equation}\label{eq:local_linear_convergence_10}
            \begin{aligned}
             \frac{\eps}{ 8.25\sqrt{\kappa}\geophi } < \left(\frac{1}{e}\right)^{N_2-\tilde{N}_0-2}\cdot \frac{\xi}{269 \cdot \kappa^{1.5} {\geophi}} 
             \Rightarrow   N_2 - \tilde{N}_0-1  \le \ln\left(\frac{ e\cdot 8.25}{269\kappa}\right)+\ln\left(\frac{\xi}{\eps}\right) \le \ln\left(\frac{\xi}{\eps}\right)  \ .
            \end{aligned}
        \end{equation}
        Here the final inequality is due to $\kappa \ge 1 $ and $\frac{ e\cdot 8.25}{269\kappa} \le 1$.
        Since we have assumed $N_2 > \tilde{N}_0 + 1$, the upper bound of $(N_2-\tilde{N}_0-1)$ in \eqref{eq:local_linear_convergence_10} has to be strictly positive. Once $\ln\left(\frac{\xi}{\eps}\right) \le 0$, then it is no longer in the case $N_2 > \tilde{N}_0 + 1$ and as previously stated before we already have $N_2 \le \tilde{N}_0 + 1$ and $T_2 \le T_{basis}$. 
        In conclusion, we can assert that $N_2 - \tilde{N}_0-1 \le \max\left\{0, \ln\left(\frac{\xi}{\eps}\right)\right\}$.

        We now turn our attention to the number of \textsc{OnePDHG} iterations  between $z^{N_0+1,0}$ and $z^{N_2,0}$. Remark \ref{rmk:local_condL} ensures that for all $N \ge \tilde{N}_0+1$:
        \begin{equation}\label{eq:local_linear_convergence_5}
            \|z^{N,0}- z^\star\|_M \le 4\|B^{-1}\|\|A\| \cdot \rho\left(\|z^{N,0} - z^{N-1,0} \|_M;z^{N,0}\right) \ .
        \end{equation}
        Then due to Lemma \ref{lm:when_to_restart}, the number of  iterations in inner loops is at most $\left\lceil\frac{4 \cdot 4\|B^{-1}\|\|A\|}{1/e}\right\rceil$.
        Therefore, we can bound the overall number of iterations $T_2$ as follows:  
        \begin{equation}\label{eq:local_linear_convergence_6}
            \begin{aligned}
                T_2 \le \  & T_{basis} + (N_2 - \tilde{N}_0-1) \cdot  \left\lceil\frac{4 \cdot 4\|B^{-1}\|\|A\|}{1/e}\right\rceil
            \end{aligned}
        \end{equation}
        Substituting the upper bound $N_2 - \tilde{N}_0-1  \le \max\left\{0, \ln\left(\frac{\xi}{\eps}\right)\right\}$ into \eqref{eq:local_linear_convergence_6} finishes the proof.\Halmos
    \end{proof}

    \end{APPENDICES}

    \ACKNOWLEDGMENT{The author is grateful to Robert M. Freund for the valuable suggestions while preparing this manuscript. The author also thanks Yinyu Ye, Haihao Lu, Tianyi Lin, Javier Pena and Xin Jiang for the helpful discussions. The author also thanks the anonymous referees and editors for helpful comments and suggestions that improved the quality of the manuscript. }



\bibliographystyle{informs2014} 
\bibliography{reference} 





  



\end{document}